\documentclass [12pt]{amsart}
\usepackage[utf8]{inputenc}

\usepackage{amsmath,amssymb,amsthm,amsfonts}
\usepackage{mathtools}%Mathematical tools to use with amsmath

\usepackage[english]{babel}%Multilingual support for Plain TeX or LaTeX
\usepackage{comment}%Selectively include/exclude portions of text
\usepackage{hyperref}%Extensive support for hypertext in LaTeX
\usepackage{bbm}%"Blackboard-style" cm fonts
\usepackage{mathrsfs}%Support for using RSFS fonts in maths (\mathcal fonts via \mathscr)

\usepackage{tikz,graphicx,color}
\usepackage{tikz-cd}%Create commutative diagrams with TikZ
\usepackage{tikz-3dplot}%Coordinate transformation styles for 3d plotting in TikZ
\usetikzlibrary{calc}%to make complex coordinate caculations.
\usetikzlibrary{arrows}%Provides various new and costumizable arrow tips 
\usetikzlibrary{shapes}%Shape library, used to define shapes other than rectangle, circle and co-ordinate
\usetikzlibrary{patterns}%This package "defines patterns for filling areas".
\usetikzlibrary{positioning}%a way of positioning one node in respect with several others
\usepackage{epstopdf}%Convert EPS to PDF using Ghostscript

\usepackage{enumerate}%Enumerate with redefinable labels. Precedes enumitem
\usepackage[arrow]{xy}%Typesetting Graphs and Diagrams in TeX 
\usepackage{diagbox}%Table heads with diagonal lines
\usepackage{subfig}%Figures broken into subfigures
\usepackage{cite}%Improved citation handling in LaTeX
\usepackage{chngcntr} %Change the resetting of counters
\usepackage{arcs}%Draw arcs over and under text. \overarc and \underarc
\usepackage{xcolor}%Driver-independent color extensions for LaTeX and pdfLaTeX

\newtheorem{theorem}{Theorem}[section]

\newtheorem{lemma}[theorem]{Lemma}
\newtheorem{proposition}[theorem]{Proposition}
\newtheorem{corollary}[theorem]{Corollary}
\theoremstyle{definition}
\newtheorem{remark}[theorem]{Remark}
\theoremstyle{definition}
\newtheorem{definition}[theorem]{Definition}

\theoremstyle{definition}

\theoremstyle{definition}
\newtheorem{example}[theorem]{Example}

\def\Acal{\mathcal{A}}\def\Bcal{\mathcal{B}}\def\Ccal{\mathcal{C}}\def\Dcal{\mathcal{D}}\def\Ecal{\mathcal{E}}\def\Fcal{\mathcal{F}}\def\Gcal{\mathcal{G}}\def\Hcal{\mathcal{H}}\def\Ical{\mathcal{I}}\def\Jcal{\mathcal{J}}\def\Kcal{\mathcal{K}}\def\Lcal{\mathcal{L}}\def\Mcal{\mathcal{M}}\def\Ncal{\mathcal{N}}\def\Ocal{\mathcal{O}}\def\Pcal{\mathcal{P}}\def\Qcal{\mathcal{Q}}\def\Rcal{\mathcal{R}}\def\Scal{\mathcal{S}}\def\Tcal{\mathcal{T}}\def\Ucal{\mathcal{U}}\def\Vcal{\mathcal{V}}\def\Wcal{\mathcal{W}}\def\Xcal{\mathcal{X}}\def\Ycal{\mathcal{Y}}\def\Zcal{\mathcal{Z}}

\def\abf{\mathbf{a}}\def\bbf{\mathbf{b}}\def\cbf{\mathbf{c}}\def\dbf{\mathbf{d}}\def\ebf{\mathbf{e}}\def\fbf{\mathbf{f}}\def\gbf{\mathbf{g}}\def\hbf{\mathbf{h}}\def\ibf{\mathbf{i}}\def\jbf{\mathbf{j}}\def\kbf{\mathbf{k}}\def\lbf{\mathbf{l}}\def\mbf{\mathbf{m}}\def\nbf{\mathbf{n}}\def\obf{\mathbf{o}}\def\pbf{\mathbf{p}}\def\qbf{\mathbf{q}}\def\rbf{\mathbf{r}}\def\sbf{\mathbf{s}}\def\tbf{\mathbf{t}}\def\ubf{\mathbf{u}}\def\vbf{\mathbf{v}}\def\wbf{\mathbf{w}}\def\xbf{\mathbf{x}}\def\ybf{\mathbf{y}}\def\zbf{\mathbf{z}}

\def\afr{\mathfrak{a}}
\def\Sfr{{ \mathfrak{S}}}
\def\one{{\mathbbm{1}}}
\def\C{\mathbb{C}}
\def\R{\mathbb{R}}
\def\N{\mathbb{N}}
\def\Z{\mathbb{Z}}
\def\Q{\mathbb{Q}}
\def\P{\mathbb{P}}
\def\Ahat{\hat{A}} 
\def\ipar{{(i)}}
\def\kpar{{(k)}}
\newcommand\parr[1]{{({#1})}}
\def\<{{\langle}}
\def\>{{\rangle}}
\def\e{{\epsilon}}
\def\l{{\lambda}}
\def\m{{\mu}}
\def\lm{{\l/\m}}
\def\RP{{\R P}}
\def\CP{{\C P}}
\def\multiset#1#2{\left(\!\left({#1\atopwithdelims..#2}\right)\!\right)}
\def\toi{{\xhookrightarrow{i}}}
\def\weakMap{\leadsto}

\def\Id{\operatorname{Id}}
\def\Vol{\operatorname{Vol}}
\def\det{{ \operatorname{det}}}
\def\tr{{ \operatorname{tr}}}
\def\Ker{{ \operatorname{Ker}}}
\def\Im{{ \operatorname{Im}}}
\def\diag{{ \operatorname{diag}}}
\def\rank{{ \operatorname{rank}}}
\def\rk{{\mathrm{rk}}}
\def\cork{ \operatorname{cork}}
\def\codim{ \operatorname{codim}}
\def\Vert{{ \operatorname{Vert}}}
\def\Edge{{ \operatorname{Edges}}}
\def\op{{ \operatorname{op}}}
\def\type{{\operatorname{type}}}
\def\sh{{ \operatorname{sh}}}
\def\Conv{ \operatorname{Conv}}
\def\Span{ \operatorname{Span}}
\def\proj{ \operatorname{proj}}

\def\TR{{\Delta}}
\def\Cube{{f}}
\def\TS{{\mathcal{S}}}
\def\TC{{\mathcal{O}}}
\def\TP{{\mathbb{P}}}
\def\x{{\xbf}}
\def\i{{\ibf}}
\def\Cyc{{ \Ccal}}
\def\Sol{{ \operatorname{Sol}}}
\def\Pbf{\mathbf{P}}
\def\Cbf{\mathbf{C}}
\def\sgn{ \operatorname{sgn}}%TODO MOVE TO THE TEMPLATE
\def\ut{u}
\def\vt{v}
\def\g{{\tilde{g}}}
\def\Nt{{\tilde{N}}}

\def\Vt{{{V}}}
\def\et{{{e}}}
\def\Et{{{E}}}
\def\Pt{{{P}}}
\def\wtld{{{\wt}}}
\def\wtt{{{w}}}

\def\wind{{ \operatorname{wind}}}
\def\ee{{\Nt}}
\def\gen{h}
\def\pathst{\Pcal}
\def\outg{ {\operatorname{out}_\Nt}}
\def\inc{ {\operatorname{in}_\Nt}}

\def\ubft{{{\ubf}}}
\def\wbft{{{\wbf}}}
\def\vbft{{{\vbf}}}
\def\Pbft{{{\Pbf}}}

\newcommand\pleh[1]{{[h_{#1}]}}
\newcommand\plee[1]{{[e_{#1}]}}
\def\Kbar{{\bar{K}}}

\usepackage{tabu}%TODO MOVE TO THE TEMPLATE
\usepackage{booktabs}% for better rules in the table

\begin{document}
\numberwithin{equation}{section}

\title{Periodicity and integrability for the cube recurrence}
\author{Pavel Galashin}
\address{Department of Mathematics, Massachusetts Institute of Technology,
Cambridge, MA 02139, USA}
\email{{\href{mailto:galashin@mit.edu}{galashin@mit.edu}}}

\date{\today}



\begin{abstract}
Zamolodchikov periodicity is a property of $T$- and $Y$-systems,
arising in the thermodynamic Bethe ansatz. Zamolodchikov integrability was
recently considered as its affine analog in our joint work with P.~Pylyavskyy. Here we prove periodicity and
integrability for similar discrete dynamical systems based on the cube recurrence, also known
as the discrete BKP equation. The periodicity part was conjectured by Henriques
in 2007.
\end{abstract}

\maketitle

\setcounter{tocdepth}{2}
\tableofcontents

\section{Introduction}
 
 The \emph{cube recurrence} is a discrete dynamical system that has been studied under two different guises. It was shown by Miwa~\cite{Miwa} in 1982 that the \emph{$\tau$-function} of the \emph{BKP hierarchy} satisfies a certain recurrence relation (equation~\eqref{eq:cube} below) which has been extensively studied afterwards, see e.g.~\cite{Schief,DM1,DM2,DM3}. The same recurrence relation was introduced by Propp~\cite{Propp} under the name \emph{cube recurrence} and was studied from the point of view of algebraic combinatorics~\cite{FZCube,CS,HS}. 
 
 In our recent work with Pavlo Pylyavskyy~\cite{GP1,GP2,GP3}, we investigated the behavior of \emph{$T$- and $Y$-systems}. These are discrete dynamical systems associated to a directed graph (a \emph{quiver}) $Q$. The celebrated \emph{Zamolodchikov periodicity conjecture}~\cite{Z,RVT,KN,KNS,FZ} states that when $Q$ is a \emph{tensor product} of two finite $ADE$ Dynkin diagrams then both of the systems associated with $Q$ are periodic. This conjecture has been proved by Keller~\cite{K} in 2013. In~\cite{GP1} we gave a complete classification of quivers for which the $T$-system is periodic (which is equivalent to the $Y$-system being periodic), we called these quivers \emph{finite $\boxtimes$ finite quivers}. For the case when $Q$ is a specific orientation of a square grid graph, the $T$-system associated with $Q$ becomes just the \emph{octahedron recurrence} introduced by Speyer~\cite{Sp}. Thus Zamolodchikov periodicity conjecture applied to this case states that the octahedron recurrence in a rectangle is periodic. This fact was shown by Volkov~\cite{Vo} and later by Di Francesco-Kedem~\cite{DK1}. 
 
 If instead of a rectangle one takes the octahedron recurrence in a cylinder, the values at every vertex satisfy a linear recurrence as it was shown by Pylyavskyy in~\cite{P}. In our joined work~\cite{GP2} later we gave a combinatorial formula for the recurrence coefficients in terms of domino tilings of the cylinder, conjectured that the same holds for all \emph{affine $\boxtimes$ finite quivers}, and showed that it does not hold for any other quiver. 
 
 Finally, in~\cite{GP3} we showed that if the $T$-system has zero \emph{algebraic entropy} then the quiver is an \emph{affine $\boxtimes$ affine quiver}, and we deduced the converse for the case of the octahedron recurrence in a torus as a simple consequence of Speyer's matching formula~\cite{Sp}. Moreover, we conjectured that for each of the affine $\boxtimes$ affine quivers, the degrees of the Laurent polynomials appearing in the $T$-system grow quadratically.
 
 In this text, we investigate the cube recurrence from a similar point of view. We show that the cube recurrence \emph{in a triangle} is periodic which has been conjectured by A.~Henriques~\cite{Henriques}. Next, for the cube recurrence \emph{in a cylinder}, we show that the values at every vertex satisfy a linear recurrence. We also give a combinatorial formula for the recurrence coefficients in terms of \emph{groves} introduced by Carroll and Speyer~\cite{CS}. Finally, we show that the cube recurrence \emph{in a torus} has zero algebraic entropy and its degrees grow quadratically. Both of these facts follow immediately from the results of~\cite{CS}.

 \section{Main results}
\subsection{The unbounded cube recurrence}

Let us recall the original definition of the cube recurrence from~\cite{Propp}. Since we will consider several bounded variations, we call it the \emph{unbounded} cube recurrence. For any $m\in\Z$, let $\TP_m=\{(i,j,k)\in\Z^3\mid i+j+k=m\}$ be a triangular lattice in the plane. For the unbounded case, we will concentrate on $\TP_0$. Let $e_{12}=(1,-1,0),e_{23}=(0,1,-1),$ and $e_{31}=(-1,0,1)$ be three vectors in $\TP_0$. For every vertex $v\in\TP_0$, we introduce a variable $x_v$ and we let $\x$ be the set of all these variables. For every vertex $v=(i,j,k)\in\TP_0$, we define its \emph{color} $\e_v\in\{0,1,2\}$ by $\e_v\equiv j-k\pmod 3\in\{0,1,2\}$. The following definition is an analog of the \emph{octahedron recurrence} of~\cite{Sp}. 

\begin{definition}\label{dfn:cube_unbounded}
 The \emph{unbounded cube recurrence} is a family $\Cube_v(t)$ of rational functions in $\x$ defined whenever $t\equiv \e_v\pmod 3$.  For every vertex $v=(i,j,k)\in\TP_0$ we set $\Cube_v(\e_v)=x_v$. For every such $v$ and every $t\equiv\e_v\pmod 3$, $\Cube_v$ satisfies
 \begin{equation}\label{eq:cube}
 \begin{split}
  \Cube_v(t+3)\Cube_v(t)&=\Cube_{v+e_{12}}(t+2)\Cube_{v-e_{12}}(t+1)+\\
  \Cube_{v+e_{23}}(t+2)&\Cube_{v-e_{23}}(t+1)+\Cube_{v+e_{31}}(t+2)\Cube_{v-e_{31}}(t+1).
 \end{split}
 \end{equation}
\end{definition}

One can easily observe that Definition~\ref{dfn:cube_unbounded} determines $\Cube_v(t)$ uniquely for any $t\equiv\e_v\pmod 3$. Propp~\cite{Propp} conjectured that the values of $\Cube_v(t)$ are Laurent polynomials in $\x$. This was proved by Fomin-Zelevinsky~\cite{FZCube}, and Carroll and Speyer~\cite{CS} later gave an explicit formula for them in terms of \emph{groves}, see Section~\ref{sect:groves}. %In particular, it follows from~\cite{CS} that $\Cube_v(t)$ is a \emph{positive} Laurent polynomial in $\x$.

We are going to consider three versions of the unbounded cube recurrence: the cube recurrence \emph{in a triangle}, \emph{in a cylinder}, and \emph{in a torus}.

\subsection{Cube recurrence in a triangle}

\begin{figure}
 
\makebox[\textwidth]{
\scalebox{0.7}{
\begin{tikzpicture}[scale=0.6]
\coordinate (L1) at (90:5.00);
\coordinate (L2) at (330:5.00);
\coordinate (L3) at (210:5.00);
\coordinate (C005) at (barycentric cs:L1=0.0,L2=0.0,L3=5.0);
\coordinate (C014) at (barycentric cs:L1=0.0,L2=1.0,L3=4.0);
\coordinate (C023) at (barycentric cs:L1=0.0,L2=2.0,L3=3.0);
\coordinate (C032) at (barycentric cs:L1=0.0,L2=3.0,L3=2.0);
\coordinate (C041) at (barycentric cs:L1=0.0,L2=4.0,L3=1.0);
\coordinate (C050) at (barycentric cs:L1=0.0,L2=5.0,L3=0.0);
\coordinate (C104) at (barycentric cs:L1=1.0,L2=0.0,L3=4.0);
\coordinate (C113) at (barycentric cs:L1=1.0,L2=1.0,L3=3.0);
\coordinate (C122) at (barycentric cs:L1=1.0,L2=2.0,L3=2.0);
\coordinate (C131) at (barycentric cs:L1=1.0,L2=3.0,L3=1.0);
\coordinate (C140) at (barycentric cs:L1=1.0,L2=4.0,L3=0.0);
\coordinate (C203) at (barycentric cs:L1=2.0,L2=0.0,L3=3.0);
\coordinate (C212) at (barycentric cs:L1=2.0,L2=1.0,L3=2.0);
\coordinate (C221) at (barycentric cs:L1=2.0,L2=2.0,L3=1.0);
\coordinate (C230) at (barycentric cs:L1=2.0,L2=3.0,L3=0.0);
\coordinate (C302) at (barycentric cs:L1=3.0,L2=0.0,L3=2.0);
\coordinate (C311) at (barycentric cs:L1=3.0,L2=1.0,L3=1.0);
\coordinate (C320) at (barycentric cs:L1=3.0,L2=2.0,L3=0.0);
\coordinate (C401) at (barycentric cs:L1=4.0,L2=0.0,L3=1.0);
\coordinate (C410) at (barycentric cs:L1=4.0,L2=1.0,L3=0.0);
\coordinate (C500) at (barycentric cs:L1=5.0,L2=0.0,L3=0.0);
\draw[line width=0.25mm,black] (C050) -- (C005);
\draw[line width=0.25mm,black] (C005) -- (C005);
\draw[line width=0.25mm,black] (C050) -- (C050);
\draw[line width=0.25mm,black] (C140) -- (C104);
\draw[line width=0.25mm,black] (C104) -- (C014);
\draw[line width=0.25mm,black] (C140) -- (C041);
\draw[line width=0.25mm,black] (C230) -- (C203);
\draw[line width=0.25mm,black] (C203) -- (C023);
\draw[line width=0.25mm,black] (C230) -- (C032);
\draw[line width=0.25mm,black] (C320) -- (C302);
\draw[line width=0.25mm,black] (C302) -- (C032);
\draw[line width=0.25mm,black] (C320) -- (C023);
\draw[line width=0.25mm,black] (C410) -- (C401);
\draw[line width=0.25mm,black] (C401) -- (C041);
\draw[line width=0.25mm,black] (C410) -- (C014);
\draw[line width=0.25mm,black] (C500) -- (C500);
\draw[line width=0.25mm,black] (C500) -- (C050);
\draw[line width=0.25mm,black] (C500) -- (C005);
\draw[fill=green] (barycentric cs:L1=0.0,L2=0.0,L3=5.0) circle (0.14285714285714288);
\draw[fill=red] (barycentric cs:L1=0.0,L2=1.0,L3=4.0) circle (0.14285714285714288);
\draw[fill=blue] (barycentric cs:L1=0.0,L2=2.0,L3=3.0) circle (0.14285714285714288);
\draw[fill=green] (barycentric cs:L1=0.0,L2=3.0,L3=2.0) circle (0.14285714285714288);
\draw[fill=red] (barycentric cs:L1=0.0,L2=4.0,L3=1.0) circle (0.14285714285714288);
\draw[fill=blue] (barycentric cs:L1=0.0,L2=5.0,L3=0.0) circle (0.14285714285714288);
\draw[fill=blue] (barycentric cs:L1=1.0,L2=0.0,L3=4.0) circle (0.14285714285714288);
\draw[fill=green] (barycentric cs:L1=1.0,L2=1.0,L3=3.0) circle (0.14285714285714288);
\draw[fill=red] (barycentric cs:L1=1.0,L2=2.0,L3=2.0) circle (0.14285714285714288);
\draw[fill=blue] (barycentric cs:L1=1.0,L2=3.0,L3=1.0) circle (0.14285714285714288);
\draw[fill=green] (barycentric cs:L1=1.0,L2=4.0,L3=0.0) circle (0.14285714285714288);
\draw[fill=red] (barycentric cs:L1=2.0,L2=0.0,L3=3.0) circle (0.14285714285714288);
\draw[fill=blue] (barycentric cs:L1=2.0,L2=1.0,L3=2.0) circle (0.14285714285714288);
\draw[fill=green] (barycentric cs:L1=2.0,L2=2.0,L3=1.0) circle (0.14285714285714288);
\draw[fill=red] (barycentric cs:L1=2.0,L2=3.0,L3=0.0) circle (0.14285714285714288);
\draw[fill=green] (barycentric cs:L1=3.0,L2=0.0,L3=2.0) circle (0.14285714285714288);
\draw[fill=red] (barycentric cs:L1=3.0,L2=1.0,L3=1.0) circle (0.14285714285714288);
\draw[fill=blue] (barycentric cs:L1=3.0,L2=2.0,L3=0.0) circle (0.14285714285714288);
\draw[fill=blue] (barycentric cs:L1=4.0,L2=0.0,L3=1.0) circle (0.14285714285714288);
\draw[fill=green] (barycentric cs:L1=4.0,L2=1.0,L3=0.0) circle (0.14285714285714288);
\draw[fill=red] (barycentric cs:L1=5.0,L2=0.0,L3=0.0) circle (0.14285714285714288);
\node[anchor=south west] (NC140) at (barycentric cs:L1=1.0,L2=4.0,L3=0.0) {$1$};
\node[anchor=south east] (NC104) at (barycentric cs:L1=1.0,L2=0.0,L3=4.0) {$1$};
\node[anchor=north] (NC014) at (barycentric cs:L1=0.0,L2=1.0,L3=4.0) {$1$};
\node[anchor=south west] (NC230) at (barycentric cs:L1=2.0,L2=3.0,L3=0.0) {$1$};
\node[anchor=south east] (NC203) at (barycentric cs:L1=2.0,L2=0.0,L3=3.0) {$1$};
\node[anchor=north] (NC023) at (barycentric cs:L1=0.0,L2=2.0,L3=3.0) {$1$};
\node[anchor=south west] (NC320) at (barycentric cs:L1=3.0,L2=2.0,L3=0.0) {$1$};
\node[anchor=south east] (NC302) at (barycentric cs:L1=3.0,L2=0.0,L3=2.0) {$1$};
\node[anchor=north] (NC032) at (barycentric cs:L1=0.0,L2=3.0,L3=2.0) {$1$};
\node[anchor=south west] (NC410) at (barycentric cs:L1=4.0,L2=1.0,L3=0.0) {$1$};
\node[anchor=south east] (NC401) at (barycentric cs:L1=4.0,L2=0.0,L3=1.0) {$1$};
\node[anchor=north] (NC041) at (barycentric cs:L1=0.0,L2=4.0,L3=1.0) {$1$};
\node[anchor=south] (NC100) at (barycentric cs:L1=1.0,L2=0.0,L3=0.0) {$1$};
\node[anchor=north west] (NC010) at (barycentric cs:L1=0.0,L2=1.0,L3=0.0) {$1$};
\node[anchor=north east] (NC001) at (barycentric cs:L1=0.0,L2=0.0,L3=1.0) {$1$};
\node[anchor=north] (NC113) at (barycentric cs:L1=1.0,L2=1.0,L3=3.0) {$a$};
\node[anchor=north] (NC122) at (barycentric cs:L1=1.0,L2=2.0,L3=2.0) {$b$};
\node[anchor=north] (NC131) at (barycentric cs:L1=1.0,L2=3.0,L3=1.0) {$c$};
\node[anchor=north] (NC212) at (barycentric cs:L1=2.0,L2=1.0,L3=2.0) {$d$};
\node[anchor=north] (NC221) at (barycentric cs:L1=2.0,L2=2.0,L3=1.0) {$e$};
\node[anchor=north] (NC311) at (barycentric cs:L1=3.0,L2=1.0,L3=1.0) {$f$};

  \begin{scope}[transform canvas={xshift = -1cm}]
  
\draw[->] (-5.00,2.00) -- (-6.00,3.73) node[anchor=south west] {$e_{12}$};
\draw[->] (-5.00,2.00) -- (-3.00,2.00) node[anchor=south] {$e_{23}$};
\draw[->] (-5.00,2.00) -- (-6.00,0.27) node[anchor=south east] {$e_{31}$};
\end{scope}
\end{tikzpicture}}
}
\caption{\label{fig:tr_5} The triangle $\TR_5$.} 
\end{figure}
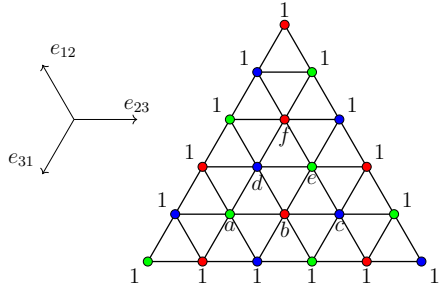

Given an integer $m\geq 3$, we define the $m$-th \emph{triangle} $\TR_m\subset\Z^3$ by $\TR_m=\{(i,j,k)\in\TP_m\mid i,j,k\geq 0\}$. For example, $\TR_5$ is shown in Figure~\ref{fig:tr_5}. We refer to $v\in\TR_m$ as a \emph{boundary vertex} if either one of $i,j,k$ is zero. For every non-boundary vertex $v$ we introduce a variable $x_v$ and we let $\x^\TR$ be the set of all these variables.

% \begin{equation}\label{blah}
%  \begin{split}
% A & = \frac{\pi r^2}{2} \\
%  & = \frac{1}{2} \pi r^2
%  \end{split}
% \end{equation}
% 
% \begin{align} \label{align}
% 2x - 5y &=  8 \\ 
% 3x + 9y &=  -12
% \end{align}

\def\Cubet{\Cube^\TR}
\def\Cubec{\Cube^\TC}
\def\Cubetorus{\Cube^\Torus}

\begin{definition}\label{dfn:cube_triangle}
 The \emph{cube recurrence in a triangle} is a family $\Cubet_v(t)$ of rational functions in $\x^\TR$ defined whenever $t\equiv \e_v\pmod 3$ and $v\in\TR_m$. For boundary vertices $v\in\TR_m$ we set $\Cubet_v(t)=1$ and for every non-boundary vertex $v=(i,j,k)\in\TR_m$ we set $\Cubet_v(\e_v)=x_v$. For every such $v$ and every $t\equiv\e_v\pmod 3$, $\Cubet_v$ satisfies~\eqref{eq:cube}. 
 \end{definition}

As we have already mentioned, the bounded version of the octahedron recurrence in a rectangle is periodic, see \cite{Vo,DK1,K}. Our first main result is the following analogous assertion for the cube recurrence in a triangle, which was conjectured by Henriques in~\cite[Section~6]{Henriques}:
 
\begin{theorem}\label{thm:cube_periodic}
  The values of the cube recurrence in a triangle $\TR_m$ are Laurent polynomials in $\x^\TR$. Moreover, let $\sigma:\TR_m\to\TR_m$ be the counterclockwise rotation of $\TR_m$: $\sigma(i,j,k)=(j,k,i)$. Then for every $v\in\TR_m$ and every $t\equiv\e_v\pmod3$, we have 
  \[\Cubet_v(t+2m)=\Cubet_{\sigma v} (t).\] 
  Thus the cube recurrence in a triangle satisfies $\Cubet_v(t+6m)=\Cubet_v(t)$.%TODO: positive?
\end{theorem}
Our proof of Theorem~\ref{thm:cube_periodic} in Section~\ref{sect:periodicity} is based on Henriques and Speyer's \emph{multidimensional cube recurrence}~\cite{HS}.

\subsection{Cube recurrence in a cylinder}

\def\wt{ \operatorname{wt}}
\def\htt{ \operatorname{ht}}
\def\adj{ \operatorname{adj}}

We define the \emph{cube recurrence in a cylinder} as follows. Let $m\geq 2,n\geq 1$ be two integers and define the \emph{strip} 
\[\TS_m=\{(i,j,k)\in\TP_0\mid 0\leq i\leq m\}.\] 
We let $g$ be the vector $ne_{23}=(0,n,-n)$ and consider the cylinder $\TC=\TS_m/3\Z g$. Informally speaking, the \emph{cube recurrence in a cylinder} is just the cube recurrence in a strip $\TS_m$ with initial conditions being invariant with respect to the shift by $3g$. Let us explain this more precisely. For every $v=(i,j,k)\in\TS_{m}$ with $0<i<m$, we introduce a variable $x_v$. For any $k\in\Z$, we set $x_{v+3kg}=x_v$. Let $\x^\TC$ be the (finite) family of these indeterminates. We say that $v=(i,j,k)\in\TS_m$ is a \emph{boundary vertex} if $i=0$ or $i=m$. 

\begin{definition}\label{dfn:cube_cylinder}
 The \emph{cube recurrence in a cylinder} is a family $\Cubec_v(t)$ of rational functions in $\x^\TC$ for $v\in\TS_{m}$ that satisfies~\eqref{eq:cube} for every non-boundary vertex $v\in\TS_m$. The boundary conditions are $\Cubec_v(t)=1$ for all $t\in\Z$ and $v$ a boundary vertex, and also $\Cubec_v(\e_v)=x_v$ for all non-boundary $v\in\TS_m$.
\end{definition}

As we have explained in the introduction, the values of the octahedron recurrence in a cylinder satisfy a linear recurrence whose coefficients admit a nice formula in terms of domino tilings~\cite{P,GP2}. Theorems~\ref{thm:recurrence_1},~\ref{thm:cube_formula}, and~\ref{thm:pleth} below give analogous statements for the cube recurrence in a cylinder.

\begin{theorem}\label{thm:recurrence_1}
 Fix any $n\geq1$ and $m\geq2$ and let $v\in\TS_{m}$ be a vertex. Then the sequence \[\left(\Cubec_v(\e_v+3t)\right)_{t\in\N}\] 
 satisfies a linear recurrence: there exist Laurent polynomials $H_0,H_1,\dots,H_M$ in $\x^\TC$ such that $H_0,H_M\neq 0$ and
 \[H_0\Cubec_v(\e_v+3t)+H_1\Cubec_v(\e_v+3(t+1))+\dots+H_M\Cubec_v(\e_v+3(t+M))=0\]
 holds for any $t\in\N$ that is sufficiently large. 
\end{theorem}

\begin{figure}
 
\makebox[\textwidth]{
\begin{tabular}{c|c}

\scalebox{0.7}{
\begin{tikzpicture}[scale=0.6]
\coordinate (L1) at (90:5.00);
\coordinate (L2) at (330:5.00);
\coordinate (L3) at (210:5.00);
\coordinate (C005) at (barycentric cs:L1=0.0,L2=0.0,L3=5.0);
\coordinate (C014) at (barycentric cs:L1=0.0,L2=1.0,L3=4.0);
\coordinate (C023) at (barycentric cs:L1=0.0,L2=2.0,L3=3.0);
\coordinate (C032) at (barycentric cs:L1=0.0,L2=3.0,L3=2.0);
\coordinate (C041) at (barycentric cs:L1=0.0,L2=4.0,L3=1.0);
\coordinate (C050) at (barycentric cs:L1=0.0,L2=5.0,L3=0.0);
\coordinate (C104) at (barycentric cs:L1=1.0,L2=0.0,L3=4.0);
\coordinate (C113) at (barycentric cs:L1=1.0,L2=1.0,L3=3.0);
\coordinate (C122) at (barycentric cs:L1=1.0,L2=2.0,L3=2.0);
\coordinate (C131) at (barycentric cs:L1=1.0,L2=3.0,L3=1.0);
\coordinate (C140) at (barycentric cs:L1=1.0,L2=4.0,L3=0.0);
\coordinate (C15m1) at (barycentric cs:L1=1.0,L2=5.0,L3=-1.0);
\coordinate (C2m14) at (barycentric cs:L1=2.0,L2=-1.0,L3=4.0);
\coordinate (C203) at (barycentric cs:L1=2.0,L2=0.0,L3=3.0);
\coordinate (C212) at (barycentric cs:L1=2.0,L2=1.0,L3=2.0);
\coordinate (C221) at (barycentric cs:L1=2.0,L2=2.0,L3=1.0);
\coordinate (C230) at (barycentric cs:L1=2.0,L2=3.0,L3=0.0);
\coordinate (C24m1) at (barycentric cs:L1=2.0,L2=4.0,L3=-1.0);
\draw[line width=0.25mm,black] (C113) -- (C203);
\draw[line width=0.25mm,black] (C113) -- (C014);
\draw[line width=0.25mm,black] (C113) -- (C122);
\draw[line width=0.25mm,black] (C131) -- (C221);
\draw[line width=0.25mm,black] (C131) -- (C032);
\draw[line width=0.25mm,black] (C131) -- (C140);
\draw[line width=0.25mm,black] (C122) -- (C212);
\draw[line width=0.25mm,black] (C122) -- (C023);
\draw[line width=0.25mm,black] (C122) -- (C131);
\draw[line width=0.25mm,black] (C212) -- (C113);
\draw[line width=0.25mm,black] (C212) -- (C221);
\draw[line width=0.25mm,black] (C221) -- (C122);
\draw[line width=0.25mm,black] (C221) -- (C230);
\draw[line width=0.25mm,black] (C014) -- (C104);
\draw[line width=0.25mm,black] (C014) -- (C023);
\draw[line width=0.25mm,black] (C041) -- (C131);
\draw[line width=0.25mm,black] (C041) -- (C050);
\draw[line width=0.25mm,black] (C104) -- (C2m14);
\draw[line width=0.25mm,black] (C104) -- (C005);
\draw[line width=0.25mm,black] (C104) -- (C113);
\draw[line width=0.25mm,black] (C140) -- (C230);
\draw[line width=0.25mm,black] (C140) -- (C041);
\draw[line width=0.25mm,black] (C140) -- (C15m1);
\draw[line width=0.25mm,black] (C15m1) -- (C24m1);
\draw[line width=0.25mm,black] (C15m1) -- (C050);
\draw[line width=0.25mm,black] (C023) -- (C113);
\draw[line width=0.25mm,black] (C023) -- (C032);
\draw[line width=0.25mm,black] (C032) -- (C122);
\draw[line width=0.25mm,black] (C032) -- (C041);
\draw[line width=0.25mm,black] (C203) -- (C104);
\draw[line width=0.25mm,black] (C203) -- (C212);
\draw[line width=0.25mm,black] (C230) -- (C131);
\draw[line width=0.25mm,black] (C230) -- (C24m1);
\draw[line width=0.25mm,black] (C2m14) -- (C203);
\draw[line width=0.25mm,black] (C24m1) -- (C140);
\draw[line width=0.25mm,black] (C005) -- (C014);
\draw[line width=0.25mm,black] (C050) -- (C140);
\draw[fill=green] (barycentric cs:L1=1.0,L2=1.0,L3=3.0) circle (0.14285714285714288);
\draw[fill=blue] (barycentric cs:L1=1.0,L2=3.0,L3=1.0) circle (0.14285714285714288);
\draw[fill=red] (barycentric cs:L1=1.0,L2=2.0,L3=2.0) circle (0.14285714285714288);
\draw[fill=blue] (barycentric cs:L1=2.0,L2=1.0,L3=2.0) circle (0.14285714285714288);
\draw[fill=green] (barycentric cs:L1=2.0,L2=2.0,L3=1.0) circle (0.14285714285714288);
\draw[fill=red] (barycentric cs:L1=0.0,L2=1.0,L3=4.0) circle (0.14285714285714288);
\draw[fill=red] (barycentric cs:L1=0.0,L2=4.0,L3=1.0) circle (0.14285714285714288);
\draw[fill=blue] (barycentric cs:L1=1.0,L2=0.0,L3=4.0) circle (0.14285714285714288);
\draw[fill=green] (barycentric cs:L1=1.0,L2=4.0,L3=0.0) circle (0.14285714285714288);
\draw[fill=red] (barycentric cs:L1=1.0,L2=5.0,L3=-1.0) circle (0.14285714285714288);
\draw[fill=blue] (barycentric cs:L1=0.0,L2=2.0,L3=3.0) circle (0.14285714285714288);
\draw[fill=green] (barycentric cs:L1=0.0,L2=3.0,L3=2.0) circle (0.14285714285714288);
\draw[fill=red] (barycentric cs:L1=2.0,L2=0.0,L3=3.0) circle (0.14285714285714288);
\draw[fill=red] (barycentric cs:L1=2.0,L2=3.0,L3=0.0) circle (0.14285714285714288);
\draw[fill=green] (barycentric cs:L1=2.0,L2=-1.0,L3=4.0) circle (0.14285714285714288);
\draw[fill=blue] (barycentric cs:L1=2.0,L2=4.0,L3=-1.0) circle (0.14285714285714288);
\draw[fill=green] (barycentric cs:L1=0.0,L2=0.0,L3=5.0) circle (0.14285714285714288);
\draw[fill=blue] (barycentric cs:L1=0.0,L2=5.0,L3=0.0) circle (0.14285714285714288);
\node[anchor=north] (NC005) at (C005) {$1$};
\node[anchor=north] (NC014) at (C014) {$1$};
\node[anchor=north] (NC023) at (C023) {$1$};
\node[anchor=north] (NC032) at (C032) {$1$};
\node[anchor=north] (NC041) at (C041) {$1$};
\node[anchor=north] (NC050) at (C050) {$1$};
\node[anchor=north] (NC104) at (C104) {$c$};
\node[anchor=north] (NC113) at (C113) {$b$};
\node[anchor=north] (NC122) at (C122) {$a$};
\node[anchor=north] (NC131) at (C131) {$c$};
\node[anchor=north] (NC140) at (C140) {$b$};
\node[anchor=north] (NC15m1) at (C15m1) {$a$};
\node[anchor=south] (NC2m14) at (C2m14) {$1$};
\node[anchor=south] (NC203) at (C203) {$1$};
\node[anchor=south] (NC212) at (C212) {$1$};
\node[anchor=south] (NC221) at (C221) {$1$};
\node[anchor=south] (NC230) at (C230) {$1$};
\node[anchor=south] (NC24m1) at (C24m1) {$1$};
\node[anchor=east] (DC104) at (C104) {$\dots$};
\node[anchor=west] (DC15m1) at (C15m1) {$\dots$};
\end{tikzpicture}}
&
\scalebox{1.0}{
\begin{tikzpicture}[scale=0.3]
\coordinate (L1) at (90:7.00);
\coordinate (L2) at (330:7.00);
\coordinate (L3) at (210:7.00);
\coordinate (C007) at (barycentric cs:L1=0.0,L2=0.0,L3=7.0);
\coordinate (C016) at (barycentric cs:L1=0.0,L2=1.0,L3=6.0);
\coordinate (C025) at (barycentric cs:L1=0.0,L2=2.0,L3=5.0);
\coordinate (C034) at (barycentric cs:L1=0.0,L2=3.0,L3=4.0);
\coordinate (C043) at (barycentric cs:L1=0.0,L2=4.0,L3=3.0);
\coordinate (C052) at (barycentric cs:L1=0.0,L2=5.0,L3=2.0);
\coordinate (C061) at (barycentric cs:L1=0.0,L2=6.0,L3=1.0);
\coordinate (C070) at (barycentric cs:L1=0.0,L2=7.0,L3=0.0);
\coordinate (C106) at (barycentric cs:L1=1.0,L2=0.0,L3=6.0);
\coordinate (C115) at (barycentric cs:L1=1.0,L2=1.0,L3=5.0);
\coordinate (C124) at (barycentric cs:L1=1.0,L2=2.0,L3=4.0);
\coordinate (C133) at (barycentric cs:L1=1.0,L2=3.0,L3=3.0);
\coordinate (C142) at (barycentric cs:L1=1.0,L2=4.0,L3=2.0);
\coordinate (C151) at (barycentric cs:L1=1.0,L2=5.0,L3=1.0);
\coordinate (C160) at (barycentric cs:L1=1.0,L2=6.0,L3=0.0);
\coordinate (C17m1) at (barycentric cs:L1=1.0,L2=7.0,L3=-1.0);
\coordinate (C2m16) at (barycentric cs:L1=2.0,L2=-1.0,L3=6.0);
\coordinate (C205) at (barycentric cs:L1=2.0,L2=0.0,L3=5.0);
\coordinate (C214) at (barycentric cs:L1=2.0,L2=1.0,L3=4.0);
\coordinate (C223) at (barycentric cs:L1=2.0,L2=2.0,L3=3.0);
\coordinate (C232) at (barycentric cs:L1=2.0,L2=3.0,L3=2.0);
\coordinate (C241) at (barycentric cs:L1=2.0,L2=4.0,L3=1.0);
\coordinate (C250) at (barycentric cs:L1=2.0,L2=5.0,L3=0.0);
\coordinate (C26m1) at (barycentric cs:L1=2.0,L2=6.0,L3=-1.0);
\coordinate (C3m15) at (barycentric cs:L1=3.0,L2=-1.0,L3=5.0);
\coordinate (C304) at (barycentric cs:L1=3.0,L2=0.0,L3=4.0);
\coordinate (C313) at (barycentric cs:L1=3.0,L2=1.0,L3=3.0);
\coordinate (C322) at (barycentric cs:L1=3.0,L2=2.0,L3=2.0);
\coordinate (C331) at (barycentric cs:L1=3.0,L2=3.0,L3=1.0);
\coordinate (C340) at (barycentric cs:L1=3.0,L2=4.0,L3=0.0);
\coordinate (C35m1) at (barycentric cs:L1=3.0,L2=5.0,L3=-1.0);
\coordinate (C36m2) at (barycentric cs:L1=3.0,L2=6.0,L3=-2.0);
\coordinate (C4m25) at (barycentric cs:L1=4.0,L2=-2.0,L3=5.0);
\coordinate (C4m14) at (barycentric cs:L1=4.0,L2=-1.0,L3=4.0);
\coordinate (C403) at (barycentric cs:L1=4.0,L2=0.0,L3=3.0);
\coordinate (C412) at (barycentric cs:L1=4.0,L2=1.0,L3=2.0);
\coordinate (C421) at (barycentric cs:L1=4.0,L2=2.0,L3=1.0);
\coordinate (C430) at (barycentric cs:L1=4.0,L2=3.0,L3=0.0);
\coordinate (C44m1) at (barycentric cs:L1=4.0,L2=4.0,L3=-1.0);
\coordinate (C45m2) at (barycentric cs:L1=4.0,L2=5.0,L3=-2.0);
\coordinate (C5m24) at (barycentric cs:L1=5.0,L2=-2.0,L3=4.0);
\coordinate (C5m13) at (barycentric cs:L1=5.0,L2=-1.0,L3=3.0);
\coordinate (C502) at (barycentric cs:L1=5.0,L2=0.0,L3=2.0);
\coordinate (C511) at (barycentric cs:L1=5.0,L2=1.0,L3=1.0);
\coordinate (C520) at (barycentric cs:L1=5.0,L2=2.0,L3=0.0);
\coordinate (C53m1) at (barycentric cs:L1=5.0,L2=3.0,L3=-1.0);
\coordinate (C54m2) at (barycentric cs:L1=5.0,L2=4.0,L3=-2.0);
\coordinate (C55m3) at (barycentric cs:L1=5.0,L2=5.0,L3=-3.0);
\coordinate (C6m34) at (barycentric cs:L1=6.0,L2=-3.0,L3=4.0);
\coordinate (C6m23) at (barycentric cs:L1=6.0,L2=-2.0,L3=3.0);
\coordinate (C6m12) at (barycentric cs:L1=6.0,L2=-1.0,L3=2.0);
\coordinate (C601) at (barycentric cs:L1=6.0,L2=0.0,L3=1.0);
\coordinate (C610) at (barycentric cs:L1=6.0,L2=1.0,L3=0.0);
\coordinate (C62m1) at (barycentric cs:L1=6.0,L2=2.0,L3=-1.0);
\coordinate (C63m2) at (barycentric cs:L1=6.0,L2=3.0,L3=-2.0);
\coordinate (C64m3) at (barycentric cs:L1=6.0,L2=4.0,L3=-3.0);
\coordinate (C7m33) at (barycentric cs:L1=7.0,L2=-3.0,L3=3.0);
\coordinate (C7m22) at (barycentric cs:L1=7.0,L2=-2.0,L3=2.0);
\coordinate (C7m11) at (barycentric cs:L1=7.0,L2=-1.0,L3=1.0);
\coordinate (C700) at (barycentric cs:L1=7.0,L2=0.0,L3=0.0);
\coordinate (C71m1) at (barycentric cs:L1=7.0,L2=1.0,L3=-1.0);
\coordinate (C72m2) at (barycentric cs:L1=7.0,L2=2.0,L3=-2.0);
\coordinate (C73m3) at (barycentric cs:L1=7.0,L2=3.0,L3=-3.0);
\coordinate (C74m4) at (barycentric cs:L1=7.0,L2=4.0,L3=-4.0);
\draw[line width=0.25mm,black] (C115) -- (C205);
\draw[line width=0.25mm,black] (C115) -- (C016);
\draw[line width=0.25mm,black] (C115) -- (C124);
\draw[line width=0.25mm,black] (C151) -- (C241);
\draw[line width=0.25mm,black] (C151) -- (C052);
\draw[line width=0.25mm,black] (C151) -- (C160);
\draw[line width=0.25mm,black] (C124) -- (C214);
\draw[line width=0.25mm,black] (C124) -- (C025);
\draw[line width=0.25mm,black] (C124) -- (C133);
\draw[line width=0.25mm,black] (C142) -- (C232);
\draw[line width=0.25mm,black] (C142) -- (C043);
\draw[line width=0.25mm,black] (C142) -- (C151);
\draw[line width=0.25mm,black] (C412) -- (C502);
\draw[line width=0.25mm,black] (C412) -- (C313);
\draw[line width=0.25mm,black] (C412) -- (C421);
\draw[line width=0.25mm,black] (C421) -- (C511);
\draw[line width=0.25mm,black] (C421) -- (C322);
\draw[line width=0.25mm,black] (C421) -- (C430);
\draw[line width=0.25mm,black] (C313) -- (C403);
\draw[line width=0.25mm,black] (C313) -- (C214);
\draw[line width=0.25mm,black] (C313) -- (C322);
\draw[line width=0.25mm,black] (C331) -- (C421);
\draw[line width=0.25mm,black] (C331) -- (C232);
\draw[line width=0.25mm,black] (C331) -- (C340);
\draw[line width=0.25mm,black] (C016) -- (C106);
\draw[line width=0.25mm,black] (C016) -- (C025);
\draw[line width=0.25mm,black] (C061) -- (C151);
\draw[line width=0.25mm,black] (C061) -- (C070);
\draw[line width=0.25mm,black] (C601) -- (C7m11);
\draw[line width=0.25mm,black] (C601) -- (C502);
\draw[line width=0.25mm,black] (C601) -- (C610);
\draw[line width=0.25mm,black] (C610) -- (C700);
\draw[line width=0.25mm,black] (C610) -- (C511);
\draw[line width=0.25mm,black] (C610) -- (C62m1);
\draw[line width=0.25mm,black] (C17m1) -- (C26m1);
\draw[line width=0.25mm,black] (C17m1) -- (C070);
\draw[line width=0.25mm,black] (C7m11) -- (C6m12);
\draw[line width=0.25mm,black] (C7m11) -- (C700);
\draw[line width=0.25mm,black] (C71m1) -- (C610);
\draw[line width=0.25mm,black] (C71m1) -- (C72m2);
\draw[line width=0.25mm,black] (C223) -- (C313);
\draw[line width=0.25mm,black] (C223) -- (C124);
\draw[line width=0.25mm,black] (C223) -- (C232);
\draw[line width=0.25mm,black] (C232) -- (C322);
\draw[line width=0.25mm,black] (C232) -- (C133);
\draw[line width=0.25mm,black] (C232) -- (C241);
\draw[line width=0.25mm,black] (C205) -- (C3m15);
\draw[line width=0.25mm,black] (C205) -- (C106);
\draw[line width=0.25mm,black] (C205) -- (C214);
\draw[line width=0.25mm,black] (C250) -- (C340);
\draw[line width=0.25mm,black] (C250) -- (C151);
\draw[line width=0.25mm,black] (C250) -- (C26m1);
\draw[line width=0.25mm,black] (C502) -- (C6m12);
\draw[line width=0.25mm,black] (C502) -- (C403);
\draw[line width=0.25mm,black] (C502) -- (C511);
\draw[line width=0.25mm,black] (C520) -- (C610);
\draw[line width=0.25mm,black] (C520) -- (C421);
\draw[line width=0.25mm,black] (C520) -- (C53m1);
\draw[line width=0.25mm,black] (C2m16) -- (C205);
\draw[line width=0.25mm,black] (C26m1) -- (C35m1);
\draw[line width=0.25mm,black] (C26m1) -- (C160);
\draw[line width=0.25mm,black] (C7m22) -- (C6m23);
\draw[line width=0.25mm,black] (C7m22) -- (C7m11);
\draw[line width=0.25mm,black] (C72m2) -- (C62m1);
\draw[line width=0.25mm,black] (C72m2) -- (C73m3);
\draw[line width=0.25mm,black] (C034) -- (C124);
\draw[line width=0.25mm,black] (C034) -- (C043);
\draw[line width=0.25mm,black] (C043) -- (C133);
\draw[line width=0.25mm,black] (C043) -- (C052);
\draw[line width=0.25mm,black] (C304) -- (C4m14);
\draw[line width=0.25mm,black] (C304) -- (C205);
\draw[line width=0.25mm,black] (C304) -- (C313);
\draw[line width=0.25mm,black] (C340) -- (C430);
\draw[line width=0.25mm,black] (C340) -- (C241);
\draw[line width=0.25mm,black] (C340) -- (C35m1);
\draw[line width=0.25mm,black] (C5m13) -- (C6m23);
\draw[line width=0.25mm,black] (C5m13) -- (C4m14);
\draw[line width=0.25mm,black] (C5m13) -- (C502);
\draw[line width=0.25mm,black] (C53m1) -- (C62m1);
\draw[line width=0.25mm,black] (C53m1) -- (C430);
\draw[line width=0.25mm,black] (C53m1) -- (C54m2);
\draw[line width=0.25mm,black] (C36m2) -- (C45m2);
\draw[line width=0.25mm,black] (C36m2) -- (C26m1);
\draw[line width=0.25mm,black] (C6m23) -- (C7m33);
\draw[line width=0.25mm,black] (C6m23) -- (C5m24);
\draw[line width=0.25mm,black] (C6m23) -- (C6m12);
\draw[line width=0.25mm,black] (C63m2) -- (C72m2);
\draw[line width=0.25mm,black] (C63m2) -- (C53m1);
\draw[line width=0.25mm,black] (C63m2) -- (C64m3);
\draw[line width=0.25mm,black] (C4m14) -- (C5m24);
\draw[line width=0.25mm,black] (C4m14) -- (C3m15);
\draw[line width=0.25mm,black] (C4m14) -- (C403);
\draw[line width=0.25mm,black] (C44m1) -- (C53m1);
\draw[line width=0.25mm,black] (C44m1) -- (C340);
\draw[line width=0.25mm,black] (C44m1) -- (C45m2);
\draw[line width=0.25mm,black] (C4m25) -- (C4m14);
\draw[line width=0.25mm,black] (C45m2) -- (C54m2);
\draw[line width=0.25mm,black] (C45m2) -- (C35m1);
\draw[line width=0.25mm,black] (C6m34) -- (C6m23);
\draw[line width=0.25mm,black] (C64m3) -- (C73m3);
\draw[line width=0.25mm,black] (C64m3) -- (C54m2);
\draw[line width=0.25mm,black] (C74m4) -- (C64m3);
\draw[line width=0.25mm,black] (C55m3) -- (C64m3);
\draw[line width=0.25mm,black] (C55m3) -- (C45m2);
\draw[line width=0.25mm,black] (C007) -- (C016);
\draw[line width=0.25mm,black] (C070) -- (C160);
\draw[fill=blue] (barycentric cs:L1=1.0,L2=1.0,L3=5.0) circle (0.2);
\draw[fill=green] (barycentric cs:L1=1.0,L2=5.0,L3=1.0) circle (0.2);
\draw[fill=red] (barycentric cs:L1=5.0,L2=1.0,L3=1.0) circle (0.2);
\draw[fill=green] (barycentric cs:L1=1.0,L2=2.0,L3=4.0) circle (0.2);
\draw[fill=blue] (barycentric cs:L1=1.0,L2=4.0,L3=2.0) circle (0.2);
\draw[fill=red] (barycentric cs:L1=2.0,L2=1.0,L3=4.0) circle (0.2);
\draw[fill=red] (barycentric cs:L1=2.0,L2=4.0,L3=1.0) circle (0.2);
\draw[fill=blue] (barycentric cs:L1=4.0,L2=1.0,L3=2.0) circle (0.2);
\draw[fill=green] (barycentric cs:L1=4.0,L2=2.0,L3=1.0) circle (0.2);
\draw[fill=red] (barycentric cs:L1=1.0,L2=3.0,L3=3.0) circle (0.2);
\draw[fill=green] (barycentric cs:L1=3.0,L2=1.0,L3=3.0) circle (0.2);
\draw[fill=blue] (barycentric cs:L1=3.0,L2=3.0,L3=1.0) circle (0.2);
\draw[fill=green] (barycentric cs:L1=0.0,L2=1.0,L3=6.0) circle (0.2);
\draw[fill=blue] (barycentric cs:L1=0.0,L2=6.0,L3=1.0) circle (0.2);
\draw[fill=red] (barycentric cs:L1=1.0,L2=0.0,L3=6.0) circle (0.2);
\draw[fill=red] (barycentric cs:L1=1.0,L2=6.0,L3=0.0) circle (0.2);
\draw[fill=blue] (barycentric cs:L1=6.0,L2=0.0,L3=1.0) circle (0.2);
\draw[fill=green] (barycentric cs:L1=6.0,L2=1.0,L3=0.0) circle (0.2);
\draw[fill=blue] (barycentric cs:L1=1.0,L2=7.0,L3=-1.0) circle (0.2);
\draw[fill=green] (barycentric cs:L1=7.0,L2=-1.0,L3=1.0) circle (0.2);
\draw[fill=blue] (barycentric cs:L1=7.0,L2=1.0,L3=-1.0) circle (0.2);
\draw[fill=blue] (barycentric cs:L1=2.0,L2=2.0,L3=3.0) circle (0.2);
\draw[fill=green] (barycentric cs:L1=2.0,L2=3.0,L3=2.0) circle (0.2);
\draw[fill=red] (barycentric cs:L1=3.0,L2=2.0,L3=2.0) circle (0.2);
\draw[fill=red] (barycentric cs:L1=0.0,L2=2.0,L3=5.0) circle (0.2);
\draw[fill=red] (barycentric cs:L1=0.0,L2=5.0,L3=2.0) circle (0.2);
\draw[fill=green] (barycentric cs:L1=2.0,L2=0.0,L3=5.0) circle (0.2);
\draw[fill=blue] (barycentric cs:L1=2.0,L2=5.0,L3=0.0) circle (0.2);
\draw[fill=green] (barycentric cs:L1=5.0,L2=0.0,L3=2.0) circle (0.2);
\draw[fill=blue] (barycentric cs:L1=5.0,L2=2.0,L3=0.0) circle (0.2);
\draw[fill=blue] (barycentric cs:L1=2.0,L2=-1.0,L3=6.0) circle (0.2);
\draw[fill=green] (barycentric cs:L1=2.0,L2=6.0,L3=-1.0) circle (0.2);
\draw[fill=red] (barycentric cs:L1=6.0,L2=-1.0,L3=2.0) circle (0.2);
\draw[fill=red] (barycentric cs:L1=6.0,L2=2.0,L3=-1.0) circle (0.2);
\draw[fill=blue] (barycentric cs:L1=7.0,L2=-2.0,L3=2.0) circle (0.2);
\draw[fill=green] (barycentric cs:L1=7.0,L2=2.0,L3=-2.0) circle (0.2);
\draw[fill=blue] (barycentric cs:L1=0.0,L2=3.0,L3=4.0) circle (0.2);
\draw[fill=green] (barycentric cs:L1=0.0,L2=4.0,L3=3.0) circle (0.2);
\draw[fill=blue] (barycentric cs:L1=3.0,L2=0.0,L3=4.0) circle (0.2);
\draw[fill=green] (barycentric cs:L1=3.0,L2=4.0,L3=0.0) circle (0.2);
\draw[fill=red] (barycentric cs:L1=4.0,L2=0.0,L3=3.0) circle (0.2);
\draw[fill=red] (barycentric cs:L1=4.0,L2=3.0,L3=0.0) circle (0.2);
\draw[fill=red] (barycentric cs:L1=3.0,L2=-1.0,L3=5.0) circle (0.2);
\draw[fill=red] (barycentric cs:L1=3.0,L2=5.0,L3=-1.0) circle (0.2);
\draw[fill=blue] (barycentric cs:L1=5.0,L2=-1.0,L3=3.0) circle (0.2);
\draw[fill=green] (barycentric cs:L1=5.0,L2=3.0,L3=-1.0) circle (0.2);
\draw[fill=blue] (barycentric cs:L1=3.0,L2=6.0,L3=-2.0) circle (0.2);
\draw[fill=green] (barycentric cs:L1=6.0,L2=-2.0,L3=3.0) circle (0.2);
\draw[fill=blue] (barycentric cs:L1=6.0,L2=3.0,L3=-2.0) circle (0.2);
\draw[fill=red] (barycentric cs:L1=7.0,L2=-3.0,L3=3.0) circle (0.2);
\draw[fill=red] (barycentric cs:L1=7.0,L2=3.0,L3=-3.0) circle (0.2);
\draw[fill=green] (barycentric cs:L1=4.0,L2=-1.0,L3=4.0) circle (0.2);
\draw[fill=blue] (barycentric cs:L1=4.0,L2=4.0,L3=-1.0) circle (0.2);
\draw[fill=blue] (barycentric cs:L1=4.0,L2=-2.0,L3=5.0) circle (0.2);
\draw[fill=green] (barycentric cs:L1=4.0,L2=5.0,L3=-2.0) circle (0.2);
\draw[fill=red] (barycentric cs:L1=5.0,L2=-2.0,L3=4.0) circle (0.2);
\draw[fill=red] (barycentric cs:L1=5.0,L2=4.0,L3=-2.0) circle (0.2);
\draw[fill=blue] (barycentric cs:L1=6.0,L2=-3.0,L3=4.0) circle (0.2);
\draw[fill=green] (barycentric cs:L1=6.0,L2=4.0,L3=-3.0) circle (0.2);
\draw[fill=blue] (barycentric cs:L1=7.0,L2=4.0,L3=-4.0) circle (0.2);
\draw[fill=blue] (barycentric cs:L1=5.0,L2=5.0,L3=-3.0) circle (0.2);
\draw[fill=blue] (barycentric cs:L1=0.0,L2=0.0,L3=7.0) circle (0.2);
\draw[fill=green] (barycentric cs:L1=0.0,L2=7.0,L3=0.0) circle (0.2);
\draw[fill=red] (barycentric cs:L1=7.0,L2=0.0,L3=0.0) circle (0.2);
\end{tikzpicture}}
\\

\end{tabular}
}
\caption{\label{fig:strip} The strip $\TS_2$ with initial values shown for $n=1$ (left). The graph $G$ (right). The red, green, and blue colors correspond to $\e_v=0,1,2$ respectively.} 
\end{figure}
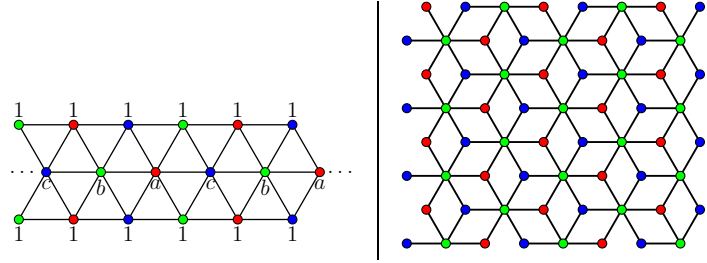

Let us now describe a formula for the recurrence coefficients $H_0,\dots,H_M$ when $v=(m-1,j,k)$ for some $j$ and $k$. We recall the definition of groves from~\cite{CS}. Consider the following infinite undirected graph $G$ with vertex set $\TP_0$ and edge set consisting of edges $(v,v+e_{12})$, $(v,v+e_{23})$, and $(v,v+e_{31})$ for every vertex $v\in\TP_0$ with $\e_v\neq 0$, that is, for every blue and green vertex $v$, see Figure~\ref{fig:strip} (right).

We let $G_{m}$ be the restriction of $G$ to $\TS_{m}$, thus $G_{m}$ is a graph in a strip with vertex set $\TS_{m}$ whose faces are all either lozenges or boundary triangles, see Figure~\ref{fig:grove_exc} (a). 
\begin{definition} \label{dfn:3n_m_grove}
A \emph{$(3n,m)$-grove} is a forest $F$ with vertex set $\TS_{m}$ satisfying the following conditions:
\begin{enumerate}
 \item $F$ is invariant under the shift by $3g$: if $(u,v)$ is an edge of $F$ then so is $(u+3g,v+3g)$;
 \item\label{item:all_edges} $F$ contains all edges $(v,v+e_{23})$ where $v$ is a red (i.e. $\e_v=0$) boundary vertex;
 \item for every lozenge face of $G_{m}$, $F$ contains exactly one of its two diagonals, and there are no other edges in $F$;
 \item\label{item:up_down} every connected component of $F$ contains a vertex $(0,j,k)$ and a vertex $(m,j',k')$ for some $j,j',k,k'\in\Z$.
\end{enumerate}
\end{definition}

For $v\in\TS_{m}$ and a $(3n,m)$-grove $F$, define $\deg_F(v)$ to be the number of edges of $F$ incident to $v$. Define the \emph{weight} of $F$ to be 
\[\wt(F)=\prod x_v^{\deg_F(v)-2},\] 
where the product is taken over all non-boundary vertices $v=(i,j,k)$ of $\TS_{m}$ satisfying $0\leq j<3n$, or equivalently, over all non-boundary vertices of the cylinder $\TC=\TS_m/3\Z g$.
Condition~\eqref{item:all_edges} together with the construction of $G_{m}$ implies that every connected component of $F$ is either \emph{green} (i.e. involves either only vertices $v$ with $\e_v=1$) or \emph{red-blue} (i.e. involves only vertices $v$ with $\e_v\neq 1$), see Figure~\ref{fig:grove_reg}. Consider any green connected component $C$ of $F$. Given such $C$, the unique green lower boundary vertex of $C$ is $u(C)=(0,j,-j)$ for some $j\equiv 2\pmod 3$, and there is a unique green upper boundary vertex $w(C)=(m,j',k')$. Such vertices exist by~\eqref{item:up_down} and are unique by~\eqref{item:up_down} as well since the green and blue connected components do not intersect each other. The possible values of $j'$ are 
\[j-2m,j-2m+3,\dots,j+m.\] 
We define 
\begin{equation}\label{eq:h_C}
h(C):=(j'-j+2m)/3\in\{0,1,\dots,m\}, 
\end{equation}
and it is clear that this number is the same for any green connected component of $F$. We define $h(F)$ to be equal to $h(C)$ where $C$ is any green connected component of $F$. Finally, for $s=0,1,\dots,m$, we define 
\begin{equation}\label{eq:J_r}
J_s:=\sum_{F:h(F)=s} \wt(F),
\end{equation}
where the sum is taken over all $(3n,m)$-groves $F$ with $h(F)=s$. As it is clear from Figure~\ref{fig:grove_exc} ((b) and (c)), for $s=0$ or $s=m$ there is only one grove $F$ with $h(F)=s$ and it satisfies $\wt(F)=1$, thus $J_0=J_m=1$. 

We are ready to state a formula for the recurrence coefficients when $v$ is adjacent to the boundary of $\TS_m$.
\begin{theorem}\label{thm:cube_formula}
 Fix any $n\geq 1$ and $m\geq 2$ and let $v=(m-1,j,k)\in\TS_{m}$. Then for any sufficiently large $\ell\equiv \e_v\pmod 3$ we have 
 \begin{equation}\label{eq:cube_formula}
 \sum_{s=0}^m (-1)^s J_s \Cubec_{v+(s+\ell) g}(\e_v+2(s+\ell)n)=0.
  \end{equation}
\end{theorem}

Note that we assumed $v=(m-1,j,k)\in\TS_{m}$ in Theorem~\ref{thm:cube_formula} while Theorem~\ref{thm:recurrence_1} holds for any $v\in\TS_m$. We give an analogous explicit formula for the recurrence coefficients for arbitrary $v\in\TS_m$. As we show in Section~\ref{sect:integrability}, both theorems can be deduced from our general results on \emph{cylindrical networks}~\cite{Networks} which we recall in Section~\ref{sect:networks}. 

To state a formula for general $v\in\TS_m$, we need to define a certain operation on polynomials as we did in~\cite[Section~4.1]{Networks}. Let $K$ be a field and consider a monic polynomial $Q(t)\in K[t]$ of degree $m$:
\[Q(t)=t^m-\alpha_1t^{m-1}+\dots+(-1)^m\alpha_m.\]
For each $1\leq r\leq m$, we would like to define a polynomial $Q^\plee{r}(t)$ of degree $m\choose r$. To do so, let us factor $Q(t)$ as a product of linear terms over the algebraic closure $\Kbar$ of $K$:
\[Q(t)=\prod_{i=1}^m (t-\gamma_i),\quad \gamma_1,\dots,\gamma_m\in\Kbar.\]
Then we define 
\begin{equation}\label{eq:plee}
Q^\plee{r}(t):=\prod_{1\leq i_1<i_2<\dots<i_r\leq m} (t-\gamma_{i_1}\gamma_{i_2}\dots\gamma_{i_r}).
\end{equation}
As we have shown in~\cite[Section~4.1]{Networks}, the polynomial $Q^\plee{r}(t)$ belongs to $K[t]$ rather than $\Kbar[t]$ since its coefficients are manifestly symmetric functions in $\gamma_1,\dots,\gamma_m$. 

Let us consider a specific polynomial $Q(t)$ defined by 
\begin{equation}\label{eq:Q_t}
Q(t)=J_mt^m-J_{m-1}t^{m-1}+\dots+(-1)^m J_0.
\end{equation}
In this case, $K$ is the field of rational functions in the variables $\x^\TC$. The polynomial $Q(t)$ is monic and has constant term $1$ since $J_0=J_m=1$ as we have already mentioned. We see that $Q(t)$ is the characteristic polynomial of the linear recurrence~\eqref{eq:cube_formula} for the vertex $v$ adjacent to the boundary of $\TS_m$.

\begin{theorem}\label{thm:pleth}
	Fix any $n\geq 1$ and $m\geq 2$ and let $v=(i,j,k)\in\TS_{m}$. Let $r=m-i$. Then the sequence $(\Cubec_{v+\ell g}(\e_v+2\ell n)))_{\ell\in\N}$ satisfies a linear recurrence with characteristic polynomial $Q^\plee{r}(t)$ for all sufficiently large values of $\ell$.
\end{theorem}

\begin{remark}
	Note that Theorem~\ref{thm:recurrence_1} considers the sequence of values of $\Cubec$ at a fixed vertex $v$ whereas in Theorems~\ref{thm:cube_formula} and~\ref{thm:pleth}, the vertex $v+\ell g$ depends on $\ell$. However, it is a standard fact that if a sequence $f(\ell)$ satisfies a linear recurrence then for any integer $c\geq 1$, the sequence $f(c\ell)$ satisfies a linear recurrence as well (and the characteristic polynomial of the latter is obtained from the characteristic polynomial of the former by raising all of its roots to the $c$-th power, see~\cite[Corollary~3.4.2]{GP2}). We can choose $c=3$ so that the vertex $v+3\ell g$ would always be equal to $v$ modulo $3\Z g$ which gives a direct formula for the recurrence coefficients mentioned in Theorem~\ref{thm:recurrence_1}. 
\end{remark}

\def\Torus{{\mathcal{T}(A,B)}}
\subsection{Cube recurrence in a torus}
\begin{definition}
	Fix two linearly independent vectors $A,B\in\TP_0$ such that $\e_A=\e_B=0$. The cube recurrence $\Cubetorus$ \emph{in a torus} $\Torus$ is a special case of the unbounded cube recurrence where the initial values are required to be invariant with respect to the shifts by $A$ and $B$:
	\[x_{u}=x_{u+A}=x_{u+B},\quad \forall\,u\in\TP_0.\]
\end{definition}

The following is a stronger version of the zero algebraic entropy property which is used in the discrete dynamical systems literature as a standard test for integrability, see e.g.~\cite{OTGR}:
\begin{definition}
We say that a sequence $f(0),f(1),\dots$ of Laurent polynomials in some variables $\x$ has \emph{zero algebraic entropy} if the degrees of $f(\ell)$'s grow at most polynomially.
\end{definition}

\begin{theorem}\label{thm:entropy}
	For any vertex $v\in\Torus$, the sequence $(\Cubetorus_v(\e_v+\ell))_{\ell\geq 0}$ has zero algebraic entropy. In fact, the degrees of these polynomials grow quadratically in $\ell$.
\end{theorem}
We explain how this result is a consequence of the groves formula~\cite{CS} in Section~\ref{sect:groves_triangle}.

\subsection{Examples}
In this section, we illustrate our main results by two examples.

\newcommand{\ccr}[6]{\scalebox{1}{$\begin{smallmatrix}
                      & & #1 & & \\
                      & #2&& #3 &\\
                      #4&& #5 && #6
                    \end{smallmatrix}$}}

\begin{table}
\scalebox{0.9}{
 \begin{tabu} to 1.1\textwidth{|c|X[c]|X[c]|X[c]|X[c]|}
 	\hline
 $t$ & $0,1,2$ & $3$& $4$& $5$ \\\hline
%   $\ccr{T_f(t)}{T_d(t+2)}{T_e(t+1)}{T_a(t+1)}{T_b(t)}{T_c(t+2)}$ &
  \ccr{\Cubet_f(t)}{\Cubet_d(t)}{\Cubet_e(t)}{\Cubet_a(t)}{\Cubet_b(t)}{\Cubet_c(t)} &
  \ccr{1}{1}{1}{1}{1}{3}& 
  \ccr{3}{\ast}{\ast}{\ast}{5}{\ast}& 
  \ccr{\ast}{\ast}{15}{7}{\ast}{\ast}& 
  \ccr{\ast}{41}{\ast}{\ast}{\ast}{7}
%   \ccr{\ast}{\ast}{\ast}{\ast}{\ast}{\ast}
   \\\hline
 \end{tabu}}

 \scalebox{0.9}{
 \begin{tabu} to 1.1\textwidth{|c|X[c]|X[c]|X[c]|X[c]|X[c]|}\hline
  $t$ & $6$& $7$& $8$& $9$& $10,11,12$\\\hline
  \ccr{\Cubet_f(t)}{\Cubet_d(t)}{\Cubet_e(t)}{\Cubet_a(t)}{\Cubet_b(t)}{\Cubet_c(t)} &
  \ccr{19}{\ast}{\ast}{\ast}{21}{\ast}& 
  \ccr{\ast}{\ast}{13}{9}{\ast}{\ast}& 
  \ccr{\ast}{5}{\ast}{\ast}{\ast}{5}& 
  \ccr{1}{\ast}{\ast}{\ast}{3}{\ast}& 
%   \ccr{\ast}{\ast}{1}{1}{\ast}{\ast}& 
%   \ccr{\ast}{1}{\ast}{\ast}{\ast}{1}& 
%   \ccr{3}{\ast}{\ast}{\ast}{1}{\ast}&
   \ccr{3}{1}{1}{1}{1}{1}\\\hline
 \end{tabu}}
\caption{\label{table:cube_evolution} The evolution of the cube recurrence in $\TR_5$. When $f_v(t)$ is undefined (i.e. when $t\not\equiv \e_v\pmod 3$), it is denoted by $\ast$.}
\end{table}

\begin{example}
Consider Theorem~\ref{thm:cube_periodic} for $m=5$. Suppose that we set $\Cubet_c(\e_c)=3$ and $\Cubet_v(\e_v)=1$ for $v=a,b,d,e,f$. Then the values of $\Cubet_v(t)$ for $t=0,1,\dots,12$ are shown in Table~\ref{table:cube_evolution}. For example, $\Cubet_e(7)=\frac{\Cubet_f(6)\Cubet_c(5)+\Cubet_b(6)+\Cubet_d(5)}{\Cubet_e(4)}=\frac{19\times 7+21+41}{15}=13$. Just as Theorem~\ref{thm:cube_periodic} suggests, increasing $t$ by $10$ corresponds to rotating the triangle counterclockwise.
\end{example}

\begin{figure}
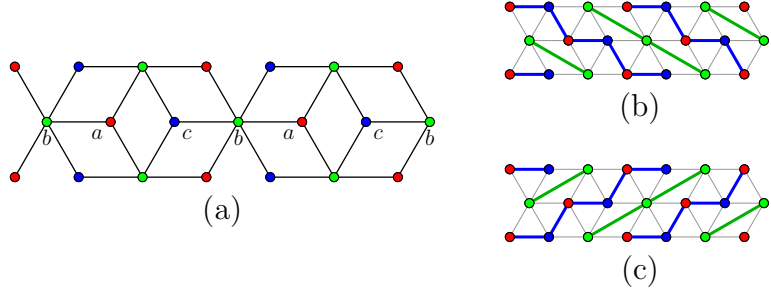

 
\makebox[\textwidth]{


}
\caption{\label{fig:grove_exc} The graph $G_2$ (a). The unique $(3,2)$-groves with $h=0$ (b) and $h=2$ (c).}
\end{figure}

\begin{figure}
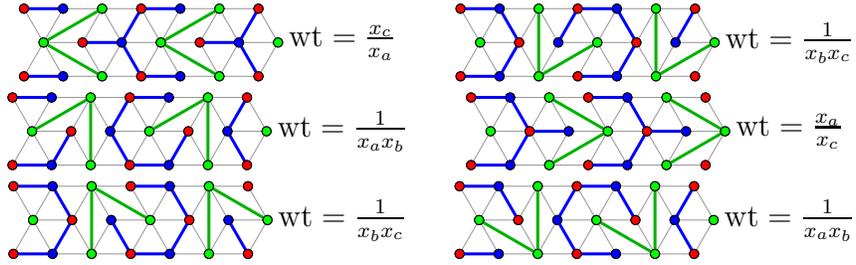


\makebox[\textwidth]{
%

}
\caption{\label{fig:grove_reg} The six $(3,2)$-groves $F$ satisfying $h(F)=1$ given together with their weights.}
\end{figure}

\begin{example}	
Consider Theorem~\ref{thm:cube_formula} for $n=1,m=2$. In this case there are only three non-boundary vertices $a,b,c$ in $\TC=\TS_m/3\Z g$ and the sequence
\[(y_\ell)_{\ell\in\N}=(\Cubec_a(0),\Cubec_b(1),\Cubec_c(2),\Cubec_a(3),\dots)\] 
is defined by
\[y_0=x_a,\quad y_1=x_b,\quad y_2=x_c,\quad y_{\ell+3}=\frac{y_{\ell+2}y_{\ell+1}+2}{y_\ell}\]
for all $\ell\in\N$. 

For $n=1,m=2$, we have $J_0=J_2=1$, and all the six groves with $h(F)=1$ are shown in Figure~\ref{fig:grove_reg} which implies that 
\[J_1=\frac{x_c}{x_a}+\frac{x_a}{x_c}+\frac2{x_bx_c}+\frac2{x_ax_b}.\] 
Thus~\eqref{eq:cube_formula} applied to this case states that 
the sequence $(y_\ell)_{\ell\in\Z}$ satisfies a linear recurrence 
\begin{equation}\label{eq:recurrence_example}
 y_{\ell+4}-\left(\frac{x_c}{x_a}+\frac{x_a}{x_c}+\frac2{x_bx_c}+\frac2{x_ax_b}\right) y_{\ell+2}+y_\ell=0 
\end{equation}
for any sufficiently large $\ell\in\N$. 

Let us plug in $x_a=x_b=x_c=1$. Then the first few values of $(y_0,y_1,\dots)$ are $(1,1,1,3,5,17,29,99,169\dots)$, this is the sequence \href{http://oeis.org/A079496}{A079496} in the OEIS~\cite{OEIS}. According to~\eqref{eq:recurrence_example}, we should get
\[y_{\ell+4}-6y_{\ell+2}+y_{\ell}=0,\] 
which is indeed true, for example, $99-6\times 17+3=0$.
\end{example}

\section{Periodicity}\label{sect:periodicity}
\subsection{Background on the multidimensional cube recurrence}
We recall some results and definitions of Henriques and Speyer~\cite{HS}. They work in the context of \emph{zonotopal tilings}, but we will translate their results into the dual language of \emph{pseudoline arrangements}. 

Fix a disc $D$ in $\R^2$ and an integer $n$. Let $q_1,q_2,\dots,q_n,q_1',q_2',\dots,q_n'$ be $2n$ marked points on the boundary of $D$ in clockwise order. A \emph{pseudoline labeled by $k$} is a piecewise-smooth embedding $p:[0,1]\to D$ such that the intersection of the image of $p$ with the boundary of $D$ consists of two points $p(0)=q_k$ and $p(1)=q_k'$ for some $k$. In other words, one may view a pseudoline as a simple closed curve in $\RP^2$. A \emph{pseudoline arrangement} $\Acal=(p_1,p_2,\dots,p_n)$ is a collection of $n$ pseudolines where $p_k$ is labeled by $k$ and such that any two pseudolines intersect exactly once. An example of a pseudoline arrangement is given in Figure~\ref{fig:pseudo_ex}. We view (labeled) pseudoline arrangements up to orientation-preserving diffeomorphisms of $D$. 

\begin{figure}

\begin{tikzpicture}[scale=0.7]
	\draw[opacity=0.5] (0,0) circle (5);
	\node[anchor=180] (q1) at (0:5) {$q_6$};
	\node[anchor=210] (q2) at (30:5) {$q_5$};
	\node[anchor=240] (q3) at (60:5) {$q_4$};
	\node[anchor=270] (q4) at (90:5) {$q_3$};
	\node[anchor=300] (q5) at (120:5) {$q_2$};
	\node[anchor=330] (q6) at (150:5) {$q_1$};
	\node[anchor=0] (qq1) at (180:5) {$q_6'$};
	\node[anchor=30] (qq2) at (210:5) {$q_5'$};
	\node[anchor=60] (qq3) at (240:5) {$q_4'$};
	\node[anchor=90] (qq4) at (270:5) {$q_3'$};
	\node[anchor=120] (qq5) at (300:5) {$q_2'$};
	\node[anchor=150] (qq6) at (330:5) {$q_1'$};
	\draw[line width=0.25mm] plot [smooth,tension=0.7] coordinates {(0:5) (1,1) (-2,-1) (180:5)};
	
	\draw[line width=0.25mm] plot [smooth,tension=0.7] coordinates {(30:5) (4,-1.5) (0,-3) (210:5)};
	\draw[line width=0.25mm] plot [smooth,tension=1] coordinates {(60:5) (2,2) (0.6,-1) (240:5)};
	\draw[line width=0.25mm] plot [smooth,tension=0.7] coordinates {(90:5) (-1,1) (-1,-2) (270:5)};
	\draw[line width=0.25mm] plot [smooth,tension=0.7] coordinates {(120:5) (0.5, 0.5) (300:5)};
	\draw[line width=0.25mm] plot [smooth,tension=1] coordinates {(150:5) (-1,-1) (330:5)};
\end{tikzpicture}

	\caption{\label{fig:pseudo_ex}An arrangement of six pseudolines.}
\end{figure}
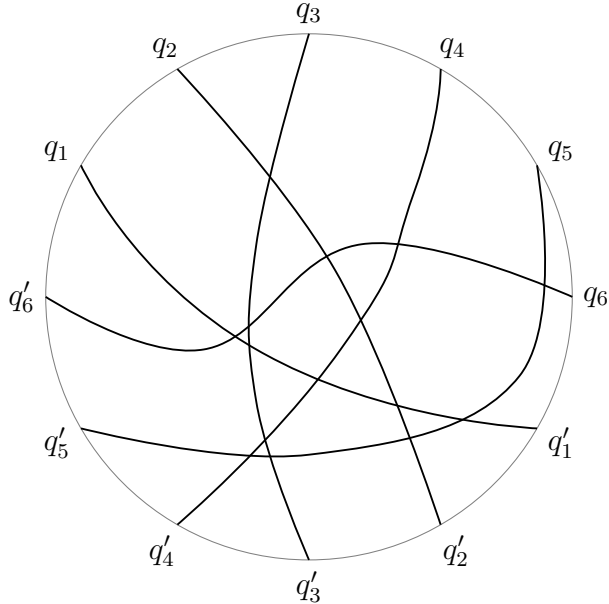

Every pseudoline arrangement subdivides $D$ into a collection of regions. We call a region \emph{unbounded} if it is adjacent to the boundary of $D$, and \emph{bounded} otherwise. We say that a bounded region $R$ is \emph{triangular}, or \emph{simplicial} if it is adjacent to exactly three pseudolines.

Consider a pseudoline arrangement $\Acal$ and a triangular region $R\subset D$ of $\Acal$. The \emph{mutation} of $\Acal$ at $R$ is another pseudoline arrangement $\Acal'$ that is obtained from $\Acal$ by replacing the small neighborhood of $R$ as shown in Figure~\ref{fig:mutation}. More precisely, the small neighborhood of $R$ is diffeomorphic to a disc $D_0$ and the three pseudolines adjacent to $R$ form a pseudoline arrangement $\Acal_0$ inside $D_0$. There are exactly two distinct pseudoline arrangements with three pseudolines in $D_0$, denote them $\Acal_0$ and $\Acal_0'$. The mutation operation replaces $\Acal_0$ by $\Acal_0'$ and does not change the rest of $\Acal$.

\begin{figure}
	
\begin{tikzpicture}
	\begin{scope}[xshift=-3cm,scale=0.7]
	\draw[opacity=0.5,dashed] (0,0) circle (3);
	\draw[line width=0.25mm] plot [smooth,tension=0.7] coordinates {(0:3) (90:1) (180:3)};
	\draw[line width=0.25mm] plot [smooth,tension=0.7] coordinates {(120:3) (210:1) (300:3)};
	\draw[line width=0.25mm] plot [smooth,tension=0.7] coordinates {(240:3) (330:1) (60:3)};
	\node[scale=1] (R1) at (30:2.3) {$R_1$};
	\node[scale=1] (R2) at (90:2.3) {$R_2$};
	\node[scale=1] (R3) at (150:2.3) {$R_3$};
	\node[scale=1] (R4) at (210:2.3) {$R_4$};
	\node[scale=1] (R5) at (270:2.3) {$R_5$};
	\node[scale=1] (R6) at (330:2.3) {$R_6$};
	\node[scale=1] (Ri) at (0,0) {$R_i$};	
	\end{scope}
	\node (math) at (0,0) {$\longleftrightarrow$};
	\begin{scope}[xshift=3cm,scale=0.7]
	\draw[opacity=0.5,dashed] (0,0) circle (3);
	\draw[line width=0.25mm] plot [smooth,tension=0.7] coordinates {(0:3) (-90:1) (180:3)};
	\draw[line width=0.25mm] plot [smooth,tension=0.7] coordinates {(120:3) (30:1) (300:3)};
	\draw[line width=0.25mm] plot [smooth,tension=0.7] coordinates {(240:3) (150:1) (60:3)};
	\node[scale=1] (R1) at (30:2.3) {$R_1$};
	\node[scale=1] (R2) at (90:2.3) {$R_2$};
	\node[scale=1] (R3) at (150:2.3) {$R_3$};
	\node[scale=1] (R4) at (210:2.3) {$R_4$};
	\node[scale=1] (R5) at (270:2.3) {$R_5$};
	\node[scale=1] (R6) at (330:2.3) {$R_6$};
	\node[scale=1] (Ri) at (0,0) {$R_i'$};
	\end{scope}
\end{tikzpicture}

\caption{\label{fig:mutation} A mutation of a pseudoline arrangement at a simplicial region $R_i$.}
\end{figure}
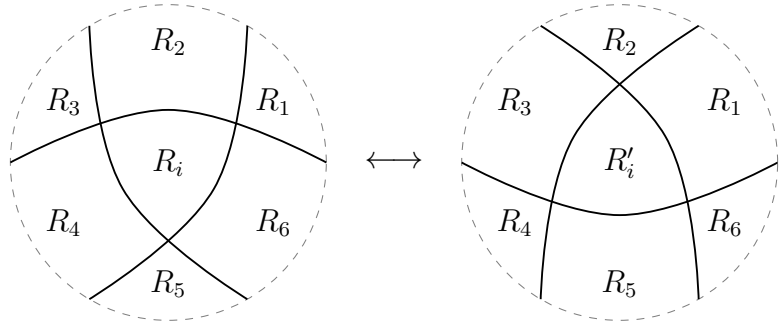

The \emph{multidimensional cube recurrence} is a certain way of assigning rational functions to all regions of all pseudoline arrangements with $n$ pseudolines. Start with some pseudoline arrangement $\Acal$ with $n$ pseudolines and let $R_1,\dots,R_N$ be its bounded regions. For $1\leq i\leq N$, assign a variable $x_i$ to the region $R_i$. Assign $1$ to every unbounded region $R$ of $\Acal$. Now, let $\Acal'$ be obtained from $\Acal$ by a mutation of a region $R_i$ into $R_i'$, and let $R_1,R_2,R_3,R_4,R_5,R_6$ be the regions adjacent to $R_i$ in counterclockwise order, see Figure~\ref{fig:mutation}. Then we assign a rational function $x_i'$ to $R_i'$ according to the following rule:
\[x_i'=\frac{x_1x_4+x_2x_5+x_3x_6}{x_i}.\]
Thus, for any sequence 
\[\Acal=\Acal^\parr 0,\Acal^\parr1,\dots,\Acal^\parr m=\Acal'\]
of pseudoline arrangements such that $\Acal^\parr i$ and $\Acal^\parr {i-1}$ are connected by a mutation for any $i=1,\dots, m$, this procedure defines an assignment of rational functions in $x_1,\dots,x_N$ to the regions of $\Acal'$. 
\begin{theorem}[\cite{HS}]\label{thm:HS}
 For each $\Acal'$ and each of its bounded regions $R$, the rational function in $x_1,\dots,x_N$ assigned to $R$ via the above procedure is actually a Laurent polynomial, and it does not depend on the mutation sequence that connects $\Acal$ with $\Acal'$. 
\end{theorem}

\subsection{The proof of Theorem~\ref{thm:cube_periodic}}
Theorem~\ref{thm:HS} provides a clear strategy to prove Theorem~\ref{thm:cube_periodic}: we will assign a pseudoline arrangement $\Acal^\parr t$ to every $t\in \Z$ so that the arrangements $\Acal^\parr t$ and $\Acal^\parr {t+1}$ will be related by a sequence of mutations, and so that the non-boundary vertices of $\TR_m$ will correspond to the bounded regions of $\Acal^\parr t$.

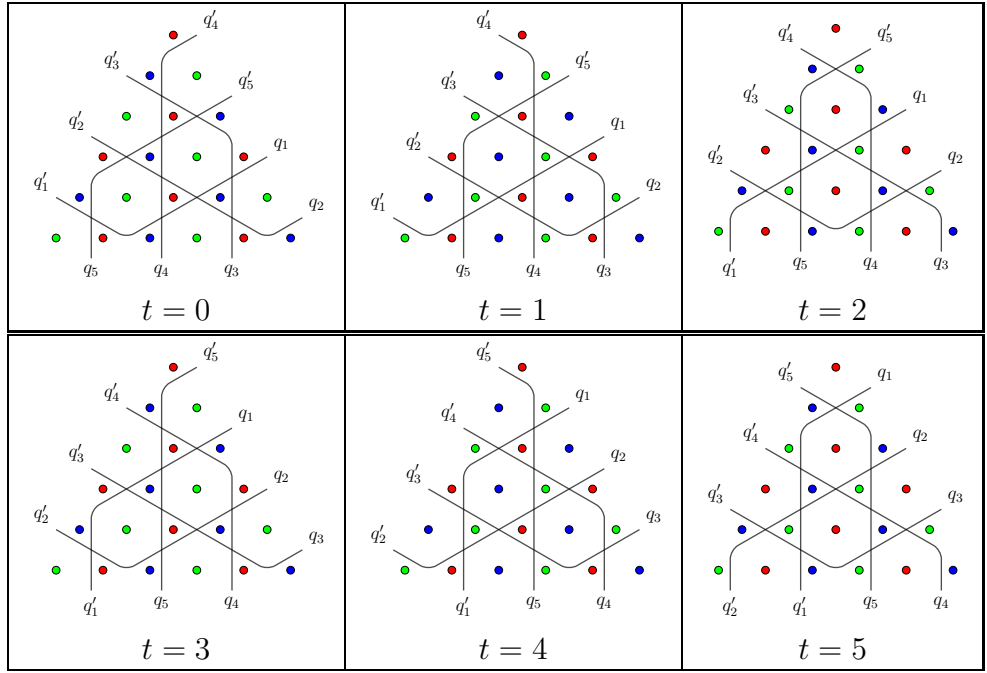
\begin{figure}
% \centering

\begin{tabular}{|c|c|c|}\hline
\scalebox{0.6}{
\begin{tikzpicture}[scale=0.6]
\coordinate (L1) at (90:5.00);
\coordinate (L2) at (330:5.00);
\coordinate (L3) at (210:5.00);
\coordinate (C005) at (barycentric cs:L1=0.0,L2=0.0,L3=5.0);
\coordinate (C014) at (barycentric cs:L1=0.0,L2=1.0,L3=4.0);
\coordinate (C023) at (barycentric cs:L1=0.0,L2=2.0,L3=3.0);
\coordinate (C032) at (barycentric cs:L1=0.0,L2=3.0,L3=2.0);
\coordinate (C041) at (barycentric cs:L1=0.0,L2=4.0,L3=1.0);
\coordinate (C050) at (barycentric cs:L1=0.0,L2=5.0,L3=0.0);
\coordinate (C104) at (barycentric cs:L1=1.0,L2=0.0,L3=4.0);
\coordinate (C113) at (barycentric cs:L1=1.0,L2=1.0,L3=3.0);
\coordinate (C122) at (barycentric cs:L1=1.0,L2=2.0,L3=2.0);
\coordinate (C131) at (barycentric cs:L1=1.0,L2=3.0,L3=1.0);
\coordinate (C140) at (barycentric cs:L1=1.0,L2=4.0,L3=0.0);
\coordinate (C203) at (barycentric cs:L1=2.0,L2=0.0,L3=3.0);
\coordinate (C212) at (barycentric cs:L1=2.0,L2=1.0,L3=2.0);
\coordinate (C221) at (barycentric cs:L1=2.0,L2=2.0,L3=1.0);
\coordinate (C230) at (barycentric cs:L1=2.0,L2=3.0,L3=0.0);
\coordinate (C302) at (barycentric cs:L1=3.0,L2=0.0,L3=2.0);
\coordinate (C311) at (barycentric cs:L1=3.0,L2=1.0,L3=1.0);
\coordinate (C320) at (barycentric cs:L1=3.0,L2=2.0,L3=0.0);
\coordinate (C401) at (barycentric cs:L1=4.0,L2=0.0,L3=1.0);
\coordinate (C410) at (barycentric cs:L1=4.0,L2=1.0,L3=0.0);
\coordinate (C500) at (barycentric cs:L1=5.0,L2=0.0,L3=0.0);
\draw[fill=green] (barycentric cs:L1=0.0,L2=0.0,L3=5.0) circle (0.14285714285714288);
\draw[fill=red] (barycentric cs:L1=0.0,L2=1.0,L3=4.0) circle (0.14285714285714288);
\draw[fill=blue] (barycentric cs:L1=0.0,L2=2.0,L3=3.0) circle (0.14285714285714288);
\draw[fill=green] (barycentric cs:L1=0.0,L2=3.0,L3=2.0) circle (0.14285714285714288);
\draw[fill=red] (barycentric cs:L1=0.0,L2=4.0,L3=1.0) circle (0.14285714285714288);
\draw[fill=blue] (barycentric cs:L1=0.0,L2=5.0,L3=0.0) circle (0.14285714285714288);
\draw[fill=blue] (barycentric cs:L1=1.0,L2=0.0,L3=4.0) circle (0.14285714285714288);
\draw[fill=green] (barycentric cs:L1=1.0,L2=1.0,L3=3.0) circle (0.14285714285714288);
\draw[fill=red] (barycentric cs:L1=1.0,L2=2.0,L3=2.0) circle (0.14285714285714288);
\draw[fill=blue] (barycentric cs:L1=1.0,L2=3.0,L3=1.0) circle (0.14285714285714288);
\draw[fill=green] (barycentric cs:L1=1.0,L2=4.0,L3=0.0) circle (0.14285714285714288);
\draw[fill=red] (barycentric cs:L1=2.0,L2=0.0,L3=3.0) circle (0.14285714285714288);
\draw[fill=blue] (barycentric cs:L1=2.0,L2=1.0,L3=2.0) circle (0.14285714285714288);
\draw[fill=green] (barycentric cs:L1=2.0,L2=2.0,L3=1.0) circle (0.14285714285714288);
\draw[fill=red] (barycentric cs:L1=2.0,L2=3.0,L3=0.0) circle (0.14285714285714288);
\draw[fill=green] (barycentric cs:L1=3.0,L2=0.0,L3=2.0) circle (0.14285714285714288);
\draw[fill=red] (barycentric cs:L1=3.0,L2=1.0,L3=1.0) circle (0.14285714285714288);
\draw[fill=blue] (barycentric cs:L1=3.0,L2=2.0,L3=0.0) circle (0.14285714285714288);
\draw[fill=blue] (barycentric cs:L1=4.0,L2=0.0,L3=1.0) circle (0.14285714285714288);
\draw[fill=green] (barycentric cs:L1=4.0,L2=1.0,L3=0.0) circle (0.14285714285714288);
\draw[fill=red] (barycentric cs:L1=5.0,L2=0.0,L3=0.0) circle (0.14285714285714288);
\draw[opacity=0.7, line width=0.3mm, rounded corners=0.2cm] (barycentric cs:L1=-0.5,L2=2.5,L3=3.0)--(barycentric cs:L1=4.5,L2=0.0,L3=0.5)--(barycentric cs:L1=5.0,L2=0.5,L3=-0.5);
\node[anchor=north] (nodeA4B0) at (barycentric cs:L1=-0.5,L2=2.5,L3=3.0) {$q_4$};
\node[anchor=south west] (nodeA4B0) at (barycentric cs:L1=5.0,L2=0.5,L3=-0.5) {$q'_4$};
\draw[opacity=0.7, line width=0.3mm, rounded corners=0.2cm] (barycentric cs:L1=0.5,L2=5.0,L3=-0.5)--(barycentric cs:L1=0.0,L2=4.5,L3=0.5)--(barycentric cs:L1=2.5,L2=-0.5,L3=3.0);
\node[anchor=south west] (nodeA4B0) at (barycentric cs:L1=0.5,L2=5.0,L3=-0.5) {$q_2$};
\node[anchor=south east] (nodeA4B0) at (barycentric cs:L1=2.5,L2=-0.5,L3=3.0) {$q'_2$};
\draw[opacity=0.7, line width=0.3mm, rounded corners=0.2cm] (barycentric cs:L1=-0.5,L2=4.0,L3=1.5)--(barycentric cs:L1=2.5,L2=2.5,L3=0.0)--(barycentric cs:L1=4.0,L2=-0.5,L3=1.5);
\node[anchor=north] (nodeA2B2) at (barycentric cs:L1=-0.5,L2=4.0,L3=1.5) {$q_3$};
\node[anchor=south east] (nodeA2B2) at (barycentric cs:L1=4.0,L2=-0.5,L3=1.5) {$q'_3$};
\draw[opacity=0.7, line width=0.3mm, rounded corners=0.2cm] (barycentric cs:L1=-0.5,L2=1.0,L3=4.5)--(barycentric cs:L1=1.5,L2=0.0,L3=3.5)--(barycentric cs:L1=3.5,L2=2.0,L3=-0.5);
\node[anchor=north] (nodeA1B3) at (barycentric cs:L1=-0.5,L2=1.0,L3=4.5) {$q_5$};
\node[anchor=south west] (nodeA1B3) at (barycentric cs:L1=3.5,L2=2.0,L3=-0.5) {$q'_5$};
\draw[opacity=0.7, line width=0.3mm, rounded corners=0.2cm] (barycentric cs:L1=2.0,L2=3.5,L3=-0.5)--(barycentric cs:L1=0.0,L2=1.5,L3=3.5)--(barycentric cs:L1=1.0,L2=-0.5,L3=4.5);
\node[anchor=south west] (nodeA1B3) at (barycentric cs:L1=2.0,L2=3.5,L3=-0.5) {$q_1$};
\node[anchor=south east] (nodeA1B3) at (barycentric cs:L1=1.0,L2=-0.5,L3=4.5) {$q'_1$};
\end{tikzpicture}}
&
\scalebox{0.6}{
\begin{tikzpicture}[scale=0.6]
\coordinate (L1) at (90:5.00);
\coordinate (L2) at (330:5.00);
\coordinate (L3) at (210:5.00);
\coordinate (C005) at (barycentric cs:L1=0.0,L2=0.0,L3=5.0);
\coordinate (C014) at (barycentric cs:L1=0.0,L2=1.0,L3=4.0);
\coordinate (C023) at (barycentric cs:L1=0.0,L2=2.0,L3=3.0);
\coordinate (C032) at (barycentric cs:L1=0.0,L2=3.0,L3=2.0);
\coordinate (C041) at (barycentric cs:L1=0.0,L2=4.0,L3=1.0);
\coordinate (C050) at (barycentric cs:L1=0.0,L2=5.0,L3=0.0);
\coordinate (C104) at (barycentric cs:L1=1.0,L2=0.0,L3=4.0);
\coordinate (C113) at (barycentric cs:L1=1.0,L2=1.0,L3=3.0);
\coordinate (C122) at (barycentric cs:L1=1.0,L2=2.0,L3=2.0);
\coordinate (C131) at (barycentric cs:L1=1.0,L2=3.0,L3=1.0);
\coordinate (C140) at (barycentric cs:L1=1.0,L2=4.0,L3=0.0);
\coordinate (C203) at (barycentric cs:L1=2.0,L2=0.0,L3=3.0);
\coordinate (C212) at (barycentric cs:L1=2.0,L2=1.0,L3=2.0);
\coordinate (C221) at (barycentric cs:L1=2.0,L2=2.0,L3=1.0);
\coordinate (C230) at (barycentric cs:L1=2.0,L2=3.0,L3=0.0);
\coordinate (C302) at (barycentric cs:L1=3.0,L2=0.0,L3=2.0);
\coordinate (C311) at (barycentric cs:L1=3.0,L2=1.0,L3=1.0);
\coordinate (C320) at (barycentric cs:L1=3.0,L2=2.0,L3=0.0);
\coordinate (C401) at (barycentric cs:L1=4.0,L2=0.0,L3=1.0);
\coordinate (C410) at (barycentric cs:L1=4.0,L2=1.0,L3=0.0);
\coordinate (C500) at (barycentric cs:L1=5.0,L2=0.0,L3=0.0);
\draw[fill=green] (barycentric cs:L1=0.0,L2=0.0,L3=5.0) circle (0.14285714285714288);
\draw[fill=red] (barycentric cs:L1=0.0,L2=1.0,L3=4.0) circle (0.14285714285714288);
\draw[fill=blue] (barycentric cs:L1=0.0,L2=2.0,L3=3.0) circle (0.14285714285714288);
\draw[fill=green] (barycentric cs:L1=0.0,L2=3.0,L3=2.0) circle (0.14285714285714288);
\draw[fill=red] (barycentric cs:L1=0.0,L2=4.0,L3=1.0) circle (0.14285714285714288);
\draw[fill=blue] (barycentric cs:L1=0.0,L2=5.0,L3=0.0) circle (0.14285714285714288);
\draw[fill=blue] (barycentric cs:L1=1.0,L2=0.0,L3=4.0) circle (0.14285714285714288);
\draw[fill=green] (barycentric cs:L1=1.0,L2=1.0,L3=3.0) circle (0.14285714285714288);
\draw[fill=red] (barycentric cs:L1=1.0,L2=2.0,L3=2.0) circle (0.14285714285714288);
\draw[fill=blue] (barycentric cs:L1=1.0,L2=3.0,L3=1.0) circle (0.14285714285714288);
\draw[fill=green] (barycentric cs:L1=1.0,L2=4.0,L3=0.0) circle (0.14285714285714288);
\draw[fill=red] (barycentric cs:L1=2.0,L2=0.0,L3=3.0) circle (0.14285714285714288);
\draw[fill=blue] (barycentric cs:L1=2.0,L2=1.0,L3=2.0) circle (0.14285714285714288);
\draw[fill=green] (barycentric cs:L1=2.0,L2=2.0,L3=1.0) circle (0.14285714285714288);
\draw[fill=red] (barycentric cs:L1=2.0,L2=3.0,L3=0.0) circle (0.14285714285714288);
\draw[fill=green] (barycentric cs:L1=3.0,L2=0.0,L3=2.0) circle (0.14285714285714288);
\draw[fill=red] (barycentric cs:L1=3.0,L2=1.0,L3=1.0) circle (0.14285714285714288);
\draw[fill=blue] (barycentric cs:L1=3.0,L2=2.0,L3=0.0) circle (0.14285714285714288);
\draw[fill=blue] (barycentric cs:L1=4.0,L2=0.0,L3=1.0) circle (0.14285714285714288);
\draw[fill=green] (barycentric cs:L1=4.0,L2=1.0,L3=0.0) circle (0.14285714285714288);
\draw[fill=red] (barycentric cs:L1=5.0,L2=0.0,L3=0.0) circle (0.14285714285714288);
\draw[opacity=0.7, line width=0.3mm, rounded corners=0.2cm] (barycentric cs:L1=-0.5,L2=3.0,L3=2.5)--(barycentric cs:L1=4.5,L2=0.5,L3=0.0)--(barycentric cs:L1=5.0,L2=-0.5,L3=0.5);
\node[anchor=north] (nodeA4B0) at (barycentric cs:L1=-0.5,L2=3.0,L3=2.5) {$q_4$};
\node[anchor=south east] (nodeA4B0) at (barycentric cs:L1=5.0,L2=-0.5,L3=0.5) {$q'_4$};
\draw[opacity=0.7, line width=0.3mm, rounded corners=0.2cm] (barycentric cs:L1=1.0,L2=4.5,L3=-0.5)--(barycentric cs:L1=0.0,L2=3.5,L3=1.5)--(barycentric cs:L1=2.0,L2=-0.5,L3=3.5);
\node[anchor=south west] (nodeA3B1) at (barycentric cs:L1=1.0,L2=4.5,L3=-0.5) {$q_2$};
\node[anchor=south east] (nodeA3B1) at (barycentric cs:L1=2.0,L2=-0.5,L3=3.5) {$q'_2$};
\draw[opacity=0.7, line width=0.3mm, rounded corners=0.2cm] (barycentric cs:L1=-0.5,L2=1.5,L3=4.0)--(barycentric cs:L1=2.5,L2=0.0,L3=2.5)--(barycentric cs:L1=4.0,L2=1.5,L3=-0.5);
\node[anchor=north] (nodeA2B2) at (barycentric cs:L1=-0.5,L2=1.5,L3=4.0) {$q_5$};
\node[anchor=south west] (nodeA2B2) at (barycentric cs:L1=4.0,L2=1.5,L3=-0.5) {$q'_5$};
\draw[opacity=0.7, line width=0.3mm, rounded corners=0.2cm] (barycentric cs:L1=-0.5,L2=4.5,L3=1.0)--(barycentric cs:L1=1.5,L2=3.5,L3=0.0)--(barycentric cs:L1=3.5,L2=-0.5,L3=2.0);
\node[anchor=north] (nodeA1B3) at (barycentric cs:L1=-0.5,L2=4.5,L3=1.0) {$q_3$};
\node[anchor=south east] (nodeA1B3) at (barycentric cs:L1=3.5,L2=-0.5,L3=2.0) {$q'_3$};
\draw[opacity=0.7, line width=0.3mm, rounded corners=0.2cm] (barycentric cs:L1=2.5,L2=3.0,L3=-0.5)--(barycentric cs:L1=0.0,L2=0.5,L3=4.5)--(barycentric cs:L1=0.5,L2=-0.5,L3=5.0);
\node[anchor=south west] (nodeA0B4) at (barycentric cs:L1=2.5,L2=3.0,L3=-0.5) {$q_1$};
\node[anchor=south east] (nodeA0B4) at (barycentric cs:L1=0.5,L2=-0.5,L3=5.0) {$q'_1$};
\end{tikzpicture}}
&
\scalebox{0.6}{
\begin{tikzpicture}[scale=0.6]
\coordinate (L1) at (90:5.00);
\coordinate (L2) at (330:5.00);
\coordinate (L3) at (210:5.00);
\coordinate (C005) at (barycentric cs:L1=0.0,L2=0.0,L3=5.0);
\coordinate (C014) at (barycentric cs:L1=0.0,L2=1.0,L3=4.0);
\coordinate (C023) at (barycentric cs:L1=0.0,L2=2.0,L3=3.0);
\coordinate (C032) at (barycentric cs:L1=0.0,L2=3.0,L3=2.0);
\coordinate (C041) at (barycentric cs:L1=0.0,L2=4.0,L3=1.0);
\coordinate (C050) at (barycentric cs:L1=0.0,L2=5.0,L3=0.0);
\coordinate (C104) at (barycentric cs:L1=1.0,L2=0.0,L3=4.0);
\coordinate (C113) at (barycentric cs:L1=1.0,L2=1.0,L3=3.0);
\coordinate (C122) at (barycentric cs:L1=1.0,L2=2.0,L3=2.0);
\coordinate (C131) at (barycentric cs:L1=1.0,L2=3.0,L3=1.0);
\coordinate (C140) at (barycentric cs:L1=1.0,L2=4.0,L3=0.0);
\coordinate (C203) at (barycentric cs:L1=2.0,L2=0.0,L3=3.0);
\coordinate (C212) at (barycentric cs:L1=2.0,L2=1.0,L3=2.0);
\coordinate (C221) at (barycentric cs:L1=2.0,L2=2.0,L3=1.0);
\coordinate (C230) at (barycentric cs:L1=2.0,L2=3.0,L3=0.0);
\coordinate (C302) at (barycentric cs:L1=3.0,L2=0.0,L3=2.0);
\coordinate (C311) at (barycentric cs:L1=3.0,L2=1.0,L3=1.0);
\coordinate (C320) at (barycentric cs:L1=3.0,L2=2.0,L3=0.0);
\coordinate (C401) at (barycentric cs:L1=4.0,L2=0.0,L3=1.0);
\coordinate (C410) at (barycentric cs:L1=4.0,L2=1.0,L3=0.0);
\coordinate (C500) at (barycentric cs:L1=5.0,L2=0.0,L3=0.0);
\draw[fill=green] (barycentric cs:L1=0.0,L2=0.0,L3=5.0) circle (0.14285714285714288);
\draw[fill=red] (barycentric cs:L1=0.0,L2=1.0,L3=4.0) circle (0.14285714285714288);
\draw[fill=blue] (barycentric cs:L1=0.0,L2=2.0,L3=3.0) circle (0.14285714285714288);
\draw[fill=green] (barycentric cs:L1=0.0,L2=3.0,L3=2.0) circle (0.14285714285714288);
\draw[fill=red] (barycentric cs:L1=0.0,L2=4.0,L3=1.0) circle (0.14285714285714288);
\draw[fill=blue] (barycentric cs:L1=0.0,L2=5.0,L3=0.0) circle (0.14285714285714288);
\draw[fill=blue] (barycentric cs:L1=1.0,L2=0.0,L3=4.0) circle (0.14285714285714288);
\draw[fill=green] (barycentric cs:L1=1.0,L2=1.0,L3=3.0) circle (0.14285714285714288);
\draw[fill=red] (barycentric cs:L1=1.0,L2=2.0,L3=2.0) circle (0.14285714285714288);
\draw[fill=blue] (barycentric cs:L1=1.0,L2=3.0,L3=1.0) circle (0.14285714285714288);
\draw[fill=green] (barycentric cs:L1=1.0,L2=4.0,L3=0.0) circle (0.14285714285714288);
\draw[fill=red] (barycentric cs:L1=2.0,L2=0.0,L3=3.0) circle (0.14285714285714288);
\draw[fill=blue] (barycentric cs:L1=2.0,L2=1.0,L3=2.0) circle (0.14285714285714288);
\draw[fill=green] (barycentric cs:L1=2.0,L2=2.0,L3=1.0) circle (0.14285714285714288);
\draw[fill=red] (barycentric cs:L1=2.0,L2=3.0,L3=0.0) circle (0.14285714285714288);
\draw[fill=green] (barycentric cs:L1=3.0,L2=0.0,L3=2.0) circle (0.14285714285714288);
\draw[fill=red] (barycentric cs:L1=3.0,L2=1.0,L3=1.0) circle (0.14285714285714288);
\draw[fill=blue] (barycentric cs:L1=3.0,L2=2.0,L3=0.0) circle (0.14285714285714288);
\draw[fill=blue] (barycentric cs:L1=4.0,L2=0.0,L3=1.0) circle (0.14285714285714288);
\draw[fill=green] (barycentric cs:L1=4.0,L2=1.0,L3=0.0) circle (0.14285714285714288);
\draw[fill=red] (barycentric cs:L1=5.0,L2=0.0,L3=0.0) circle (0.14285714285714288);
\draw[opacity=0.7, line width=0.3mm, rounded corners=0.2cm] (barycentric cs:L1=-0.5,L2=3.5,L3=2.0)--(barycentric cs:L1=3.5,L2=1.5,L3=0.0)--(barycentric cs:L1=4.5,L2=-0.5,L3=1.0);
\node[anchor=north] (nodeA3B1) at (barycentric cs:L1=-0.5,L2=3.5,L3=2.0) {$q_4$};
\node[anchor=south east] (nodeA3B1) at (barycentric cs:L1=4.5,L2=-0.5,L3=1.0) {$q'_4$};
\draw[opacity=0.7, line width=0.3mm, rounded corners=0.2cm] (barycentric cs:L1=-0.5,L2=2.0,L3=3.5)--(barycentric cs:L1=3.5,L2=0.0,L3=1.5)--(barycentric cs:L1=4.5,L2=1.0,L3=-0.5);
\node[anchor=north] (nodeA3B1) at (barycentric cs:L1=-0.5,L2=2.0,L3=3.5) {$q_5$};
\node[anchor=south west] (nodeA3B1) at (barycentric cs:L1=4.5,L2=1.0,L3=-0.5) {$q'_5$};
\draw[opacity=0.7, line width=0.3mm, rounded corners=0.2cm] (barycentric cs:L1=1.5,L2=4.0,L3=-0.5)--(barycentric cs:L1=0.0,L2=2.5,L3=2.5)--(barycentric cs:L1=1.5,L2=-0.5,L3=4.0);
\node[anchor=south west] (nodeA2B2) at (barycentric cs:L1=1.5,L2=4.0,L3=-0.5) {$q_2$};
\node[anchor=south east] (nodeA2B2) at (barycentric cs:L1=1.5,L2=-0.5,L3=4.0) {$q'_2$};
\draw[opacity=0.7, line width=0.3mm, rounded corners=0.2cm] (barycentric cs:L1=-0.5,L2=5.0,L3=0.5)--(barycentric cs:L1=0.5,L2=4.5,L3=0.0)--(barycentric cs:L1=3.0,L2=-0.5,L3=2.5);
\node[anchor=north] (nodeA0B4) at (barycentric cs:L1=-0.5,L2=5.0,L3=0.5) {$q_3$};
\node[anchor=south east] (nodeA0B4) at (barycentric cs:L1=3.0,L2=-0.5,L3=2.5) {$q'_3$};
\draw[opacity=0.7, line width=0.3mm, rounded corners=0.2cm] (barycentric cs:L1=-0.5,L2=0.5,L3=5.0)--(barycentric cs:L1=0.5,L2=0.0,L3=4.5)--(barycentric cs:L1=3.0,L2=2.5,L3=-0.5);
\node[anchor=north] (nodeA0B4) at (barycentric cs:L1=-0.5,L2=0.5,L3=5.0) {$q'_1$};
\node[anchor=south west] (nodeA0B4) at (barycentric cs:L1=3.0,L2=2.5,L3=-0.5) {$q_1$};
\end{tikzpicture}}\\

$t=0$ & $t=1$ & $t=2$\\\hline
\end{tabular}

\begin{tabular}{|c|c|c|}\hline
\scalebox{0.6}{
\begin{tikzpicture}[scale=0.6]
\coordinate (L1) at (90:5.00);
\coordinate (L2) at (330:5.00);
\coordinate (L3) at (210:5.00);
\coordinate (C005) at (barycentric cs:L1=0.0,L2=0.0,L3=5.0);
\coordinate (C014) at (barycentric cs:L1=0.0,L2=1.0,L3=4.0);
\coordinate (C023) at (barycentric cs:L1=0.0,L2=2.0,L3=3.0);
\coordinate (C032) at (barycentric cs:L1=0.0,L2=3.0,L3=2.0);
\coordinate (C041) at (barycentric cs:L1=0.0,L2=4.0,L3=1.0);
\coordinate (C050) at (barycentric cs:L1=0.0,L2=5.0,L3=0.0);
\coordinate (C104) at (barycentric cs:L1=1.0,L2=0.0,L3=4.0);
\coordinate (C113) at (barycentric cs:L1=1.0,L2=1.0,L3=3.0);
\coordinate (C122) at (barycentric cs:L1=1.0,L2=2.0,L3=2.0);
\coordinate (C131) at (barycentric cs:L1=1.0,L2=3.0,L3=1.0);
\coordinate (C140) at (barycentric cs:L1=1.0,L2=4.0,L3=0.0);
\coordinate (C203) at (barycentric cs:L1=2.0,L2=0.0,L3=3.0);
\coordinate (C212) at (barycentric cs:L1=2.0,L2=1.0,L3=2.0);
\coordinate (C221) at (barycentric cs:L1=2.0,L2=2.0,L3=1.0);
\coordinate (C230) at (barycentric cs:L1=2.0,L2=3.0,L3=0.0);
\coordinate (C302) at (barycentric cs:L1=3.0,L2=0.0,L3=2.0);
\coordinate (C311) at (barycentric cs:L1=3.0,L2=1.0,L3=1.0);
\coordinate (C320) at (barycentric cs:L1=3.0,L2=2.0,L3=0.0);
\coordinate (C401) at (barycentric cs:L1=4.0,L2=0.0,L3=1.0);
\coordinate (C410) at (barycentric cs:L1=4.0,L2=1.0,L3=0.0);
\coordinate (C500) at (barycentric cs:L1=5.0,L2=0.0,L3=0.0);
\draw[fill=green] (barycentric cs:L1=0.0,L2=0.0,L3=5.0) circle (0.14285714285714288);
\draw[fill=red] (barycentric cs:L1=0.0,L2=1.0,L3=4.0) circle (0.14285714285714288);
\draw[fill=blue] (barycentric cs:L1=0.0,L2=2.0,L3=3.0) circle (0.14285714285714288);
\draw[fill=green] (barycentric cs:L1=0.0,L2=3.0,L3=2.0) circle (0.14285714285714288);
\draw[fill=red] (barycentric cs:L1=0.0,L2=4.0,L3=1.0) circle (0.14285714285714288);
\draw[fill=blue] (barycentric cs:L1=0.0,L2=5.0,L3=0.0) circle (0.14285714285714288);
\draw[fill=blue] (barycentric cs:L1=1.0,L2=0.0,L3=4.0) circle (0.14285714285714288);
\draw[fill=green] (barycentric cs:L1=1.0,L2=1.0,L3=3.0) circle (0.14285714285714288);
\draw[fill=red] (barycentric cs:L1=1.0,L2=2.0,L3=2.0) circle (0.14285714285714288);
\draw[fill=blue] (barycentric cs:L1=1.0,L2=3.0,L3=1.0) circle (0.14285714285714288);
\draw[fill=green] (barycentric cs:L1=1.0,L2=4.0,L3=0.0) circle (0.14285714285714288);
\draw[fill=red] (barycentric cs:L1=2.0,L2=0.0,L3=3.0) circle (0.14285714285714288);
\draw[fill=blue] (barycentric cs:L1=2.0,L2=1.0,L3=2.0) circle (0.14285714285714288);
\draw[fill=green] (barycentric cs:L1=2.0,L2=2.0,L3=1.0) circle (0.14285714285714288);
\draw[fill=red] (barycentric cs:L1=2.0,L2=3.0,L3=0.0) circle (0.14285714285714288);
\draw[fill=green] (barycentric cs:L1=3.0,L2=0.0,L3=2.0) circle (0.14285714285714288);
\draw[fill=red] (barycentric cs:L1=3.0,L2=1.0,L3=1.0) circle (0.14285714285714288);
\draw[fill=blue] (barycentric cs:L1=3.0,L2=2.0,L3=0.0) circle (0.14285714285714288);
\draw[fill=blue] (barycentric cs:L1=4.0,L2=0.0,L3=1.0) circle (0.14285714285714288);
\draw[fill=green] (barycentric cs:L1=4.0,L2=1.0,L3=0.0) circle (0.14285714285714288);
\draw[fill=red] (barycentric cs:L1=5.0,L2=0.0,L3=0.0) circle (0.14285714285714288);
\draw[opacity=0.7, line width=0.3mm, rounded corners=0.2cm] (barycentric cs:L1=-0.5,L2=2.5,L3=3.0)--(barycentric cs:L1=4.5,L2=0.0,L3=0.5)--(barycentric cs:L1=5.0,L2=0.5,L3=-0.5);
\node[anchor=north] (nodeA4B0) at (barycentric cs:L1=-0.5,L2=2.5,L3=3.0) {$q_5$};
\node[anchor=south west] (nodeA4B0) at (barycentric cs:L1=5.0,L2=0.5,L3=-0.5) {$q'_5$};
\draw[opacity=0.7, line width=0.3mm, rounded corners=0.2cm] (barycentric cs:L1=0.5,L2=5.0,L3=-0.5)--(barycentric cs:L1=0.0,L2=4.5,L3=0.5)--(barycentric cs:L1=2.5,L2=-0.5,L3=3.0);
\node[anchor=south west] (nodeA4B0) at (barycentric cs:L1=0.5,L2=5.0,L3=-0.5) {$q_3$};
\node[anchor=south east] (nodeA4B0) at (barycentric cs:L1=2.5,L2=-0.5,L3=3.0) {$q'_3$};
\draw[opacity=0.7, line width=0.3mm, rounded corners=0.2cm] (barycentric cs:L1=-0.5,L2=4.0,L3=1.5)--(barycentric cs:L1=2.5,L2=2.5,L3=0.0)--(barycentric cs:L1=4.0,L2=-0.5,L3=1.5);
\node[anchor=north] (nodeA2B2) at (barycentric cs:L1=-0.5,L2=4.0,L3=1.5) {$q_4$};
\node[anchor=south east] (nodeA2B2) at (barycentric cs:L1=4.0,L2=-0.5,L3=1.5) {$q'_4$};
\draw[opacity=0.7, line width=0.3mm, rounded corners=0.2cm] (barycentric cs:L1=-0.5,L2=1.0,L3=4.5)--(barycentric cs:L1=1.5,L2=0.0,L3=3.5)--(barycentric cs:L1=3.5,L2=2.0,L3=-0.5);
\node[anchor=north] (nodeA1B3) at (barycentric cs:L1=-0.5,L2=1.0,L3=4.5) {$q'_1$};
\node[anchor=south west] (nodeA1B3) at (barycentric cs:L1=3.5,L2=2.0,L3=-0.5) {$q_1$};
\draw[opacity=0.7, line width=0.3mm, rounded corners=0.2cm] (barycentric cs:L1=2.0,L2=3.5,L3=-0.5)--(barycentric cs:L1=0.0,L2=1.5,L3=3.5)--(barycentric cs:L1=1.0,L2=-0.5,L3=4.5);
\node[anchor=south west] (nodeA1B3) at (barycentric cs:L1=2.0,L2=3.5,L3=-0.5) {$q_2$};
\node[anchor=south east] (nodeA1B3) at (barycentric cs:L1=1.0,L2=-0.5,L3=4.5) {$q'_2$};
\end{tikzpicture}}
&
\scalebox{0.6}{
\begin{tikzpicture}[scale=0.6]
\coordinate (L1) at (90:5.00);
\coordinate (L2) at (330:5.00);
\coordinate (L3) at (210:5.00);
\coordinate (C005) at (barycentric cs:L1=0.0,L2=0.0,L3=5.0);
\coordinate (C014) at (barycentric cs:L1=0.0,L2=1.0,L3=4.0);
\coordinate (C023) at (barycentric cs:L1=0.0,L2=2.0,L3=3.0);
\coordinate (C032) at (barycentric cs:L1=0.0,L2=3.0,L3=2.0);
\coordinate (C041) at (barycentric cs:L1=0.0,L2=4.0,L3=1.0);
\coordinate (C050) at (barycentric cs:L1=0.0,L2=5.0,L3=0.0);
\coordinate (C104) at (barycentric cs:L1=1.0,L2=0.0,L3=4.0);
\coordinate (C113) at (barycentric cs:L1=1.0,L2=1.0,L3=3.0);
\coordinate (C122) at (barycentric cs:L1=1.0,L2=2.0,L3=2.0);
\coordinate (C131) at (barycentric cs:L1=1.0,L2=3.0,L3=1.0);
\coordinate (C140) at (barycentric cs:L1=1.0,L2=4.0,L3=0.0);
\coordinate (C203) at (barycentric cs:L1=2.0,L2=0.0,L3=3.0);
\coordinate (C212) at (barycentric cs:L1=2.0,L2=1.0,L3=2.0);
\coordinate (C221) at (barycentric cs:L1=2.0,L2=2.0,L3=1.0);
\coordinate (C230) at (barycentric cs:L1=2.0,L2=3.0,L3=0.0);
\coordinate (C302) at (barycentric cs:L1=3.0,L2=0.0,L3=2.0);
\coordinate (C311) at (barycentric cs:L1=3.0,L2=1.0,L3=1.0);
\coordinate (C320) at (barycentric cs:L1=3.0,L2=2.0,L3=0.0);
\coordinate (C401) at (barycentric cs:L1=4.0,L2=0.0,L3=1.0);
\coordinate (C410) at (barycentric cs:L1=4.0,L2=1.0,L3=0.0);
\coordinate (C500) at (barycentric cs:L1=5.0,L2=0.0,L3=0.0);
\draw[fill=green] (barycentric cs:L1=0.0,L2=0.0,L3=5.0) circle (0.14285714285714288);
\draw[fill=red] (barycentric cs:L1=0.0,L2=1.0,L3=4.0) circle (0.14285714285714288);
\draw[fill=blue] (barycentric cs:L1=0.0,L2=2.0,L3=3.0) circle (0.14285714285714288);
\draw[fill=green] (barycentric cs:L1=0.0,L2=3.0,L3=2.0) circle (0.14285714285714288);
\draw[fill=red] (barycentric cs:L1=0.0,L2=4.0,L3=1.0) circle (0.14285714285714288);
\draw[fill=blue] (barycentric cs:L1=0.0,L2=5.0,L3=0.0) circle (0.14285714285714288);
\draw[fill=blue] (barycentric cs:L1=1.0,L2=0.0,L3=4.0) circle (0.14285714285714288);
\draw[fill=green] (barycentric cs:L1=1.0,L2=1.0,L3=3.0) circle (0.14285714285714288);
\draw[fill=red] (barycentric cs:L1=1.0,L2=2.0,L3=2.0) circle (0.14285714285714288);
\draw[fill=blue] (barycentric cs:L1=1.0,L2=3.0,L3=1.0) circle (0.14285714285714288);
\draw[fill=green] (barycentric cs:L1=1.0,L2=4.0,L3=0.0) circle (0.14285714285714288);
\draw[fill=red] (barycentric cs:L1=2.0,L2=0.0,L3=3.0) circle (0.14285714285714288);
\draw[fill=blue] (barycentric cs:L1=2.0,L2=1.0,L3=2.0) circle (0.14285714285714288);
\draw[fill=green] (barycentric cs:L1=2.0,L2=2.0,L3=1.0) circle (0.14285714285714288);
\draw[fill=red] (barycentric cs:L1=2.0,L2=3.0,L3=0.0) circle (0.14285714285714288);
\draw[fill=green] (barycentric cs:L1=3.0,L2=0.0,L3=2.0) circle (0.14285714285714288);
\draw[fill=red] (barycentric cs:L1=3.0,L2=1.0,L3=1.0) circle (0.14285714285714288);
\draw[fill=blue] (barycentric cs:L1=3.0,L2=2.0,L3=0.0) circle (0.14285714285714288);
\draw[fill=blue] (barycentric cs:L1=4.0,L2=0.0,L3=1.0) circle (0.14285714285714288);
\draw[fill=green] (barycentric cs:L1=4.0,L2=1.0,L3=0.0) circle (0.14285714285714288);
\draw[fill=red] (barycentric cs:L1=5.0,L2=0.0,L3=0.0) circle (0.14285714285714288);
\draw[opacity=0.7, line width=0.3mm, rounded corners=0.2cm] (barycentric cs:L1=-0.5,L2=3.0,L3=2.5)--(barycentric cs:L1=4.5,L2=0.5,L3=0.0)--(barycentric cs:L1=5.0,L2=-0.5,L3=0.5);
\node[anchor=north] (nodeA4B0) at (barycentric cs:L1=-0.5,L2=3.0,L3=2.5) {$q_5$};
\node[anchor=south east] (nodeA4B0) at (barycentric cs:L1=5.0,L2=-0.5,L3=0.5) {$q'_5$};
\draw[opacity=0.7, line width=0.3mm, rounded corners=0.2cm] (barycentric cs:L1=1.0,L2=4.5,L3=-0.5)--(barycentric cs:L1=0.0,L2=3.5,L3=1.5)--(barycentric cs:L1=2.0,L2=-0.5,L3=3.5);
\node[anchor=south west] (nodeA3B1) at (barycentric cs:L1=1.0,L2=4.5,L3=-0.5) {$q_3$};
\node[anchor=south east] (nodeA3B1) at (barycentric cs:L1=2.0,L2=-0.5,L3=3.5) {$q'_3$};
\draw[opacity=0.7, line width=0.3mm, rounded corners=0.2cm] (barycentric cs:L1=-0.5,L2=1.5,L3=4.0)--(barycentric cs:L1=2.5,L2=0.0,L3=2.5)--(barycentric cs:L1=4.0,L2=1.5,L3=-0.5);
\node[anchor=north] (nodeA2B2) at (barycentric cs:L1=-0.5,L2=1.5,L3=4.0) {$q'_1$};
\node[anchor=south west] (nodeA2B2) at (barycentric cs:L1=4.0,L2=1.5,L3=-0.5) {$q_1$};
\draw[opacity=0.7, line width=0.3mm, rounded corners=0.2cm] (barycentric cs:L1=-0.5,L2=4.5,L3=1.0)--(barycentric cs:L1=1.5,L2=3.5,L3=0.0)--(barycentric cs:L1=3.5,L2=-0.5,L3=2.0);
\node[anchor=north] (nodeA1B3) at (barycentric cs:L1=-0.5,L2=4.5,L3=1.0) {$q_4$};
\node[anchor=south east] (nodeA1B3) at (barycentric cs:L1=3.5,L2=-0.5,L3=2.0) {$q'_4$};
\draw[opacity=0.7, line width=0.3mm, rounded corners=0.2cm] (barycentric cs:L1=2.5,L2=3.0,L3=-0.5)--(barycentric cs:L1=0.0,L2=0.5,L3=4.5)--(barycentric cs:L1=0.5,L2=-0.5,L3=5.0);
\node[anchor=south west] (nodeA0B4) at (barycentric cs:L1=2.5,L2=3.0,L3=-0.5) {$q_2$};
\node[anchor=south east] (nodeA0B4) at (barycentric cs:L1=0.5,L2=-0.5,L3=5.0) {$q'_2$};
\end{tikzpicture}}

&
\scalebox{0.6}{
\begin{tikzpicture}[scale=0.6]
\coordinate (L1) at (90:5.00);
\coordinate (L2) at (330:5.00);
\coordinate (L3) at (210:5.00);
\coordinate (C005) at (barycentric cs:L1=0.0,L2=0.0,L3=5.0);
\coordinate (C014) at (barycentric cs:L1=0.0,L2=1.0,L3=4.0);
\coordinate (C023) at (barycentric cs:L1=0.0,L2=2.0,L3=3.0);
\coordinate (C032) at (barycentric cs:L1=0.0,L2=3.0,L3=2.0);
\coordinate (C041) at (barycentric cs:L1=0.0,L2=4.0,L3=1.0);
\coordinate (C050) at (barycentric cs:L1=0.0,L2=5.0,L3=0.0);
\coordinate (C104) at (barycentric cs:L1=1.0,L2=0.0,L3=4.0);
\coordinate (C113) at (barycentric cs:L1=1.0,L2=1.0,L3=3.0);
\coordinate (C122) at (barycentric cs:L1=1.0,L2=2.0,L3=2.0);
\coordinate (C131) at (barycentric cs:L1=1.0,L2=3.0,L3=1.0);
\coordinate (C140) at (barycentric cs:L1=1.0,L2=4.0,L3=0.0);
\coordinate (C203) at (barycentric cs:L1=2.0,L2=0.0,L3=3.0);
\coordinate (C212) at (barycentric cs:L1=2.0,L2=1.0,L3=2.0);
\coordinate (C221) at (barycentric cs:L1=2.0,L2=2.0,L3=1.0);
\coordinate (C230) at (barycentric cs:L1=2.0,L2=3.0,L3=0.0);
\coordinate (C302) at (barycentric cs:L1=3.0,L2=0.0,L3=2.0);
\coordinate (C311) at (barycentric cs:L1=3.0,L2=1.0,L3=1.0);
\coordinate (C320) at (barycentric cs:L1=3.0,L2=2.0,L3=0.0);
\coordinate (C401) at (barycentric cs:L1=4.0,L2=0.0,L3=1.0);
\coordinate (C410) at (barycentric cs:L1=4.0,L2=1.0,L3=0.0);
\coordinate (C500) at (barycentric cs:L1=5.0,L2=0.0,L3=0.0);
\draw[fill=green] (barycentric cs:L1=0.0,L2=0.0,L3=5.0) circle (0.14285714285714288);
\draw[fill=red] (barycentric cs:L1=0.0,L2=1.0,L3=4.0) circle (0.14285714285714288);
\draw[fill=blue] (barycentric cs:L1=0.0,L2=2.0,L3=3.0) circle (0.14285714285714288);
\draw[fill=green] (barycentric cs:L1=0.0,L2=3.0,L3=2.0) circle (0.14285714285714288);
\draw[fill=red] (barycentric cs:L1=0.0,L2=4.0,L3=1.0) circle (0.14285714285714288);
\draw[fill=blue] (barycentric cs:L1=0.0,L2=5.0,L3=0.0) circle (0.14285714285714288);
\draw[fill=blue] (barycentric cs:L1=1.0,L2=0.0,L3=4.0) circle (0.14285714285714288);
\draw[fill=green] (barycentric cs:L1=1.0,L2=1.0,L3=3.0) circle (0.14285714285714288);
\draw[fill=red] (barycentric cs:L1=1.0,L2=2.0,L3=2.0) circle (0.14285714285714288);
\draw[fill=blue] (barycentric cs:L1=1.0,L2=3.0,L3=1.0) circle (0.14285714285714288);
\draw[fill=green] (barycentric cs:L1=1.0,L2=4.0,L3=0.0) circle (0.14285714285714288);
\draw[fill=red] (barycentric cs:L1=2.0,L2=0.0,L3=3.0) circle (0.14285714285714288);
\draw[fill=blue] (barycentric cs:L1=2.0,L2=1.0,L3=2.0) circle (0.14285714285714288);
\draw[fill=green] (barycentric cs:L1=2.0,L2=2.0,L3=1.0) circle (0.14285714285714288);
\draw[fill=red] (barycentric cs:L1=2.0,L2=3.0,L3=0.0) circle (0.14285714285714288);
\draw[fill=green] (barycentric cs:L1=3.0,L2=0.0,L3=2.0) circle (0.14285714285714288);
\draw[fill=red] (barycentric cs:L1=3.0,L2=1.0,L3=1.0) circle (0.14285714285714288);
\draw[fill=blue] (barycentric cs:L1=3.0,L2=2.0,L3=0.0) circle (0.14285714285714288);
\draw[fill=blue] (barycentric cs:L1=4.0,L2=0.0,L3=1.0) circle (0.14285714285714288);
\draw[fill=green] (barycentric cs:L1=4.0,L2=1.0,L3=0.0) circle (0.14285714285714288);
\draw[fill=red] (barycentric cs:L1=5.0,L2=0.0,L3=0.0) circle (0.14285714285714288);
\draw[opacity=0.7, line width=0.3mm, rounded corners=0.2cm] (barycentric cs:L1=-0.5,L2=3.5,L3=2.0)--(barycentric cs:L1=3.5,L2=1.5,L3=0.0)--(barycentric cs:L1=4.5,L2=-0.5,L3=1.0);
\node[anchor=north] (nodeA3B1) at (barycentric cs:L1=-0.5,L2=3.5,L3=2.0) {$q_5$};
\node[anchor=south east] (nodeA3B1) at (barycentric cs:L1=4.5,L2=-0.5,L3=1.0) {$q'_5$};
\draw[opacity=0.7, line width=0.3mm, rounded corners=0.2cm] (barycentric cs:L1=-0.5,L2=2.0,L3=3.5)--(barycentric cs:L1=3.5,L2=0.0,L3=1.5)--(barycentric cs:L1=4.5,L2=1.0,L3=-0.5);
\node[anchor=north] (nodeA3B1) at (barycentric cs:L1=-0.5,L2=2.0,L3=3.5) {$q'_1$};
\node[anchor=south west] (nodeA3B1) at (barycentric cs:L1=4.5,L2=1.0,L3=-0.5) {$q_1$};
\draw[opacity=0.7, line width=0.3mm, rounded corners=0.2cm] (barycentric cs:L1=1.5,L2=4.0,L3=-0.5)--(barycentric cs:L1=0.0,L2=2.5,L3=2.5)--(barycentric cs:L1=1.5,L2=-0.5,L3=4.0);
\node[anchor=south west] (nodeA2B2) at (barycentric cs:L1=1.5,L2=4.0,L3=-0.5) {$q_3$};
\node[anchor=south east] (nodeA2B2) at (barycentric cs:L1=1.5,L2=-0.5,L3=4.0) {$q'_3$};
\draw[opacity=0.7, line width=0.3mm, rounded corners=0.2cm] (barycentric cs:L1=-0.5,L2=5.0,L3=0.5)--(barycentric cs:L1=0.5,L2=4.5,L3=0.0)--(barycentric cs:L1=3.0,L2=-0.5,L3=2.5);
\node[anchor=north] (nodeA0B4) at (barycentric cs:L1=-0.5,L2=5.0,L3=0.5) {$q_4$};
\node[anchor=south east] (nodeA0B4) at (barycentric cs:L1=3.0,L2=-0.5,L3=2.5) {$q'_4$};
\draw[opacity=0.7, line width=0.3mm, rounded corners=0.2cm] (barycentric cs:L1=-0.5,L2=0.5,L3=5.0)--(barycentric cs:L1=0.5,L2=0.0,L3=4.5)--(barycentric cs:L1=3.0,L2=2.5,L3=-0.5);
\node[anchor=north] (nodeA0B4) at (barycentric cs:L1=-0.5,L2=0.5,L3=5.0) {$q'_2$};
\node[anchor=south west] (nodeA0B4) at (barycentric cs:L1=3.0,L2=2.5,L3=-0.5) {$q_2$};
\end{tikzpicture}}\\

$t=3$ & $t=4$ & $t=5$\\\hline
\end{tabular}

\caption{\label{fig:pseudolines} The arrangements $\Acal^\parr t$ for $m=5$ and $0\leq t\leq 5$.}
\end{figure}

Let 
\[w_1=(-1,1/2,1/2),\quad w_2=(1/2,-1,1/2),\quad w_3=(1/2,1/2,-1),\] 
and $\e\in\{0,1,2\}$ be a color. We are going to introduce a pseudoline arrangement $\Acal^\parr \e$. For each integer $0\leq b\leq m-1$, put $a=m-1-b$ and add the following pseudolines to $\Acal^\parr\e$:
\begin{itemize}
 \item if $b\equiv 2+\e\pmod 3$, let $p=(a+1/2,b+1/2,0)$ and the pseudoline is the union of two rays\footnote{more precisely, we fix some big disc $D$ and intersect this union of two rays with $D$.}: $p+tw_2$ and $p+tw_1$ for $t\in\R_{\geq 0}$;
 \item if $b\equiv 1-\e\pmod 3$, let $p=(a+1/2,0,b+1/2)$ and the pseudoline is the union of two rays: $p+tw_1$ and $p+tw_3$ for $t\in\R_{\geq 0}$;
 \item if $b\equiv 2-m+\e\pmod 3$, let $p=(0,a+1/2,b+1/2)$ and the pseudoline is the union of two rays: $p+tw_2$ and $p+tw_3$ for $t\in\R_{\geq 0}$.
\end{itemize}

It is clear that for each $\e$, the above rules indeed define a family of pseudolines, i.e. every two of them intersect exactly once. Note that it is not true that each $b$ contributes exactly one pseudoline from the above list. However, regardless of the residue of $m$ modulo $3$, we get exactly $m$ pseudolines in $\Acal^\parr \e$. This is true for the following reason. Let $S$ be the set of all points on the boundary of $\TR_m$ that have a zero and two half-integers as coordinates. We get that $S$ contains precisely $3m$ points, and if we label them as $s_0,s_2,\dots,s_{3m-1}$ in the clockwise order then each $s_i$ appears in the definition of $\Acal^\parr \e$ whenever $i\equiv\e\pmod 3$. This proves that we get $m$ pseudolines in each of $\Acal^\parr\e$, and also gives a natural way to label the pseudolines of $\Acal^\parr\e$ by the elements of $[m]:=\{1,2,\dots,m\}$ by just saying that the $i$-th pseudoline in $\Acal^\parr\e$ is the one that comes from the point $s_{3(i-1)+\e}$. 

We now extend this family $\Acal^\parr0,\Acal^\parr 1,\Acal^\parr 2$ of pseudoline arrangements to a bigger family $\Acal^\parr t$ for $t\in \Z$ as follows. For each $t\in\Z$, the pseudoline arrangement $\Acal^\parr t$ coincides with $\Acal^\parr \e$ for $t\equiv\e\pmod 3$, except that the labels of the pseudolines are different. Namely, the $i$-th pseudoline in $\Acal^\parr t$ is the one that comes from the point $s_{3(i-1)+t}$ where we take the index $3(i-1)+t$ modulo $3m$. 

It is easy to see that for every $t$, each non-boundary vertex of $\TR_m$ is contained in a unique bounded region of $\Acal^\parr t$, and conversely, every bounded region of $\Acal^\parr t$ contains precisely one vertex of $\TR_m$. For $t\equiv\e\pmod 3$, mutating all vertices of color $\e$ transforms $\Acal^\parr t$ into $\Acal^\parr{t+1}$, and the formulas clearly match each other. Finally, we can see that the pseudoline arrangements $\Acal^\parr t$ and $\Acal^\parr{t+m}$ differ by a $180^\circ$ rotation, that is, by a transformation that switches $q_k$ with $q_k'$ for all $k\in [m]$. For example, in Figure~\ref{fig:pseudolines}, we see that $\Acal^\parr 0$ and $\Acal^\parr 5$ are the same modulo switching $q_k$ with $q_k'$ for $k=1,2,3,4,5$. Thus the arrangements $\Acal^\parr t$ and $\Acal^\parr{t+2m}$ coincide. However, for every non-boundary vertex $v\in\TR_m$ inside some bounded region of $\Acal^\parr t$, the vertex inside the corresponding bounded region of $\Acal^\parr{t+2m}$ is not $v$ but $\sigma v$. This finishes the proof of Theorem~\ref{thm:cube_periodic}.\qed

\def\Gbound{{\partial G(v,t)}}
\def\GboundG{{\partial^{\operatorname{g}} G(v,t)}}
\def\GboundB{{\partial^{\operatorname{b}} G(v,t)}}

\section{Groves and networks}\label{sect:groves}

In this section, we prove Theorems~\ref{thm:recurrence_1},~\ref{thm:cube_formula},~\ref{thm:pleth}, and~\ref{thm:entropy}. 

\subsection{Groves in a triangle}\label{sect:groves_triangle}

Recall that $\TP_0$ denotes the set of lattice points in the plane $i+j+k=0$ and $G$ is an undirected graph with vertex set $\TP_0$ and edge set consisting of edges $(v,v+e_{12})$, $(v,v+e_{23})$, $(v,v+e_{31})$ for every blue and green vertex $v\in\TP_0$. Let us fix a vertex $v=(i,j,k)\in\TP_0$ and an integer $t\geq 2$ such that $t+1\equiv\e_v\pmod 3$. We would like to explain Carroll-Speyer's formula for the value of the unbounded cube recurrence $\Cube_v(t+1)$ in terms of its initial values $\x$.

Let $\TR(v,t)$ be the convex hull of 
\[v+te_{12},v+te_{23},v+te_{31}.\] 
Define the graph $G(v,t)$ consisting of all the lozenges of $G$ that lie inside $\TR(v,t)$. Let $\Vert(v,t)$ be the vertex set of $G(v,t)$. We say that $u\in\Vert(v,t)$ is a \emph{boundary vertex} of $G(v,t)$ if it belongs to the outer face of $G(v,t)$. We denote by $\Gbound$ the set of all boundary vertices of $G(v,t)$. Let us list the elements of $\Gbound$ and label them
\[(a_1,a_2,\dots,a_{2t-1}=b_1,b_2,\dots,b_{2t-1}=c_1,c_2,\dots,c_{2t-1}=a_1)\]
in counterclockwise order. Explicitly, for $i=1,2,\dots,t-1$, we put
\[a_{2i}=v+(i, t-2i,i-t);\quad b_{2i}=v+(t-2i,i-t,i);\quad c_{2i}=v+(i-t,i,t-2i).\]
Similarly, for $i=0,1,\dots,t-1$, we put
\begin{eqnarray*}
a_{2i+1}&=&v+(i,t-1-2i,i-t+1);\\ 
b_{2i+1}&=&v+(t-1-2i,i-t+1,i);\\
c_{2i+1}&=&v+(i-t+1,i,t-1-2i). 
\end{eqnarray*}

\begin{figure}
 \centering
\scalebox{1.0}{
\begin{tikzpicture}[scale=0.3]
\coordinate (L1) at (90:5.00);
\coordinate (L2) at (330:5.00);
\coordinate (L3) at (210:5.00);
\coordinate (Cm4I0I4) at (barycentric cs:L1=-3.0,L2=1.0,L3=5.0);
\coordinate (Cm4I1I3) at (barycentric cs:L1=-3.0,L2=2.0,L3=4.0);
\coordinate (Cm3Im1I4) at (barycentric cs:L1=-2.0,L2=0.0,L3=5.0);
\coordinate (Cm3I0I3) at (barycentric cs:L1=-2.0,L2=1.0,L3=4.0);
\coordinate (Cm3I1I2) at (barycentric cs:L1=-2.0,L2=2.0,L3=3.0);
\coordinate (Cm3I2I1) at (barycentric cs:L1=-2.0,L2=3.0,L3=2.0);
\coordinate (Cm2Im1I3) at (barycentric cs:L1=-1.0,L2=0.0,L3=4.0);
\coordinate (Cm2I0I2) at (barycentric cs:L1=-1.0,L2=1.0,L3=3.0);
\coordinate (Cm2I1I1) at (barycentric cs:L1=-1.0,L2=2.0,L3=2.0);
\coordinate (Cm2I2I0) at (barycentric cs:L1=-1.0,L2=3.0,L3=1.0);
\coordinate (Cm2I3Im1) at (barycentric cs:L1=-1.0,L2=4.0,L3=0.0);
\coordinate (Cm1Im2I3) at (barycentric cs:L1=0.0,L2=-1.0,L3=4.0);
\coordinate (Cm1Im1I2) at (barycentric cs:L1=0.0,L2=0.0,L3=3.0);
\coordinate (Cm1I0I1) at (barycentric cs:L1=0.0,L2=1.0,L3=2.0);
\coordinate (Cm1I1I0) at (barycentric cs:L1=0.0,L2=2.0,L3=1.0);
\coordinate (Cm1I2Im1) at (barycentric cs:L1=0.0,L2=3.0,L3=0.0);
\coordinate (Cm1I3Im2) at (barycentric cs:L1=0.0,L2=4.0,L3=-1.0);
\coordinate (Cm1I4Im3) at (barycentric cs:L1=0.0,L2=5.0,L3=-2.0);
\coordinate (C0Im2I2) at (barycentric cs:L1=1.0,L2=-1.0,L3=3.0);
\coordinate (C0Im1I1) at (barycentric cs:L1=1.0,L2=0.0,L3=2.0);
\coordinate (C0I0I0) at (barycentric cs:L1=1.0,L2=1.0,L3=1.0);
\coordinate (C0I1Im1) at (barycentric cs:L1=1.0,L2=2.0,L3=0.0);
\coordinate (C0I2Im2) at (barycentric cs:L1=1.0,L2=3.0,L3=-1.0);
\coordinate (C0I3Im3) at (barycentric cs:L1=1.0,L2=4.0,L3=-2.0);
\coordinate (C0I4Im4) at (barycentric cs:L1=1.0,L2=5.0,L3=-3.0);
\coordinate (C1Im3I2) at (barycentric cs:L1=2.0,L2=-2.0,L3=3.0);
\coordinate (C1Im2I1) at (barycentric cs:L1=2.0,L2=-1.0,L3=2.0);
\coordinate (C1Im1I0) at (barycentric cs:L1=2.0,L2=0.0,L3=1.0);
\coordinate (C1I0Im1) at (barycentric cs:L1=2.0,L2=1.0,L3=0.0);
\coordinate (C1I1Im2) at (barycentric cs:L1=2.0,L2=2.0,L3=-1.0);
\coordinate (C1I2Im3) at (barycentric cs:L1=2.0,L2=3.0,L3=-2.0);
\coordinate (C1I3Im4) at (barycentric cs:L1=2.0,L2=4.0,L3=-3.0);
\coordinate (C2Im3I1) at (barycentric cs:L1=3.0,L2=-2.0,L3=2.0);
\coordinate (C2Im2I0) at (barycentric cs:L1=3.0,L2=-1.0,L3=1.0);
\coordinate (C2Im1Im1) at (barycentric cs:L1=3.0,L2=0.0,L3=0.0);
\coordinate (C2I0Im2) at (barycentric cs:L1=3.0,L2=1.0,L3=-1.0);
\coordinate (C2I1Im3) at (barycentric cs:L1=3.0,L2=2.0,L3=-2.0);
\coordinate (C3Im4I1) at (barycentric cs:L1=4.0,L2=-3.0,L3=2.0);
\coordinate (C3Im3I0) at (barycentric cs:L1=4.0,L2=-2.0,L3=1.0);
\coordinate (C3Im2Im1) at (barycentric cs:L1=4.0,L2=-1.0,L3=0.0);
\coordinate (C3Im1Im2) at (barycentric cs:L1=4.0,L2=0.0,L3=-1.0);
\coordinate (C4Im4I0) at (barycentric cs:L1=5.0,L2=-3.0,L3=1.0);
\coordinate (C4Im3Im1) at (barycentric cs:L1=5.0,L2=-2.0,L3=0.0);
\draw[opacity=0.4] (Cm3I1I2) -- (Cm2I0I2);
\draw[opacity=0.4] (Cm3I1I2) -- (Cm4I1I3);
\draw[opacity=0.4] (Cm3I1I2) -- (Cm3I2I1);
\draw[opacity=0.4] (Cm3I2I1) -- (Cm2I1I1);
\draw[opacity=0.4] (C1Im3I2) -- (C1Im2I1);
\draw[opacity=0.4] (C1I2Im3) -- (C2I1Im3);
\draw[opacity=0.4] (C1I2Im3) -- (C0I2Im2);
\draw[opacity=0.4] (C1I2Im3) -- (C1I3Im4);
\draw[opacity=0.4] (C2Im3I1) -- (C3Im4I1);
\draw[opacity=0.4] (C2Im3I1) -- (C1Im3I2);
\draw[opacity=0.4] (C2Im3I1) -- (C2Im2I0);
\draw[opacity=0.4] (C2I1Im3) -- (C1I1Im2);
\draw[opacity=0.4] (Cm4I1I3) -- (Cm3I0I3);
\draw[opacity=0.4] (C1I3Im4) -- (C0I3Im3);
\draw[opacity=0.4] (C3Im4I1) -- (C3Im3I0);
\draw[opacity=0.4] (Cm1I0I1) -- (C0Im1I1);
\draw[opacity=0.4] (Cm1I0I1) -- (Cm2I0I2);
\draw[opacity=0.4] (Cm1I0I1) -- (Cm1I1I0);
\draw[opacity=0.4] (Cm1I1I0) -- (C0I0I0);
\draw[opacity=0.4] (Cm1I1I0) -- (Cm2I1I1);
\draw[opacity=0.4] (Cm1I1I0) -- (Cm1I2Im1);
\draw[opacity=0.4] (C0Im1I1) -- (C1Im2I1);
\draw[opacity=0.4] (C0Im1I1) -- (Cm1Im1I2);
\draw[opacity=0.4] (C0Im1I1) -- (C0I0I0);
\draw[opacity=0.4] (C0I1Im1) -- (C1I0Im1);
\draw[opacity=0.4] (C0I1Im1) -- (Cm1I1I0);
\draw[opacity=0.4] (C0I1Im1) -- (C0I2Im2);
\draw[opacity=0.4] (C1Im1I0) -- (C2Im2I0);
\draw[opacity=0.4] (C1Im1I0) -- (C0Im1I1);
\draw[opacity=0.4] (C1Im1I0) -- (C1I0Im1);
\draw[opacity=0.4] (C1I0Im1) -- (C2Im1Im1);
\draw[opacity=0.4] (C1I0Im1) -- (C0I0I0);
\draw[opacity=0.4] (C1I0Im1) -- (C1I1Im2);
\draw[opacity=0.4] (Cm2I0I2) -- (Cm1Im1I2);
\draw[opacity=0.4] (Cm2I0I2) -- (Cm3I0I3);
\draw[opacity=0.4] (Cm2I0I2) -- (Cm2I1I1);
\draw[opacity=0.4] (Cm2I2I0) -- (Cm1I1I0);
\draw[opacity=0.4] (Cm2I2I0) -- (Cm3I2I1);
\draw[opacity=0.4] (Cm2I2I0) -- (Cm2I3Im1);
\draw[opacity=0.4] (C0Im2I2) -- (C1Im3I2);
\draw[opacity=0.4] (C0Im2I2) -- (Cm1Im2I3);
\draw[opacity=0.4] (C0Im2I2) -- (C0Im1I1);
\draw[opacity=0.4] (C0I2Im2) -- (C1I1Im2);
\draw[opacity=0.4] (C0I2Im2) -- (Cm1I2Im1);
\draw[opacity=0.4] (C0I2Im2) -- (C0I3Im3);
\draw[opacity=0.4] (C2Im2I0) -- (C3Im3I0);
\draw[opacity=0.4] (C2Im2I0) -- (C1Im2I1);
\draw[opacity=0.4] (C2Im2I0) -- (C2Im1Im1);
\draw[opacity=0.4] (C2I0Im2) -- (C3Im1Im2);
\draw[opacity=0.4] (C2I0Im2) -- (C1I0Im1);
\draw[opacity=0.4] (C2I0Im2) -- (C2I1Im3);
\draw[opacity=0.4] (Cm1Im2I3) -- (Cm1Im1I2);
\draw[opacity=0.4] (Cm1I3Im2) -- (C0I2Im2);
\draw[opacity=0.4] (Cm1I3Im2) -- (Cm2I3Im1);
\draw[opacity=0.4] (Cm1I3Im2) -- (Cm1I4Im3);
\draw[opacity=0.4] (Cm2Im1I3) -- (Cm1Im2I3);
\draw[opacity=0.4] (Cm2Im1I3) -- (Cm3Im1I4);
\draw[opacity=0.4] (Cm2Im1I3) -- (Cm2I0I2);
\draw[opacity=0.4] (Cm2I3Im1) -- (Cm1I2Im1);
\draw[opacity=0.4] (C3Im1Im2) -- (C2Im1Im1);
\draw[opacity=0.4] (C3Im2Im1) -- (C4Im3Im1);
\draw[opacity=0.4] (C3Im2Im1) -- (C2Im2I0);
\draw[opacity=0.4] (C3Im2Im1) -- (C3Im1Im2);
\draw[opacity=0.4] (Cm1I4Im3) -- (C0I3Im3);
\draw[opacity=0.4] (Cm3Im1I4) -- (Cm3I0I3);
\draw[opacity=0.4] (Cm4I0I4) -- (Cm3Im1I4);
\draw[opacity=0.4] (Cm4I0I4) -- (Cm4I1I3);
\draw[opacity=0.4] (C0I4Im4) -- (C1I3Im4);
\draw[opacity=0.4] (C0I4Im4) -- (Cm1I4Im3);
\draw[opacity=0.4] (C4Im3Im1) -- (C3Im3I0);
\draw[opacity=0.4] (C4Im4I0) -- (C3Im4I1);
\draw[opacity=0.4] (C4Im4I0) -- (C4Im3Im1);
\draw[fill=red] (barycentric cs:L1=-1.0,L2=2.0,L3=2.0) circle (0.2);
\draw[fill=red] (barycentric cs:L1=2.0,L2=-1.0,L3=2.0) circle (0.2);
\draw[fill=red] (barycentric cs:L1=2.0,L2=2.0,L3=-1.0) circle (0.2);
\draw[fill=blue] (barycentric cs:L1=-2.0,L2=2.0,L3=3.0) circle (0.2);
\draw[fill=green] (barycentric cs:L1=-2.0,L2=3.0,L3=2.0) circle (0.2);
\draw[fill=green] (barycentric cs:L1=2.0,L2=-2.0,L3=3.0) circle (0.2);
\draw[fill=blue] (barycentric cs:L1=2.0,L2=3.0,L3=-2.0) circle (0.2);
\draw[fill=blue] (barycentric cs:L1=3.0,L2=-2.0,L3=2.0) circle (0.2);
\draw[fill=green] (barycentric cs:L1=3.0,L2=2.0,L3=-2.0) circle (0.2);
\draw[fill=green] (barycentric cs:L1=-3.0,L2=2.0,L3=4.0) circle (0.2);
\draw[fill=green] (barycentric cs:L1=2.0,L2=4.0,L3=-3.0) circle (0.2);
\draw[fill=green] (barycentric cs:L1=4.0,L2=-3.0,L3=2.0) circle (0.2);
\draw[fill=blue] (barycentric cs:L1=0.0,L2=1.0,L3=2.0) circle (0.2);
\draw[fill=green] (barycentric cs:L1=0.0,L2=2.0,L3=1.0) circle (0.2);
\draw[fill=green] (barycentric cs:L1=1.0,L2=0.0,L3=2.0) circle (0.2);
\draw[fill=blue] (barycentric cs:L1=1.0,L2=2.0,L3=0.0) circle (0.2);
\draw[fill=blue] (barycentric cs:L1=2.0,L2=0.0,L3=1.0) circle (0.2);
\draw[fill=green] (barycentric cs:L1=2.0,L2=1.0,L3=0.0) circle (0.2);
\draw[fill=red] (barycentric cs:L1=0.0,L2=0.0,L3=3.0) circle (0.2);
\draw[fill=red] (barycentric cs:L1=0.0,L2=3.0,L3=0.0) circle (0.2);
\draw[fill=green] (barycentric cs:L1=-1.0,L2=1.0,L3=3.0) circle (0.2);
\draw[fill=blue] (barycentric cs:L1=-1.0,L2=3.0,L3=1.0) circle (0.2);
\draw[fill=blue] (barycentric cs:L1=1.0,L2=-1.0,L3=3.0) circle (0.2);
\draw[fill=green] (barycentric cs:L1=1.0,L2=3.0,L3=-1.0) circle (0.2);
\draw[fill=red] (barycentric cs:L1=3.0,L2=0.0,L3=0.0) circle (0.2);
\draw[fill=green] (barycentric cs:L1=3.0,L2=-1.0,L3=1.0) circle (0.2);
\draw[fill=blue] (barycentric cs:L1=3.0,L2=1.0,L3=-1.0) circle (0.2);
\draw[fill=green] (barycentric cs:L1=0.0,L2=-1.0,L3=4.0) circle (0.2);
\draw[fill=blue] (barycentric cs:L1=0.0,L2=4.0,L3=-1.0) circle (0.2);
\draw[fill=blue] (barycentric cs:L1=-1.0,L2=0.0,L3=4.0) circle (0.2);
\draw[fill=green] (barycentric cs:L1=-1.0,L2=4.0,L3=0.0) circle (0.2);
\draw[fill=red] (barycentric cs:L1=-2.0,L2=1.0,L3=4.0) circle (0.2);
\draw[fill=red] (barycentric cs:L1=1.0,L2=4.0,L3=-2.0) circle (0.2);
\draw[fill=green] (barycentric cs:L1=4.0,L2=0.0,L3=-1.0) circle (0.2);
\draw[fill=blue] (barycentric cs:L1=4.0,L2=-1.0,L3=0.0) circle (0.2);
\draw[fill=red] (barycentric cs:L1=4.0,L2=-2.0,L3=1.0) circle (0.2);
\draw[fill=green] (barycentric cs:L1=0.0,L2=5.0,L3=-2.0) circle (0.2);
\draw[fill=green] (barycentric cs:L1=-2.0,L2=0.0,L3=5.0) circle (0.2);
\draw[fill=blue] (barycentric cs:L1=-3.0,L2=1.0,L3=5.0) circle (0.2);
\draw[fill=blue] (barycentric cs:L1=1.0,L2=5.0,L3=-3.0) circle (0.2);
\draw[fill=green] (barycentric cs:L1=5.0,L2=-2.0,L3=0.0) circle (0.2);
\draw[fill=blue] (barycentric cs:L1=5.0,L2=-3.0,L3=1.0) circle (0.2);
\draw[fill=red] (barycentric cs:L1=1.0,L2=1.0,L3=1.0) circle (0.2);
\node[anchor=240] (a1) at (barycentric cs:L1=1.0,L2=5.0,L3=-3.0) {$a_1$};
\node[anchor=0] (b1) at (barycentric cs:L1=5.0,L2=-3.0,L3=1.0) {$b_1$};
\node[anchor=120] (c1) at (barycentric cs:L1=-3.0,L2=1.0,L3=5.0) {$c_1$};
\node[anchor=240] (a2) at (barycentric cs:L1=2.0,L2=4.0,L3=-3.0) {$a_2$};
\node[anchor=0] (b2) at (barycentric cs:L1=4.0,L2=-3.0,L3=2.0) {$b_2$};
\node[anchor=120] (c2) at (barycentric cs:L1=-3.0,L2=2.0,L3=4.0) {$c_2$};
\node[anchor=240] (a3) at (barycentric cs:L1=2.0,L2=3.0,L3=-2.0) {$a_3$};
\node[anchor=0] (b3) at (barycentric cs:L1=3.0,L2=-2.0,L3=2.0) {$b_3$};
\node[anchor=120] (c3) at (barycentric cs:L1=-2.0,L2=2.0,L3=3.0) {$c_3$};
\node[anchor=240] (a4) at (barycentric cs:L1=3.0,L2=2.0,L3=-2.0) {$a_4$};
\node[anchor=0] (b4) at (barycentric cs:L1=2.0,L2=-2.0,L3=3.0) {$b_4$};
\node[anchor=120] (c4) at (barycentric cs:L1=-2.0,L2=3.0,L3=2.0) {$c_4$};
\node[anchor=240] (a5) at (barycentric cs:L1=3.0,L2=1.0,L3=-1.0) {$a_5$};
\node[anchor=0] (b5) at (barycentric cs:L1=1.0,L2=-1.0,L3=3.0) {$b_5$};
\node[anchor=120] (c5) at (barycentric cs:L1=-1.0,L2=3.0,L3=1.0) {$c_5$};
\node[anchor=240] (a6) at (barycentric cs:L1=4.0,L2=0.0,L3=-1.0) {$a_6$};
\node[anchor=0] (b6) at (barycentric cs:L1=0.0,L2=-1.0,L3=4.0) {$b_6$};
\node[anchor=120] (c6) at (barycentric cs:L1=-1.0,L2=4.0,L3=0.0) {$c_6$};
\node[anchor=240] (a7) at (barycentric cs:L1=4.0,L2=-1.0,L3=0.0) {$a_7$};
\node[anchor=0] (b7) at (barycentric cs:L1=-1.0,L2=0.0,L3=4.0) {$b_7$};
\node[anchor=120] (c7) at (barycentric cs:L1=0.0,L2=4.0,L3=-1.0) {$c_7$};
\node[anchor=240] (a8) at (barycentric cs:L1=5.0,L2=-2.0,L3=0.0) {$a_8$};
\node[anchor=0] (b8) at (barycentric cs:L1=-2.0,L2=0.0,L3=5.0) {$b_8$};
\node[anchor=120] (c8) at (barycentric cs:L1=0.0,L2=5.0,L3=-2.0) {$c_8$};
\node[anchor=240] (a9) at (barycentric cs:L1=5.0,L2=-3.0,L3=1.0) {$a_9$};
\node[anchor=0] (b9) at (barycentric cs:L1=-3.0,L2=1.0,L3=5.0) {$b_9$};
\node[anchor=120] (c9) at (barycentric cs:L1=1.0,L2=5.0,L3=-3.0) {$c_9$};
\node[anchor=180] (v) at (barycentric cs:L1=1.0,L2=1.0,L3=1.0) {$v$};
\end{tikzpicture}}

 \caption{\label{fig:gvt} The graph $G(v,t)$ for $t=5$.}
\end{figure}
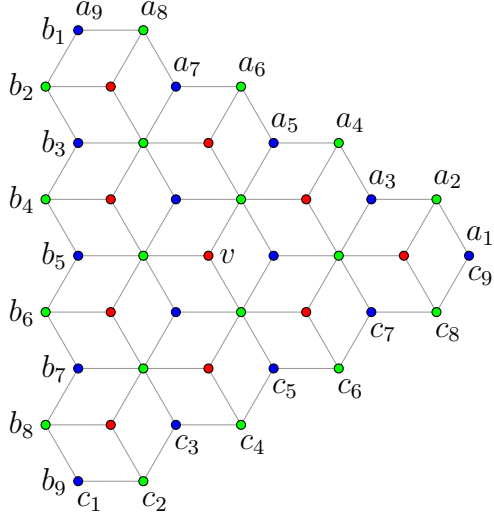

For $v=(0,0,0)$ and $t=5$, the graph $G(v,t)$ together with its labeling of the boundary is shown in Figure~\ref{fig:gvt}. 

We define $\GboundG$ (resp., $\GboundB$) to be the subsets of $\Gbound$ that consist of green (resp., blue) vertices in $\Gbound$. There are no red vertices in $\Gbound$.

% Let $F(v,t)$ be the restriction of $F^\infty(v)$ to $\TR(v,t)$. 
% It is easy to see that all the vertices on the boundary of $\TR(v,t)$ are green. Let $B(v,t)$ be the set of all these vertices, then 
% \[B(v,t)=\{v+(b,a-b,-a),v+(-a,b,a-b),v+(a-b,-a,b)\},\]
% where $0\leq a,b\leq t-1$ and $a+b=t-1$. 

\begin{definition}
A \emph{$G(v,t)$-forest} is a forest $F$ with the same vertex set as $G(v,t)$ such that for every lozenge face of $G(v,t)$, $F$ contains exactly one of its two diagonals, and there are no other edges in $F$.
%  \item two vertices $v+(x,y,z)$ and $v+(x',y',z')$ from $B(v,t)$ belong to the same connected component of $F$ if and only if there exist integers $0\leq a,b\leq t-1$ with $2b\leq t-1$ such that at least one of the following conditions is satisfied:
%  \begin{eqnarray}
%   (x,y,z)=(a-b,-a,b),\quad (x',y',z')=(a,b-a,-b);\\
%   (x,y,z)=(-a,b,a-b),\quad (x',y',z')=(b-a,-b,a);\\
%   (x,y,z)=(b,a-b,-a),\quad (x',y',z')=(-b,a,b-a).  
%  \end{eqnarray}
% \item two vertices $u$ and $w$ on the boundary of $G(v,t)$ belong to the same connected component of $F$ if and only if they belong to the same connected component of $F(v,t)$. 
% \end{enumerate}
\end{definition}

% In particular, if $t$ is even then each connected component of $F$ contains one vertex on each of two sides of $\TR(v,t)$ and no vertices on the third side of $\TR(v,t)$. If $t=2b+1$ is odd then the same holds except that there is one component of $F$ whose restriction to the boundary of $\TR(v,t)$ consists of vertices $(b,0,-b),(0,-b,b),$ and $(-b,b,0)$. 

\def\bpart{\Pi}
\def\bpartG{\Pi^{\operatorname{g}}}
\def\bpartB{\Pi^{\operatorname{b}}}

For each $G(v,t)$-forest $F$, we denote by $\bpart_F=\{B_1,\dots,B_k\}$ the non-crossing partition that $F$ induces on $\Gbound$. More precisely, the sets $B_1,\dots,B_k$ (called \emph{blocks} of $\bpart_F$) define a partition of $\Gbound$ so that two vertices $u,w\in\Gbound$ belong to the same block of $\bpart_F$ if and only if they belong to the same connected component of $F$. In particular, since every connected component of $F$ consists either entirely of green vertices or entirely of red and blue vertices, every block of $\bpart_F$ is called either \emph{green} or \emph{red-blue}. We denote by $\bpartG_F$ and $\bpartB_F$ the corresponding non-crossing partitions of $\GboundG$ and $\GboundB$. Note that the non-crossing partitions $\bpartG_F$ and $\bpartB_F$ are \emph{complementary} in the sense of Kreweras~\cite{Kreweras}. We define the non-crossing partition $\bpart_0$ of $\Gbound$ as follows: for each $1\leq i\leq t$, draw an edge between
\begin{itemize}
 \item $a_i$ and $c_{2t-i}$;
 \item $b_i$ and $a_{2t-i}$;
 \item $c_i$ and $b_{2t-i}$.
\end{itemize}
The union of these edges defines an undirected graph with vertex set $\Gbound$, and we let $\bpart_0$ be the partition of $\Gbound$ into connected components of this graph. Thus $\bpart_0$ consists of pairs of vertices together with one triple $\{a_t,b_t,c_t\}$ which is green when $t$ is even and blue when $t$ is odd.

\begin{definition}
 A $G(v,t)$-forest $F$ is called a \emph{$G(v,t)$-grove} if $\bpart_F=\bpart_0$.
\end{definition}

\begin{figure}
 \centering

\scalebox{1.0}{
\begin{tikzpicture}[scale=0.3]
\coordinate (L1) at (90:5.00);
\coordinate (L2) at (330:5.00);
\coordinate (L3) at (210:5.00);
\coordinate (Cm4I0I4) at (barycentric cs:L1=-3.0,L2=1.0,L3=5.0);
\coordinate (Cm4I1I3) at (barycentric cs:L1=-3.0,L2=2.0,L3=4.0);
\coordinate (Cm3Im1I4) at (barycentric cs:L1=-2.0,L2=0.0,L3=5.0);
\coordinate (Cm3I0I3) at (barycentric cs:L1=-2.0,L2=1.0,L3=4.0);
\coordinate (Cm3I1I2) at (barycentric cs:L1=-2.0,L2=2.0,L3=3.0);
\coordinate (Cm3I2I1) at (barycentric cs:L1=-2.0,L2=3.0,L3=2.0);
\coordinate (Cm2Im1I3) at (barycentric cs:L1=-1.0,L2=0.0,L3=4.0);
\coordinate (Cm2I0I2) at (barycentric cs:L1=-1.0,L2=1.0,L3=3.0);
\coordinate (Cm2I1I1) at (barycentric cs:L1=-1.0,L2=2.0,L3=2.0);
\coordinate (Cm2I2I0) at (barycentric cs:L1=-1.0,L2=3.0,L3=1.0);
\coordinate (Cm2I3Im1) at (barycentric cs:L1=-1.0,L2=4.0,L3=0.0);
\coordinate (Cm1Im2I3) at (barycentric cs:L1=0.0,L2=-1.0,L3=4.0);
\coordinate (Cm1Im1I2) at (barycentric cs:L1=0.0,L2=0.0,L3=3.0);
\coordinate (Cm1I0I1) at (barycentric cs:L1=0.0,L2=1.0,L3=2.0);
\coordinate (Cm1I1I0) at (barycentric cs:L1=0.0,L2=2.0,L3=1.0);
\coordinate (Cm1I2Im1) at (barycentric cs:L1=0.0,L2=3.0,L3=0.0);
\coordinate (Cm1I3Im2) at (barycentric cs:L1=0.0,L2=4.0,L3=-1.0);
\coordinate (Cm1I4Im3) at (barycentric cs:L1=0.0,L2=5.0,L3=-2.0);
\coordinate (C0Im2I2) at (barycentric cs:L1=1.0,L2=-1.0,L3=3.0);
\coordinate (C0Im1I1) at (barycentric cs:L1=1.0,L2=0.0,L3=2.0);
\coordinate (C0I0I0) at (barycentric cs:L1=1.0,L2=1.0,L3=1.0);
\coordinate (C0I1Im1) at (barycentric cs:L1=1.0,L2=2.0,L3=0.0);
\coordinate (C0I2Im2) at (barycentric cs:L1=1.0,L2=3.0,L3=-1.0);
\coordinate (C0I3Im3) at (barycentric cs:L1=1.0,L2=4.0,L3=-2.0);
\coordinate (C0I4Im4) at (barycentric cs:L1=1.0,L2=5.0,L3=-3.0);
\coordinate (C1Im3I2) at (barycentric cs:L1=2.0,L2=-2.0,L3=3.0);
\coordinate (C1Im2I1) at (barycentric cs:L1=2.0,L2=-1.0,L3=2.0);
\coordinate (C1Im1I0) at (barycentric cs:L1=2.0,L2=0.0,L3=1.0);
\coordinate (C1I0Im1) at (barycentric cs:L1=2.0,L2=1.0,L3=0.0);
\coordinate (C1I1Im2) at (barycentric cs:L1=2.0,L2=2.0,L3=-1.0);
\coordinate (C1I2Im3) at (barycentric cs:L1=2.0,L2=3.0,L3=-2.0);
\coordinate (C1I3Im4) at (barycentric cs:L1=2.0,L2=4.0,L3=-3.0);
\coordinate (C2Im3I1) at (barycentric cs:L1=3.0,L2=-2.0,L3=2.0);
\coordinate (C2Im2I0) at (barycentric cs:L1=3.0,L2=-1.0,L3=1.0);
\coordinate (C2Im1Im1) at (barycentric cs:L1=3.0,L2=0.0,L3=0.0);
\coordinate (C2I0Im2) at (barycentric cs:L1=3.0,L2=1.0,L3=-1.0);
\coordinate (C2I1Im3) at (barycentric cs:L1=3.0,L2=2.0,L3=-2.0);
\coordinate (C3Im4I1) at (barycentric cs:L1=4.0,L2=-3.0,L3=2.0);
\coordinate (C3Im3I0) at (barycentric cs:L1=4.0,L2=-2.0,L3=1.0);
\coordinate (C3Im2Im1) at (barycentric cs:L1=4.0,L2=-1.0,L3=0.0);
\coordinate (C3Im1Im2) at (barycentric cs:L1=4.0,L2=0.0,L3=-1.0);
\coordinate (C4Im4I0) at (barycentric cs:L1=5.0,L2=-3.0,L3=1.0);
\coordinate (C4Im3Im1) at (barycentric cs:L1=5.0,L2=-2.0,L3=0.0);
\draw[opacity=0.4] (Cm3I1I2) -- (Cm2I0I2);
\draw[opacity=0.4] (Cm3I1I2) -- (Cm4I1I3);
\draw[opacity=0.4] (Cm3I1I2) -- (Cm3I2I1);
\draw[opacity=0.4] (Cm3I2I1) -- (Cm2I1I1);
\draw[opacity=0.4] (C1Im3I2) -- (C1Im2I1);
\draw[opacity=0.4] (C1I2Im3) -- (C2I1Im3);
\draw[opacity=0.4] (C1I2Im3) -- (C0I2Im2);
\draw[opacity=0.4] (C1I2Im3) -- (C1I3Im4);
\draw[opacity=0.4] (C2Im3I1) -- (C3Im4I1);
\draw[opacity=0.4] (C2Im3I1) -- (C1Im3I2);
\draw[opacity=0.4] (C2Im3I1) -- (C2Im2I0);
\draw[opacity=0.4] (C2I1Im3) -- (C1I1Im2);
\draw[opacity=0.4] (Cm4I1I3) -- (Cm3I0I3);
\draw[opacity=0.4] (C1I3Im4) -- (C0I3Im3);
\draw[opacity=0.4] (C3Im4I1) -- (C3Im3I0);
\draw[opacity=0.4] (Cm1I0I1) -- (C0Im1I1);
\draw[opacity=0.4] (Cm1I0I1) -- (Cm2I0I2);
\draw[opacity=0.4] (Cm1I0I1) -- (Cm1I1I0);
\draw[opacity=0.4] (Cm1I1I0) -- (C0I0I0);
\draw[opacity=0.4] (Cm1I1I0) -- (Cm2I1I1);
\draw[opacity=0.4] (Cm1I1I0) -- (Cm1I2Im1);
\draw[opacity=0.4] (C0Im1I1) -- (C1Im2I1);
\draw[opacity=0.4] (C0Im1I1) -- (Cm1Im1I2);
\draw[opacity=0.4] (C0Im1I1) -- (C0I0I0);
\draw[opacity=0.4] (C0I1Im1) -- (C1I0Im1);
\draw[opacity=0.4] (C0I1Im1) -- (Cm1I1I0);
\draw[opacity=0.4] (C0I1Im1) -- (C0I2Im2);
\draw[opacity=0.4] (C1Im1I0) -- (C2Im2I0);
\draw[opacity=0.4] (C1Im1I0) -- (C0Im1I1);
\draw[opacity=0.4] (C1Im1I0) -- (C1I0Im1);
\draw[opacity=0.4] (C1I0Im1) -- (C2Im1Im1);
\draw[opacity=0.4] (C1I0Im1) -- (C0I0I0);
\draw[opacity=0.4] (C1I0Im1) -- (C1I1Im2);
\draw[opacity=0.4] (Cm2I0I2) -- (Cm1Im1I2);
\draw[opacity=0.4] (Cm2I0I2) -- (Cm3I0I3);
\draw[opacity=0.4] (Cm2I0I2) -- (Cm2I1I1);
\draw[opacity=0.4] (Cm2I2I0) -- (Cm1I1I0);
\draw[opacity=0.4] (Cm2I2I0) -- (Cm3I2I1);
\draw[opacity=0.4] (Cm2I2I0) -- (Cm2I3Im1);
\draw[opacity=0.4] (C0Im2I2) -- (C1Im3I2);
\draw[opacity=0.4] (C0Im2I2) -- (Cm1Im2I3);
\draw[opacity=0.4] (C0Im2I2) -- (C0Im1I1);
\draw[opacity=0.4] (C0I2Im2) -- (C1I1Im2);
\draw[opacity=0.4] (C0I2Im2) -- (Cm1I2Im1);
\draw[opacity=0.4] (C0I2Im2) -- (C0I3Im3);
\draw[opacity=0.4] (C2Im2I0) -- (C3Im3I0);
\draw[opacity=0.4] (C2Im2I0) -- (C1Im2I1);
\draw[opacity=0.4] (C2Im2I0) -- (C2Im1Im1);
\draw[opacity=0.4] (C2I0Im2) -- (C3Im1Im2);
\draw[opacity=0.4] (C2I0Im2) -- (C1I0Im1);
\draw[opacity=0.4] (C2I0Im2) -- (C2I1Im3);
\draw[opacity=0.4] (Cm1Im2I3) -- (Cm1Im1I2);
\draw[opacity=0.4] (Cm1I3Im2) -- (C0I2Im2);
\draw[opacity=0.4] (Cm1I3Im2) -- (Cm2I3Im1);
\draw[opacity=0.4] (Cm1I3Im2) -- (Cm1I4Im3);
\draw[opacity=0.4] (Cm2Im1I3) -- (Cm1Im2I3);
\draw[opacity=0.4] (Cm2Im1I3) -- (Cm3Im1I4);
\draw[opacity=0.4] (Cm2Im1I3) -- (Cm2I0I2);
\draw[opacity=0.4] (Cm2I3Im1) -- (Cm1I2Im1);
\draw[opacity=0.4] (C3Im1Im2) -- (C2Im1Im1);
\draw[opacity=0.4] (C3Im2Im1) -- (C4Im3Im1);
\draw[opacity=0.4] (C3Im2Im1) -- (C2Im2I0);
\draw[opacity=0.4] (C3Im2Im1) -- (C3Im1Im2);
\draw[opacity=0.4] (Cm1I4Im3) -- (C0I3Im3);
\draw[opacity=0.4] (Cm3Im1I4) -- (Cm3I0I3);
\draw[opacity=0.4] (Cm4I0I4) -- (Cm3Im1I4);
\draw[opacity=0.4] (Cm4I0I4) -- (Cm4I1I3);
\draw[opacity=0.4] (C0I4Im4) -- (C1I3Im4);
\draw[opacity=0.4] (C0I4Im4) -- (Cm1I4Im3);
\draw[opacity=0.4] (C4Im3Im1) -- (C3Im3I0);
\draw[opacity=0.4] (C4Im4I0) -- (C3Im4I1);
\draw[opacity=0.4] (C4Im4I0) -- (C4Im3Im1);
\draw[fill=red] (barycentric cs:L1=-1.0,L2=2.0,L3=2.0) circle (0.2);
\draw[fill=red] (barycentric cs:L1=2.0,L2=-1.0,L3=2.0) circle (0.2);
\draw[fill=red] (barycentric cs:L1=2.0,L2=2.0,L3=-1.0) circle (0.2);
\draw[fill=blue] (barycentric cs:L1=-2.0,L2=2.0,L3=3.0) circle (0.2);
\draw[fill=green] (barycentric cs:L1=-2.0,L2=3.0,L3=2.0) circle (0.2);
\draw[fill=green] (barycentric cs:L1=2.0,L2=-2.0,L3=3.0) circle (0.2);
\draw[fill=blue] (barycentric cs:L1=2.0,L2=3.0,L3=-2.0) circle (0.2);
\draw[fill=blue] (barycentric cs:L1=3.0,L2=-2.0,L3=2.0) circle (0.2);
\draw[fill=green] (barycentric cs:L1=3.0,L2=2.0,L3=-2.0) circle (0.2);
\draw[fill=green] (barycentric cs:L1=-3.0,L2=2.0,L3=4.0) circle (0.2);
\draw[fill=green] (barycentric cs:L1=2.0,L2=4.0,L3=-3.0) circle (0.2);
\draw[fill=green] (barycentric cs:L1=4.0,L2=-3.0,L3=2.0) circle (0.2);
\draw[fill=blue] (barycentric cs:L1=0.0,L2=1.0,L3=2.0) circle (0.2);
\draw[fill=green] (barycentric cs:L1=0.0,L2=2.0,L3=1.0) circle (0.2);
\draw[fill=green] (barycentric cs:L1=1.0,L2=0.0,L3=2.0) circle (0.2);
\draw[fill=blue] (barycentric cs:L1=1.0,L2=2.0,L3=0.0) circle (0.2);
\draw[fill=blue] (barycentric cs:L1=2.0,L2=0.0,L3=1.0) circle (0.2);
\draw[fill=green] (barycentric cs:L1=2.0,L2=1.0,L3=0.0) circle (0.2);
\draw[fill=red] (barycentric cs:L1=0.0,L2=0.0,L3=3.0) circle (0.2);
\draw[fill=red] (barycentric cs:L1=0.0,L2=3.0,L3=0.0) circle (0.2);
\draw[fill=green] (barycentric cs:L1=-1.0,L2=1.0,L3=3.0) circle (0.2);
\draw[fill=blue] (barycentric cs:L1=-1.0,L2=3.0,L3=1.0) circle (0.2);
\draw[fill=blue] (barycentric cs:L1=1.0,L2=-1.0,L3=3.0) circle (0.2);
\draw[fill=green] (barycentric cs:L1=1.0,L2=3.0,L3=-1.0) circle (0.2);
\draw[fill=red] (barycentric cs:L1=3.0,L2=0.0,L3=0.0) circle (0.2);
\draw[fill=green] (barycentric cs:L1=3.0,L2=-1.0,L3=1.0) circle (0.2);
\draw[fill=blue] (barycentric cs:L1=3.0,L2=1.0,L3=-1.0) circle (0.2);
\draw[fill=green] (barycentric cs:L1=0.0,L2=-1.0,L3=4.0) circle (0.2);
\draw[fill=blue] (barycentric cs:L1=0.0,L2=4.0,L3=-1.0) circle (0.2);
\draw[fill=blue] (barycentric cs:L1=-1.0,L2=0.0,L3=4.0) circle (0.2);
\draw[fill=green] (barycentric cs:L1=-1.0,L2=4.0,L3=0.0) circle (0.2);
\draw[fill=red] (barycentric cs:L1=-2.0,L2=1.0,L3=4.0) circle (0.2);
\draw[fill=red] (barycentric cs:L1=1.0,L2=4.0,L3=-2.0) circle (0.2);
\draw[fill=green] (barycentric cs:L1=4.0,L2=0.0,L3=-1.0) circle (0.2);
\draw[fill=blue] (barycentric cs:L1=4.0,L2=-1.0,L3=0.0) circle (0.2);
\draw[fill=red] (barycentric cs:L1=4.0,L2=-2.0,L3=1.0) circle (0.2);
\draw[fill=green] (barycentric cs:L1=0.0,L2=5.0,L3=-2.0) circle (0.2);
\draw[fill=green] (barycentric cs:L1=-2.0,L2=0.0,L3=5.0) circle (0.2);
\draw[fill=blue] (barycentric cs:L1=-3.0,L2=1.0,L3=5.0) circle (0.2);
\draw[fill=blue] (barycentric cs:L1=1.0,L2=5.0,L3=-3.0) circle (0.2);
\draw[fill=green] (barycentric cs:L1=5.0,L2=-2.0,L3=0.0) circle (0.2);
\draw[fill=blue] (barycentric cs:L1=5.0,L2=-3.0,L3=1.0) circle (0.2);
\draw[fill=red] (barycentric cs:L1=1.0,L2=1.0,L3=1.0) circle (0.2);
\draw[line width=0.4mm,color=black!30!green] (C1Im3I2) -- (C0Im1I1);
\draw[line width=0.4mm,color=blue] (C1I1Im2) -- (C2I0Im2);
\draw[line width=0.4mm,color=black!30!green] (C2I1Im3) -- (C0I2Im2);
\draw[line width=0.4mm,color=blue] (C0I3Im3) -- (C1I2Im3);
\draw[line width=0.4mm,color=black!30!green] (C1I3Im4) -- (Cm1I4Im3);
\draw[line width=0.4mm,color=black!30!green] (C3Im4I1) -- (C2Im2I0);
\draw[line width=0.4mm,color=black!30!green] (Cm1I1I0) -- (Cm3I2I1);
\draw[line width=0.4mm,color=black!30!green] (Cm1I1I0) -- (Cm2I0I2);
\draw[line width=0.4mm,color=blue] (Cm1I2Im1) -- (Cm2I2I0);
\draw[line width=0.4mm,color=blue] (C0I0I0) -- (Cm1I0I1);
\draw[line width=0.4mm,color=blue] (Cm1I0I1) -- (Cm1Im1I2);
\draw[line width=0.4mm,color=blue] (Cm1Im1I2) -- (C0Im2I2);
\draw[line width=0.4mm,color=blue] (C0I1Im1) -- (C0I0I0);
\draw[line width=0.4mm,color=black!30!green] (C1I0Im1) -- (C0Im1I1);
\draw[line width=0.4mm,color=blue] (C1I1Im2) -- (C0I1Im1);
\draw[line width=0.4mm,color=blue] (Cm2I1I1) -- (Cm3I1I2);
\draw[line width=0.4mm,color=blue] (Cm3I1I2) -- (Cm3I0I3);
\draw[line width=0.4mm,color=blue] (Cm3I0I3) -- (Cm2Im1I3);
\draw[line width=0.4mm,color=blue] (Cm1I2Im1) -- (C0I1Im1);
\draw[line width=0.4mm,color=black!30!green] (C0I2Im2) -- (Cm2I3Im1);
\draw[line width=0.4mm,color=blue] (C0I3Im3) -- (Cm1I3Im2);
\draw[line width=0.4mm,color=blue] (C1Im2I1) -- (C2Im3I1);
\draw[line width=0.4mm,color=blue] (C1Im1I0) -- (C1Im2I1);
\draw[line width=0.4mm,color=blue] (C2Im1Im1) -- (C1Im1I0);
\draw[line width=0.4mm,color=black!30!green] (Cm1Im2I3) -- (Cm2I0I2);
\draw[line width=0.4mm,color=black!30!green] (C3Im1Im2) -- (C1I0Im1);
\draw[line width=0.4mm,color=blue] (C2Im1Im1) -- (C3Im2Im1);
\draw[line width=0.4mm,color=black!30!green] (Cm3Im1I4) -- (Cm4I1I3);
\draw[line width=0.4mm,color=blue] (C3Im3I0) -- (C4Im4I0);
\draw[line width=0.4mm,color=black!30!green] (C4Im3Im1) -- (C2Im2I0);
\coordinate (aa1) at (barycentric cs:L1=1.4,L2=5.0,L3=-3.4);
\draw[line width=0.4mm,color=blue] (C0I4Im4) -- (aa1);
\coordinate (bb1) at (barycentric cs:L1=5.0,L2=-3.4,L3=1.4);
\draw[line width=0.4mm,color=blue] (C4Im4I0) -- (bb1);
\coordinate (cc1) at (barycentric cs:L1=-3.4,L2=1.4,L3=5.0);
\draw[line width=0.4mm,color=blue] (Cm4I0I4) -- (cc1);
\coordinate (aa2) at (barycentric cs:L1=2.4,L2=4.0,L3=-3.4);
\draw[line width=0.4mm,color=black!30!green] (C1I3Im4) -- (aa2);
\coordinate (bb2) at (barycentric cs:L1=4.0,L2=-3.4,L3=2.4);
\draw[line width=0.4mm,color=black!30!green] (C3Im4I1) -- (bb2);
\coordinate (cc2) at (barycentric cs:L1=-3.4,L2=2.4,L3=4.0);
\draw[line width=0.4mm,color=black!30!green] (Cm4I1I3) -- (cc2);
\coordinate (aa3) at (barycentric cs:L1=2.4,L2=3.0,L3=-2.4);
\draw[line width=0.4mm,color=blue] (C1I2Im3) -- (aa3);
\coordinate (bb3) at (barycentric cs:L1=3.0,L2=-2.4,L3=2.4);
\draw[line width=0.4mm,color=blue] (C2Im3I1) -- (bb3);
\coordinate (cc3) at (barycentric cs:L1=-2.4,L2=2.4,L3=3.0);
\draw[line width=0.4mm,color=blue] (Cm3I1I2) -- (cc3);
\coordinate (aa4) at (barycentric cs:L1=3.4,L2=2.0,L3=-2.4);
\draw[line width=0.4mm,color=black!30!green] (C2I1Im3) -- (aa4);
\coordinate (bb4) at (barycentric cs:L1=2.0,L2=-2.4,L3=3.4);
\draw[line width=0.4mm,color=black!30!green] (C1Im3I2) -- (bb4);
\coordinate (cc4) at (barycentric cs:L1=-2.4,L2=3.4,L3=2.0);
\draw[line width=0.4mm,color=black!30!green] (Cm3I2I1) -- (cc4);
\coordinate (aa5) at (barycentric cs:L1=3.4,L2=1.0,L3=-1.4);
\draw[line width=0.4mm,color=blue] (C2I0Im2) -- (aa5);
\coordinate (bb5) at (barycentric cs:L1=1.0,L2=-1.4,L3=3.4);
\draw[line width=0.4mm,color=blue] (C0Im2I2) -- (bb5);
\coordinate (cc5) at (barycentric cs:L1=-1.4,L2=3.4,L3=1.0);
\draw[line width=0.4mm,color=blue] (Cm2I2I0) -- (cc5);
\coordinate (aa6) at (barycentric cs:L1=4.4,L2=0.0,L3=-1.4);
\draw[line width=0.4mm,color=black!30!green] (C3Im1Im2) -- (aa6);
\coordinate (bb6) at (barycentric cs:L1=0.0,L2=-1.4,L3=4.4);
\draw[line width=0.4mm,color=black!30!green] (Cm1Im2I3) -- (bb6);
\coordinate (cc6) at (barycentric cs:L1=-1.4,L2=4.4,L3=0.0);
\draw[line width=0.4mm,color=black!30!green] (Cm2I3Im1) -- (cc6);
\coordinate (aa7) at (barycentric cs:L1=4.4,L2=-1.0,L3=-0.4);
\draw[line width=0.4mm,color=blue] (C3Im2Im1) -- (aa7);
\coordinate (bb7) at (barycentric cs:L1=-1.0,L2=-0.4,L3=4.4);
\draw[line width=0.4mm,color=blue] (Cm2Im1I3) -- (bb7);
\coordinate (cc7) at (barycentric cs:L1=-0.4,L2=4.4,L3=-1.0);
\draw[line width=0.4mm,color=blue] (Cm1I3Im2) -- (cc7);
\coordinate (aa8) at (barycentric cs:L1=5.4,L2=-2.0,L3=-0.4);
\draw[line width=0.4mm,color=black!30!green] (C4Im3Im1) -- (aa8);
\coordinate (bb8) at (barycentric cs:L1=-2.0,L2=-0.4,L3=5.4);
\draw[line width=0.4mm,color=black!30!green] (Cm3Im1I4) -- (bb8);
\coordinate (cc8) at (barycentric cs:L1=-0.4,L2=5.4,L3=-2.0);
\draw[line width=0.4mm,color=black!30!green] (Cm1I4Im3) -- (cc8);
\coordinate (aa9) at (barycentric cs:L1=5.4,L2=-3.0,L3=0.6);
\draw[line width=0.4mm,color=blue] (C4Im4I0) -- (aa9);
\coordinate (bb9) at (barycentric cs:L1=-3.0,L2=0.6,L3=5.4);
\draw[line width=0.4mm,color=blue] (Cm4I0I4) -- (bb9);
\coordinate (cc9) at (barycentric cs:L1=0.6,L2=5.4,L3=-3.0);
\draw[line width=0.4mm,color=blue] (C0I4Im4) -- (cc9);
\draw[fill=red] (barycentric cs:L1=-1.0,L2=2.0,L3=2.0) circle (0.2);
\draw[fill=red] (barycentric cs:L1=2.0,L2=-1.0,L3=2.0) circle (0.2);
\draw[fill=red] (barycentric cs:L1=2.0,L2=2.0,L3=-1.0) circle (0.2);
\draw[fill=blue] (barycentric cs:L1=-2.0,L2=2.0,L3=3.0) circle (0.2);
\draw[fill=green] (barycentric cs:L1=-2.0,L2=3.0,L3=2.0) circle (0.2);
\draw[fill=green] (barycentric cs:L1=2.0,L2=-2.0,L3=3.0) circle (0.2);
\draw[fill=blue] (barycentric cs:L1=2.0,L2=3.0,L3=-2.0) circle (0.2);
\draw[fill=blue] (barycentric cs:L1=3.0,L2=-2.0,L3=2.0) circle (0.2);
\draw[fill=green] (barycentric cs:L1=3.0,L2=2.0,L3=-2.0) circle (0.2);
\draw[fill=green] (barycentric cs:L1=-3.0,L2=2.0,L3=4.0) circle (0.2);
\draw[fill=green] (barycentric cs:L1=2.0,L2=4.0,L3=-3.0) circle (0.2);
\draw[fill=green] (barycentric cs:L1=4.0,L2=-3.0,L3=2.0) circle (0.2);
\draw[fill=blue] (barycentric cs:L1=0.0,L2=1.0,L3=2.0) circle (0.2);
\draw[fill=green] (barycentric cs:L1=0.0,L2=2.0,L3=1.0) circle (0.2);
\draw[fill=green] (barycentric cs:L1=1.0,L2=0.0,L3=2.0) circle (0.2);
\draw[fill=blue] (barycentric cs:L1=1.0,L2=2.0,L3=0.0) circle (0.2);
\draw[fill=blue] (barycentric cs:L1=2.0,L2=0.0,L3=1.0) circle (0.2);
\draw[fill=green] (barycentric cs:L1=2.0,L2=1.0,L3=0.0) circle (0.2);
\draw[fill=red] (barycentric cs:L1=0.0,L2=0.0,L3=3.0) circle (0.2);
\draw[fill=red] (barycentric cs:L1=0.0,L2=3.0,L3=0.0) circle (0.2);
\draw[fill=green] (barycentric cs:L1=-1.0,L2=1.0,L3=3.0) circle (0.2);
\draw[fill=blue] (barycentric cs:L1=-1.0,L2=3.0,L3=1.0) circle (0.2);
\draw[fill=blue] (barycentric cs:L1=1.0,L2=-1.0,L3=3.0) circle (0.2);
\draw[fill=green] (barycentric cs:L1=1.0,L2=3.0,L3=-1.0) circle (0.2);
\draw[fill=red] (barycentric cs:L1=3.0,L2=0.0,L3=0.0) circle (0.2);
\draw[fill=green] (barycentric cs:L1=3.0,L2=-1.0,L3=1.0) circle (0.2);
\draw[fill=blue] (barycentric cs:L1=3.0,L2=1.0,L3=-1.0) circle (0.2);
\draw[fill=green] (barycentric cs:L1=0.0,L2=-1.0,L3=4.0) circle (0.2);
\draw[fill=blue] (barycentric cs:L1=0.0,L2=4.0,L3=-1.0) circle (0.2);
\draw[fill=blue] (barycentric cs:L1=-1.0,L2=0.0,L3=4.0) circle (0.2);
\draw[fill=green] (barycentric cs:L1=-1.0,L2=4.0,L3=0.0) circle (0.2);
\draw[fill=red] (barycentric cs:L1=-2.0,L2=1.0,L3=4.0) circle (0.2);
\draw[fill=red] (barycentric cs:L1=1.0,L2=4.0,L3=-2.0) circle (0.2);
\draw[fill=green] (barycentric cs:L1=4.0,L2=0.0,L3=-1.0) circle (0.2);
\draw[fill=blue] (barycentric cs:L1=4.0,L2=-1.0,L3=0.0) circle (0.2);
\draw[fill=red] (barycentric cs:L1=4.0,L2=-2.0,L3=1.0) circle (0.2);
\draw[fill=green] (barycentric cs:L1=0.0,L2=5.0,L3=-2.0) circle (0.2);
\draw[fill=green] (barycentric cs:L1=-2.0,L2=0.0,L3=5.0) circle (0.2);
\draw[fill=blue] (barycentric cs:L1=-3.0,L2=1.0,L3=5.0) circle (0.2);
\draw[fill=blue] (barycentric cs:L1=1.0,L2=5.0,L3=-3.0) circle (0.2);
\draw[fill=green] (barycentric cs:L1=5.0,L2=-2.0,L3=0.0) circle (0.2);
\draw[fill=blue] (barycentric cs:L1=5.0,L2=-3.0,L3=1.0) circle (0.2);
\draw[fill=red] (barycentric cs:L1=1.0,L2=1.0,L3=1.0) circle (0.2);
\end{tikzpicture}}

 \caption{\label{fig:gvt_grove} A $G(v,t)$-grove.}
\end{figure}

An example of a $G(v,t)$-grove for $v=(0,0,0)$ and $t=5$ is given in Figure~\ref{fig:gvt_grove}. 

Given a $G(v,t)$-grove $F$, we define its \emph{weight} as follows:
\begin{equation}\label{eq:weight_s}
 \wt(F)=\prod_{u\in \Vert(v,t)} x_u^{\deg_F(u)-2+s(u)},
\end{equation}
where $s(u)\in\{0,1,2\}$ is equal to $0$ for non-boundary vertices of $G(v,t)$ and for $u\in\Gbound$, we let
\[s(a_1)=s(b_1)=s(c_1)=2,\quad s(a_i)=s(b_i)=s(c_i)=1,\quad 2\leq i\leq 2t-2.\]
It is convenient to draw $s(u)$ external half-edges from each boundary vertex $u$ as we did in Figure~\ref{fig:gvt_grove}.

% according to the position of $u$ in the triangle $\TR(v,t)$. To describe it carefully, we need to introduce a smaller triangle $\TR(v,t-1)$ which is a convex hull of $v+(t-2)e_{12},v+(t-2)e_{23},$ and $v+(t-2)e_{31}$. All the boundary vertices of $\TR(v,t-1)$ are blue. We put $s(u)$ to be equal to the combined number of sides of $\TR(v,t)$ and $\TR(v,t-1)$ that $u$ belongs to. In other words,
% \[s(u)=\begin{cases}
%         2,&\text{ if $u$ is a vertex of $\TR(v,t)$ or of $\TR(v,t-1)$};\\
%         1,&\text{ if $u$ belongs to the interior of a side of $\TR(v,t)$ or of $\TR(v,t-1)$};\\
%         0,&\text{ otherwise, that is, if $u$ belongs to the interior of $\TR(v,t-1)$}.
%        \end{cases}\]
% We imagine that $F$ is an infinite grove that is ``regular'' outside of $\TR(v,t)$ in the sense that all of its edges point ``away from $v$'', and then $s(u)$ is the number of extra edges of such an infinite grove that do not belong to $\TR(v,t)$. We depict these edges as short blue segments in Figure~TODO. 

We are ready to state the formula due to Carroll and Speyer:
\begin{theorem}[\cite{CS}]\label{thm:CS}
 For $v\in\TP_0$ and $t\geq 2$ such that $t+1\equiv \e_v\pmod 3$, we have 
 \begin{equation}\label{eq:CS}
 \Cube_v(t+1)=\sum_F \wt(F),
 \end{equation}
 where the sum is taken over all $G(v,t)$-groves $F$.
\end{theorem}

\begin{figure}
 \centering

\begin{tabular}{ccc}
\scalebox{1.0}{
\begin{tikzpicture}[scale=0.3]
\coordinate (L1) at (90:5.00);
\coordinate (L2) at (330:5.00);
\coordinate (L3) at (210:5.00);
\coordinate (Cm1I0I1) at (barycentric cs:L1=0.0,L2=1.0,L3=2.0);
\coordinate (Cm1I1I0) at (barycentric cs:L1=0.0,L2=2.0,L3=1.0);
\coordinate (C0Im1I1) at (barycentric cs:L1=1.0,L2=0.0,L3=2.0);
\coordinate (C0I0I0) at (barycentric cs:L1=1.0,L2=1.0,L3=1.0);
\coordinate (C0I1Im1) at (barycentric cs:L1=1.0,L2=2.0,L3=0.0);
\coordinate (C1Im1I0) at (barycentric cs:L1=2.0,L2=0.0,L3=1.0);
\coordinate (C1I0Im1) at (barycentric cs:L1=2.0,L2=1.0,L3=0.0);
\draw[opacity=0.4] (Cm1I0I1) -- (C0Im1I1);
\draw[opacity=0.4] (Cm1I0I1) -- (Cm1I1I0);
\draw[opacity=0.4] (Cm1I1I0) -- (C0I0I0);
\draw[opacity=0.4] (C0Im1I1) -- (C0I0I0);
\draw[opacity=0.4] (C0I1Im1) -- (C1I0Im1);
\draw[opacity=0.4] (C0I1Im1) -- (Cm1I1I0);
\draw[opacity=0.4] (C1Im1I0) -- (C0Im1I1);
\draw[opacity=0.4] (C1Im1I0) -- (C1I0Im1);
\draw[opacity=0.4] (C1I0Im1) -- (C0I0I0);
\draw[fill=blue] (barycentric cs:L1=0.0,L2=1.0,L3=2.0) circle (0.2);
\draw[fill=green] (barycentric cs:L1=0.0,L2=2.0,L3=1.0) circle (0.2);
\draw[fill=green] (barycentric cs:L1=1.0,L2=0.0,L3=2.0) circle (0.2);
\draw[fill=blue] (barycentric cs:L1=1.0,L2=2.0,L3=0.0) circle (0.2);
\draw[fill=blue] (barycentric cs:L1=2.0,L2=0.0,L3=1.0) circle (0.2);
\draw[fill=green] (barycentric cs:L1=2.0,L2=1.0,L3=0.0) circle (0.2);
\draw[fill=red] (barycentric cs:L1=1.0,L2=1.0,L3=1.0) circle (0.2);
\draw[line width=0.4mm,color=black!30!green] (C0Im1I1) -- (Cm1I1I0);
\draw[line width=0.4mm,color=blue] (C0I1Im1) -- (C0I0I0);
\draw[line width=0.4mm,color=black!30!green] (C1I0Im1) -- (C0Im1I1);
\coordinate (aa1) at (barycentric cs:L1=1.4,L2=2.0,L3=-0.4);
\draw[line width=0.4mm,color=blue] (C0I1Im1) -- (aa1);
\coordinate (bb1) at (barycentric cs:L1=2.0,L2=-0.4,L3=1.4);
\draw[line width=0.4mm,color=blue] (C1Im1I0) -- (bb1);
\coordinate (cc1) at (barycentric cs:L1=-0.4,L2=1.4,L3=2.0);
\draw[line width=0.4mm,color=blue] (Cm1I0I1) -- (cc1);
\coordinate (aa2) at (barycentric cs:L1=2.4,L2=1.0,L3=-0.4);
\draw[line width=0.4mm,color=black!30!green] (C1I0Im1) -- (aa2);
\coordinate (bb2) at (barycentric cs:L1=1.0,L2=-0.4,L3=2.4);
\draw[line width=0.4mm,color=black!30!green] (C0Im1I1) -- (bb2);
\coordinate (cc2) at (barycentric cs:L1=-0.4,L2=2.4,L3=1.0);
\draw[line width=0.4mm,color=black!30!green] (Cm1I1I0) -- (cc2);
\coordinate (aa3) at (barycentric cs:L1=2.4,L2=0.0,L3=0.6);
\draw[line width=0.4mm,color=blue] (C1Im1I0) -- (aa3);
\coordinate (bb3) at (barycentric cs:L1=0.0,L2=0.6,L3=2.4);
\draw[line width=0.4mm,color=blue] (Cm1I0I1) -- (bb3);
\coordinate (cc3) at (barycentric cs:L1=0.6,L2=2.4,L3=0.0);
\draw[line width=0.4mm,color=blue] (C0I1Im1) -- (cc3);
\coordinate (avar) at (0:3.10);
\node[scale=0.7,anchor=180] (anode) at (avar) {$a$};
\coordinate (bvar) at (60:3.10);
\node[scale=0.7,anchor=240] (bnode) at (bvar) {$b$};
\coordinate (cvar) at (120:3.10);
\node[scale=0.7,anchor=300] (cnode) at (cvar) {$c$};
\coordinate (dvar) at (180:3.10);
\node[scale=0.7,anchor=360] (dnode) at (dvar) {$d$};
\coordinate (evar) at (240:3.10);
\node[scale=0.7,anchor=420] (enode) at (evar) {$e$};
\coordinate (fvar) at (300:3.10);
\node[scale=0.7,anchor=480] (fnode) at (fvar) {$f$};
\node[scale=0.7,anchor=west] (vnode) at (0.00,0.00) {$v$};
\node[anchor=north] (math) at (0.00,-4.00) {$\wt(F)=\frac{x_ax_d}{x_v}$};
\draw[fill=blue] (barycentric cs:L1=0.0,L2=1.0,L3=2.0) circle (0.2);
\draw[fill=green] (barycentric cs:L1=0.0,L2=2.0,L3=1.0) circle (0.2);
\draw[fill=green] (barycentric cs:L1=1.0,L2=0.0,L3=2.0) circle (0.2);
\draw[fill=blue] (barycentric cs:L1=1.0,L2=2.0,L3=0.0) circle (0.2);
\draw[fill=blue] (barycentric cs:L1=2.0,L2=0.0,L3=1.0) circle (0.2);
\draw[fill=green] (barycentric cs:L1=2.0,L2=1.0,L3=0.0) circle (0.2);
\draw[fill=red] (barycentric cs:L1=1.0,L2=1.0,L3=1.0) circle (0.2);
\end{tikzpicture}}
&
\scalebox{1.0}{
\begin{tikzpicture}[scale=0.3]
\coordinate (L1) at (90:5.00);
\coordinate (L2) at (330:5.00);
\coordinate (L3) at (210:5.00);
\coordinate (Cm1I0I1) at (barycentric cs:L1=0.0,L2=1.0,L3=2.0);
\coordinate (Cm1I1I0) at (barycentric cs:L1=0.0,L2=2.0,L3=1.0);
\coordinate (C0Im1I1) at (barycentric cs:L1=1.0,L2=0.0,L3=2.0);
\coordinate (C0I0I0) at (barycentric cs:L1=1.0,L2=1.0,L3=1.0);
\coordinate (C0I1Im1) at (barycentric cs:L1=1.0,L2=2.0,L3=0.0);
\coordinate (C1Im1I0) at (barycentric cs:L1=2.0,L2=0.0,L3=1.0);
\coordinate (C1I0Im1) at (barycentric cs:L1=2.0,L2=1.0,L3=0.0);
\draw[opacity=0.4] (Cm1I0I1) -- (C0Im1I1);
\draw[opacity=0.4] (Cm1I0I1) -- (Cm1I1I0);
\draw[opacity=0.4] (Cm1I1I0) -- (C0I0I0);
\draw[opacity=0.4] (C0Im1I1) -- (C0I0I0);
\draw[opacity=0.4] (C0I1Im1) -- (C1I0Im1);
\draw[opacity=0.4] (C0I1Im1) -- (Cm1I1I0);
\draw[opacity=0.4] (C1Im1I0) -- (C0Im1I1);
\draw[opacity=0.4] (C1Im1I0) -- (C1I0Im1);
\draw[opacity=0.4] (C1I0Im1) -- (C0I0I0);
\draw[fill=blue] (barycentric cs:L1=0.0,L2=1.0,L3=2.0) circle (0.2);
\draw[fill=green] (barycentric cs:L1=0.0,L2=2.0,L3=1.0) circle (0.2);
\draw[fill=green] (barycentric cs:L1=1.0,L2=0.0,L3=2.0) circle (0.2);
\draw[fill=blue] (barycentric cs:L1=1.0,L2=2.0,L3=0.0) circle (0.2);
\draw[fill=blue] (barycentric cs:L1=2.0,L2=0.0,L3=1.0) circle (0.2);
\draw[fill=green] (barycentric cs:L1=2.0,L2=1.0,L3=0.0) circle (0.2);
\draw[fill=red] (barycentric cs:L1=1.0,L2=1.0,L3=1.0) circle (0.2);
\draw[line width=0.4mm,color=blue] (C0I0I0) -- (Cm1I0I1);
\draw[line width=0.4mm,color=black!30!green] (C1I0Im1) -- (Cm1I1I0);
\draw[line width=0.4mm,color=black!30!green] (C1I0Im1) -- (C0Im1I1);
\coordinate (aa1) at (barycentric cs:L1=1.4,L2=2.0,L3=-0.4);
\draw[line width=0.4mm,color=blue] (C0I1Im1) -- (aa1);
\coordinate (bb1) at (barycentric cs:L1=2.0,L2=-0.4,L3=1.4);
\draw[line width=0.4mm,color=blue] (C1Im1I0) -- (bb1);
\coordinate (cc1) at (barycentric cs:L1=-0.4,L2=1.4,L3=2.0);
\draw[line width=0.4mm,color=blue] (Cm1I0I1) -- (cc1);
\coordinate (aa2) at (barycentric cs:L1=2.4,L2=1.0,L3=-0.4);
\draw[line width=0.4mm,color=black!30!green] (C1I0Im1) -- (aa2);
\coordinate (bb2) at (barycentric cs:L1=1.0,L2=-0.4,L3=2.4);
\draw[line width=0.4mm,color=black!30!green] (C0Im1I1) -- (bb2);
\coordinate (cc2) at (barycentric cs:L1=-0.4,L2=2.4,L3=1.0);
\draw[line width=0.4mm,color=black!30!green] (Cm1I1I0) -- (cc2);
\coordinate (aa3) at (barycentric cs:L1=2.4,L2=0.0,L3=0.6);
\draw[line width=0.4mm,color=blue] (C1Im1I0) -- (aa3);
\coordinate (bb3) at (barycentric cs:L1=0.0,L2=0.6,L3=2.4);
\draw[line width=0.4mm,color=blue] (Cm1I0I1) -- (bb3);
\coordinate (cc3) at (barycentric cs:L1=0.6,L2=2.4,L3=0.0);
\draw[line width=0.4mm,color=blue] (C0I1Im1) -- (cc3);
\coordinate (avar) at (0:3.10);
\node[scale=0.7,anchor=180] (anode) at (avar) {$a$};
\coordinate (bvar) at (60:3.10);
\node[scale=0.7,anchor=240] (bnode) at (bvar) {$b$};
\coordinate (cvar) at (120:3.10);
\node[scale=0.7,anchor=300] (cnode) at (cvar) {$c$};
\coordinate (dvar) at (180:3.10);
\node[scale=0.7,anchor=360] (dnode) at (dvar) {$d$};
\coordinate (evar) at (240:3.10);
\node[scale=0.7,anchor=420] (enode) at (evar) {$e$};
\coordinate (fvar) at (300:3.10);
\node[scale=0.7,anchor=480] (fnode) at (fvar) {$f$};
\node[scale=0.7,anchor=west] (vnode) at (0.00,0.00) {$v$};
\node[anchor=north] (math) at (0.00,-4.00) {$\wt(F)=\frac{x_bx_e}{x_v}$};
\draw[fill=blue] (barycentric cs:L1=0.0,L2=1.0,L3=2.0) circle (0.2);
\draw[fill=green] (barycentric cs:L1=0.0,L2=2.0,L3=1.0) circle (0.2);
\draw[fill=green] (barycentric cs:L1=1.0,L2=0.0,L3=2.0) circle (0.2);
\draw[fill=blue] (barycentric cs:L1=1.0,L2=2.0,L3=0.0) circle (0.2);
\draw[fill=blue] (barycentric cs:L1=2.0,L2=0.0,L3=1.0) circle (0.2);
\draw[fill=green] (barycentric cs:L1=2.0,L2=1.0,L3=0.0) circle (0.2);
\draw[fill=red] (barycentric cs:L1=1.0,L2=1.0,L3=1.0) circle (0.2);
\end{tikzpicture}}
&
\scalebox{1.0}{
\begin{tikzpicture}[scale=0.3]
\coordinate (L1) at (90:5.00);
\coordinate (L2) at (330:5.00);
\coordinate (L3) at (210:5.00);
\coordinate (Cm1I0I1) at (barycentric cs:L1=0.0,L2=1.0,L3=2.0);
\coordinate (Cm1I1I0) at (barycentric cs:L1=0.0,L2=2.0,L3=1.0);
\coordinate (C0Im1I1) at (barycentric cs:L1=1.0,L2=0.0,L3=2.0);
\coordinate (C0I0I0) at (barycentric cs:L1=1.0,L2=1.0,L3=1.0);
\coordinate (C0I1Im1) at (barycentric cs:L1=1.0,L2=2.0,L3=0.0);
\coordinate (C1Im1I0) at (barycentric cs:L1=2.0,L2=0.0,L3=1.0);
\coordinate (C1I0Im1) at (barycentric cs:L1=2.0,L2=1.0,L3=0.0);
\draw[opacity=0.4] (Cm1I0I1) -- (C0Im1I1);
\draw[opacity=0.4] (Cm1I0I1) -- (Cm1I1I0);
\draw[opacity=0.4] (Cm1I1I0) -- (C0I0I0);
\draw[opacity=0.4] (C0Im1I1) -- (C0I0I0);
\draw[opacity=0.4] (C0I1Im1) -- (C1I0Im1);
\draw[opacity=0.4] (C0I1Im1) -- (Cm1I1I0);
\draw[opacity=0.4] (C1Im1I0) -- (C0Im1I1);
\draw[opacity=0.4] (C1Im1I0) -- (C1I0Im1);
\draw[opacity=0.4] (C1I0Im1) -- (C0I0I0);
\draw[fill=blue] (barycentric cs:L1=0.0,L2=1.0,L3=2.0) circle (0.2);
\draw[fill=green] (barycentric cs:L1=0.0,L2=2.0,L3=1.0) circle (0.2);
\draw[fill=green] (barycentric cs:L1=1.0,L2=0.0,L3=2.0) circle (0.2);
\draw[fill=blue] (barycentric cs:L1=1.0,L2=2.0,L3=0.0) circle (0.2);
\draw[fill=blue] (barycentric cs:L1=2.0,L2=0.0,L3=1.0) circle (0.2);
\draw[fill=green] (barycentric cs:L1=2.0,L2=1.0,L3=0.0) circle (0.2);
\draw[fill=red] (barycentric cs:L1=1.0,L2=1.0,L3=1.0) circle (0.2);
\draw[line width=0.4mm,color=black!30!green] (C0Im1I1) -- (Cm1I1I0);
\draw[line width=0.4mm,color=black!30!green] (C1I0Im1) -- (Cm1I1I0);
\draw[line width=0.4mm,color=blue] (C0I0I0) -- (C1Im1I0);
\coordinate (aa1) at (barycentric cs:L1=1.4,L2=2.0,L3=-0.4);
\draw[line width=0.4mm,color=blue] (C0I1Im1) -- (aa1);
\coordinate (bb1) at (barycentric cs:L1=2.0,L2=-0.4,L3=1.4);
\draw[line width=0.4mm,color=blue] (C1Im1I0) -- (bb1);
\coordinate (cc1) at (barycentric cs:L1=-0.4,L2=1.4,L3=2.0);
\draw[line width=0.4mm,color=blue] (Cm1I0I1) -- (cc1);
\coordinate (aa2) at (barycentric cs:L1=2.4,L2=1.0,L3=-0.4);
\draw[line width=0.4mm,color=black!30!green] (C1I0Im1) -- (aa2);
\coordinate (bb2) at (barycentric cs:L1=1.0,L2=-0.4,L3=2.4);
\draw[line width=0.4mm,color=black!30!green] (C0Im1I1) -- (bb2);
\coordinate (cc2) at (barycentric cs:L1=-0.4,L2=2.4,L3=1.0);
\draw[line width=0.4mm,color=black!30!green] (Cm1I1I0) -- (cc2);
\coordinate (aa3) at (barycentric cs:L1=2.4,L2=0.0,L3=0.6);
\draw[line width=0.4mm,color=blue] (C1Im1I0) -- (aa3);
\coordinate (bb3) at (barycentric cs:L1=0.0,L2=0.6,L3=2.4);
\draw[line width=0.4mm,color=blue] (Cm1I0I1) -- (bb3);
\coordinate (cc3) at (barycentric cs:L1=0.6,L2=2.4,L3=0.0);
\draw[line width=0.4mm,color=blue] (C0I1Im1) -- (cc3);
\coordinate (avar) at (0:3.10);
\node[scale=0.7,anchor=180] (anode) at (avar) {$a$};
\coordinate (bvar) at (60:3.10);
\node[scale=0.7,anchor=240] (bnode) at (bvar) {$b$};
\coordinate (cvar) at (120:3.10);
\node[scale=0.7,anchor=300] (cnode) at (cvar) {$c$};
\coordinate (dvar) at (180:3.10);
\node[scale=0.7,anchor=360] (dnode) at (dvar) {$d$};
\coordinate (evar) at (240:3.10);
\node[scale=0.7,anchor=420] (enode) at (evar) {$e$};
\coordinate (fvar) at (300:3.10);
\node[scale=0.7,anchor=480] (fnode) at (fvar) {$f$};
\node[scale=0.7,anchor=west] (vnode) at (0.00,0.00) {$v$};
\node[anchor=north] (math) at (0.00,-4.00) {$\wt(F)=\frac{x_cx_f}{x_v}$};
\draw[fill=blue] (barycentric cs:L1=0.0,L2=1.0,L3=2.0) circle (0.2);
\draw[fill=green] (barycentric cs:L1=0.0,L2=2.0,L3=1.0) circle (0.2);
\draw[fill=green] (barycentric cs:L1=1.0,L2=0.0,L3=2.0) circle (0.2);
\draw[fill=blue] (barycentric cs:L1=1.0,L2=2.0,L3=0.0) circle (0.2);
\draw[fill=blue] (barycentric cs:L1=2.0,L2=0.0,L3=1.0) circle (0.2);
\draw[fill=green] (barycentric cs:L1=2.0,L2=1.0,L3=0.0) circle (0.2);
\draw[fill=red] (barycentric cs:L1=1.0,L2=1.0,L3=1.0) circle (0.2);
\end{tikzpicture}}
\\

\end{tabular}

 \caption{\label{fig:gv3} All three $G(v,2)$-groves.}
\end{figure}
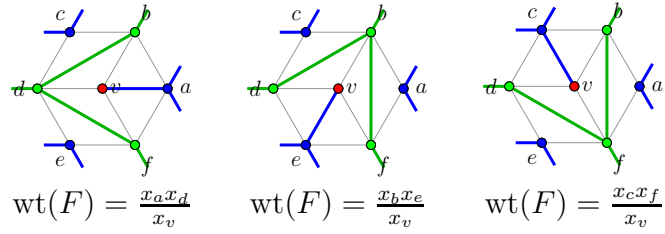

\begin{example}
 Let $v=(0,0,0)$ be the origin and let $a,b,c,d,e,f$ be its neighbors in $\TP_0$ in counterclockwise order. Then 
 \[\Cube_v(3)=\frac{x_ax_d+x_bx_e+x_cx_f}{x_v}.\]
 The graph $G(v,2)$ and the corresponding three $G(v,2)$-groves with their weights can be found in Figure~\ref{fig:gv3}.
\end{example}

\begin{proof}[Proof of Theorem~\ref{thm:entropy}]
	Note that for each $u,v\in\Torus$ and $t+1\equiv \e_v\pmod3$, the degree of $x_u$ in $\Cubetorus_v(t+1)$ grows as some constant multiple of the number of vertices in $G(v,t)$ that are equivalent to $u$ modulo $\Z A+\Z B$. The latter grows as a constant multiple of the total number of vertices in $G(v,t)$ which grows quadratically.
\end{proof}

\def\Nvt{\Nt_{(v,t)}}

\subsection{Cylindrical networks}\label{sect:networks}
In this section, we recall some of our definitions and results on \emph{cylindrical networks} from~\cite{Networks}.

Consider an acyclic directed graph $\Nt$ embedded in some horizontal strip $\TS\subset \R^2$ in the plane such that its vertices $\Vt$ and edges $\Et$ are invariant with respect to the shift by some horizontal vector $\g$. Suppose in addition that we are given a shift-invariant function $\wt:\Et\to K$ assigning weights from some field $K$ to the edges of $\Nt$. We call such a weighted directed graph \emph{a cylindrical network} if the degrees of vertices in $\Nt$ are bounded and if for every directed path in $\Nt$ connecting a vertex $\vt\in\Vt$ to some vertex $\vt+\ell\g\in\Vt$, we have $\ell>0$. We also define in an obvious way the \emph{projection} $N$ of $\Nt$ to the cylinder $\TC=\TS/\Z\g$. Thus $N$ is a weighted directed graph drawn in the cylinder. We state our results from~\cite{Networks} for the case when $\Nt$ is a \emph{planar} graph and $N$ is drawn in the cylinder $\TC$ without self-intersections. In this case, we say that $\Nt$ is a \emph{planar cylindrical network}.

\begin{definition}
	An \emph{$r$-vertex} $\vbft=(\vt_1,\dots,\vt_r)$ in $\Nt$ is an $r$-tuple of vertices of $\Nt$. An \emph{$r$-path} $\Pbft=(\Pt_1,\dots,\Pt_r)$ is an $r$-tuple of directed paths in $\Nt$ that are pairwise vertex disjoint, and we set $\wt(\Pbft)=\wt(\Pt_1)\cdots\wt(\Pt_r)$ where the weight of a path is the product of weights of its edges. If for $1\leq i\leq r$, the path $\Pt_i$ starts at $\ut_i$ and ends at $\vt_i$ then $\ubft=(\ut_1,\dots,\ut_r)$ and $\vbft=(\vt_1,\dots,\vt_r)$ are called the \emph{start} and the \emph{end} of $\Pbft$. We denote by $\pathst(\ubft,\vbft)$ the collection of all $r$-paths in $\Nt$ that start at $\ubft$ and end at $\vbft$, and we set 
	\[\ee(\ubft,\vbft):=\sum_{\Pbft\in\pathst(\ubft,\vbft)} \wt(\Pt).\] 
	
	An $r$-cycle $\Cbf=(C_1,\dots,C_r)$ in $N$ is an $r$-tuple of pairwise vertex disjoint simple directed cycles in $N$. We set $\wt(\Cbf)=\wt(C_1)\cdots\wt(C_r)$. The set of all $r$-cycles in $N$ is denoted by $\Cyc^r(N)$.
\end{definition}

Given an $r$-vertex $\vbft=(\vt_1,\dots,\vt_r)$ and a permutation $\sigma\in\Sfr_r$ of $[r]$, we denote by $\sigma\vbft=(\vt_{\sigma(1)},\dots,\vt_{\sigma(r)})$ the action of $\sigma$ on $\vbft$. We say that two $r$-vertices $\ubft$ and $\vbft$ of $\Nt$ are \emph{non-permutable} if $\pathst(\ubft,\sigma\vbft)$ is empty unless $\sigma$ is the identity permutation.

For a given planar cylindrical network $\Nt$ we define the polynomial $Q_N(t)$ as follows:
  \begin{equation}\label{eq:solenoids_intro}
  Q_N(t)=\sum_{r=0}^d(-t)^{d-r}\sum_{\Cbf\in\Cyc^r(N)} \wt(\Cbf).
  \end{equation}
Here the degree $d$ of $Q_N(t)$ is the maximum integer $r$ such that $\Cyc^r(N)$ is not empty. Recall that for each $1\leq r\leq d$, the polynomial $Q^\plee{r}_N(t)$ of degree $d\choose r$ is given by~\eqref{eq:plee}. For example, $Q^\plee{1}_N(t)=Q_N(t)$ and $Q^\plee{d}_N(t)=t-\alpha_d$, where $\alpha_d$ denotes the constant term of $Q_N(t)$.
  
\begin{theorem}[{\cite[Theorem~2.3(2)]{Networks}}]\label{thm:planar}
	Let $\Nt$ be a planar cylindrical network and let $\ubft=(\ut_1,\dots,\ut_r)$ and $\vbft=(\vt_1,\dots,\vt_r)$ be two non-permutable $r$-vertices in $\Nt$. For $\ell\geq 0$, let $\vbft_\ell=\vbft+\ell\g=(\vt_1+\ell\g,\dots,\vt_r+\ell\g)$. Define the sequence $f:\N\to K$ by 
	\[f(\ell)=\ee(\ubft,\vbft_\ell).\]
	Then the sequence $f$ satisfies a linear recurrence with characteristic polynomial $Q_N^\plee{r}(t)$ for all sufficiently large $\ell$.
\end{theorem}

\subsection{A bijection between forests and $r$-paths}\label{sect:bijection}
We define a network $\Nvt$ to be a certain weighted directed graph. Its vertex set will be $\Vert(v,t)$ together with all centers of lozenges of $G(v,t)$. For every lozenge $L$ of $G(v,t)$ with vertices $a_L,b_L,c_L,d_L$ and center $e_L$ as in Figure~\ref{fig:nvt}, we introduce four weighted directed edges of $\Nvt$ 
$a_L\to e_L, b_L\to e_L, c_L\to e_L, e_L\to d_L$ with respective weights $\alpha,1,\gamma,$ and $1$. We set
\[\alpha=\gamma=\frac{x_{a_L}x_{c_L}}{x_{b_L}x_{d_L}}.\] This defines the network $\Nvt$. Note that $\Nvt$ is acyclic.

\begin{figure}
 \centering

\scalebox{1.0}{
}

 \caption{\label{fig:nvt} The rules for constructing $\Nvt$.}
\end{figure}

\def\root{s}
\begin{definition}
 A \emph{rooted $G(v,t)$-forest} is a $G(v,t)$-forest $F$ together with a choice of a boundary vertex $\root_C$ for each connected component $C$ of $F$ called its \emph{root}.
\end{definition}

\def\bij{\phi}
\def\bijinv{\bij^{-1}}

\begin{figure}
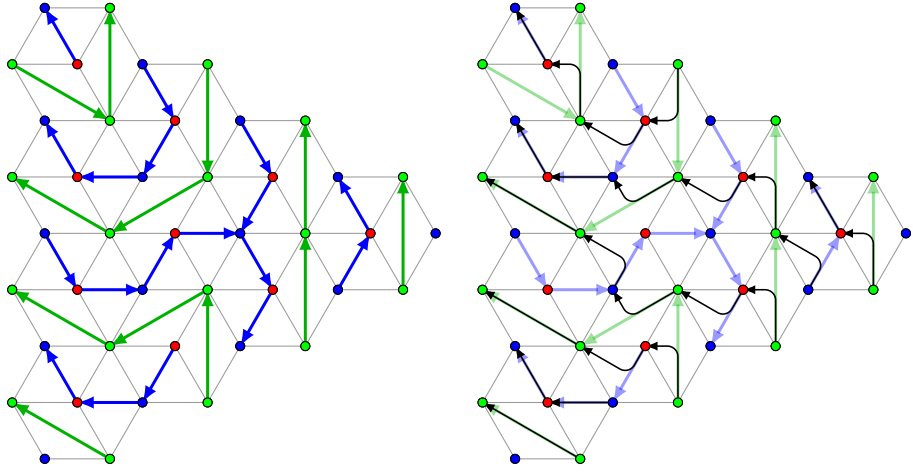

 \centering

%

 \caption{\label{fig:gvt_rooted} A $G(v,t)$-grove $F$ from Figure~\ref{fig:gvt} rooted in an arbitrary way (left). The boundary $7$-path $\bij(F)$ corresponding to $F$ (right). When $\bij(F)$ contains edges $u\to e\to v$ in $\Nvt$ and $e$ is a center of some lozenge then we draw a black edge with a rounded corner from $u$ to $v$ instead.}
\end{figure}

We view every rooted $G(v,t)$-forest $F$ as an oriented graph: we orient every edge of $F$ towards the root $\root_C$ of the corresponding connected component $C$ of $F$ that this edge belongs to. An example of a rooted $G(v,t)$-forest is given in Figure~\ref{fig:gvt_rooted} (left). Its underlying undirected graph is the $G(v,t)$-grove in Figure~\ref{fig:gvt_grove}, and the edges of each connected component $C$ point towards an arbitrarily chosen boundary vertex $\root_C$.

% Recall that a \emph{$r$-vertex} in $\Nvt$ is an $r$-tuple $\ubft=(u_1,u_2,\dots,u_r)$ of vertices of $\Nvt$.
\begin{definition}
  A \emph{boundary $r$-vertex} is an $r$-vertex $\ubft=(\ut_1,\ut_2,\dots,\ut_r)$ such that $\ut_i\in\Gbound$ for all $1\leq i\leq r$.
\end{definition}

% 
% \begin{definition}
%  A \emph{$k$-path} in $\Nvt$ is a $k$-tuple $P=(p_1,p_2,\dots,p_k)$ of paths in $\Nvt$ that are pairwise vertex disjoint. Each path $p_i$ is required to start at some boundary vertex $u_i\in\Gbound$ and end at some boundary vertex $w_i\in\Gbound$, and the $k$-vertices $U=(u_1,\dots,u_k)$ and $W=(w_1,\dots,w_k)$ are called the \emph{start} of $P$ and the \emph{end} of $P$ respectively. 
% \end{definition}

We restrict our attention to only those $r$-paths $\Pbft$ in $\Nvt$ that start and end at boundary $r$-vertices. We call such an $r$-path $\Pbft$ a \emph{boundary $r$-path}.

Our goal is to define a bijection $\bij$ from the set of all rooted $G(v,t)$-forests to the set of all boundary $r$-paths in $\Nvt$ for $r\geq 0$. We view each boundary $r$-path in $\Nvt$ as a collection of edges of $\Nvt$. It is easy to define $\bij$ but it is non-trivial to prove that it is in fact a bijection. 

Let $F$ be a rooted $G(v,t)$-forest. We are going to describe the set of edges of $\Nvt$ that belong to the boundary $r$-path $\bij(F)$. Let us orient every edge of $F$ towards the root $\root_C$ of the corresponding connected component $C$ of $F$. Consider any lozenge $L$ of $G(v,t)$. By the definition of a $G(v,t)$-forest, there is a unique (oriented) edge $u\to w$ of $F$ inside $L$. Let $a_L,b_L,c_L,d_L,e_L$ be the vertices of $\Nvt$ inside $L$ as in Figure~\ref{fig:nvt}. If $u=d_L$ then we do not choose any edges of $\Nvt$ inside $L$ to belong to $\bij(F)$. Otherwise, we choose the edges $u\to e_L$ and $e_L\to d_L$. This defines $\bij(F)$ as a collection of edges of $\Nvt$. For the rooted $G(v,t)$-forest $F$ from Figure~\ref{fig:gvt_rooted} (left), the corresponding $7$-path $\bij(F)$ is given in Figure~\ref{fig:gvt_rooted} (right). This construction is similar to the well-known bijection between domino tilings and $r$-paths, see e.g.~\cite[Figure~8]{Networks}.

% 
% \begin{figure}
%  \centering
% %  \includegraphics[width=0.5\linewidth]{bij.png}
%  
%  
%  
%  \caption{\label{fig:bij} The boundary $7$-path corresponding to the rooted $G(v,t)$-forest from Figure~\ref{fig:gvt_rooted}.}
% \end{figure}

\begin{figure}
 \centering

\scalebox{1.0}{
\begin{tikzpicture}[scale=0.3]
\coordinate (L1) at (90:5.00);
\coordinate (L2) at (330:5.00);
\coordinate (L3) at (210:5.00);
\coordinate (Cm4I0I4) at (barycentric cs:L1=-3.0,L2=1.0,L3=5.0);
\coordinate (Cm4I1I3) at (barycentric cs:L1=-3.0,L2=2.0,L3=4.0);
\coordinate (Cm3Im1I4) at (barycentric cs:L1=-2.0,L2=0.0,L3=5.0);
\coordinate (Cm3I0I3) at (barycentric cs:L1=-2.0,L2=1.0,L3=4.0);
\coordinate (Cm3I1I2) at (barycentric cs:L1=-2.0,L2=2.0,L3=3.0);
\coordinate (Cm3I2I1) at (barycentric cs:L1=-2.0,L2=3.0,L3=2.0);
\coordinate (Cm2Im1I3) at (barycentric cs:L1=-1.0,L2=0.0,L3=4.0);
\coordinate (Cm2I0I2) at (barycentric cs:L1=-1.0,L2=1.0,L3=3.0);
\coordinate (Cm2I1I1) at (barycentric cs:L1=-1.0,L2=2.0,L3=2.0);
\coordinate (Cm2I2I0) at (barycentric cs:L1=-1.0,L2=3.0,L3=1.0);
\coordinate (Cm2I3Im1) at (barycentric cs:L1=-1.0,L2=4.0,L3=0.0);
\coordinate (Cm1Im2I3) at (barycentric cs:L1=0.0,L2=-1.0,L3=4.0);
\coordinate (Cm1Im1I2) at (barycentric cs:L1=0.0,L2=0.0,L3=3.0);
\coordinate (Cm1I0I1) at (barycentric cs:L1=0.0,L2=1.0,L3=2.0);
\coordinate (Cm1I1I0) at (barycentric cs:L1=0.0,L2=2.0,L3=1.0);
\coordinate (Cm1I2Im1) at (barycentric cs:L1=0.0,L2=3.0,L3=0.0);
\coordinate (Cm1I3Im2) at (barycentric cs:L1=0.0,L2=4.0,L3=-1.0);
\coordinate (Cm1I4Im3) at (barycentric cs:L1=0.0,L2=5.0,L3=-2.0);
\coordinate (C0Im2I2) at (barycentric cs:L1=1.0,L2=-1.0,L3=3.0);
\coordinate (C0Im1I1) at (barycentric cs:L1=1.0,L2=0.0,L3=2.0);
\coordinate (C0I0I0) at (barycentric cs:L1=1.0,L2=1.0,L3=1.0);
\coordinate (C0I1Im1) at (barycentric cs:L1=1.0,L2=2.0,L3=0.0);
\coordinate (C0I2Im2) at (barycentric cs:L1=1.0,L2=3.0,L3=-1.0);
\coordinate (C0I3Im3) at (barycentric cs:L1=1.0,L2=4.0,L3=-2.0);
\coordinate (C0I4Im4) at (barycentric cs:L1=1.0,L2=5.0,L3=-3.0);
\coordinate (C1Im3I2) at (barycentric cs:L1=2.0,L2=-2.0,L3=3.0);
\coordinate (C1Im2I1) at (barycentric cs:L1=2.0,L2=-1.0,L3=2.0);
\coordinate (C1Im1I0) at (barycentric cs:L1=2.0,L2=0.0,L3=1.0);
\coordinate (C1I0Im1) at (barycentric cs:L1=2.0,L2=1.0,L3=0.0);
\coordinate (C1I1Im2) at (barycentric cs:L1=2.0,L2=2.0,L3=-1.0);
\coordinate (C1I2Im3) at (barycentric cs:L1=2.0,L2=3.0,L3=-2.0);
\coordinate (C1I3Im4) at (barycentric cs:L1=2.0,L2=4.0,L3=-3.0);
\coordinate (C2Im3I1) at (barycentric cs:L1=3.0,L2=-2.0,L3=2.0);
\coordinate (C2Im2I0) at (barycentric cs:L1=3.0,L2=-1.0,L3=1.0);
\coordinate (C2Im1Im1) at (barycentric cs:L1=3.0,L2=0.0,L3=0.0);
\coordinate (C2I0Im2) at (barycentric cs:L1=3.0,L2=1.0,L3=-1.0);
\coordinate (C2I1Im3) at (barycentric cs:L1=3.0,L2=2.0,L3=-2.0);
\coordinate (C3Im4I1) at (barycentric cs:L1=4.0,L2=-3.0,L3=2.0);
\coordinate (C3Im3I0) at (barycentric cs:L1=4.0,L2=-2.0,L3=1.0);
\coordinate (C3Im2Im1) at (barycentric cs:L1=4.0,L2=-1.0,L3=0.0);
\coordinate (C3Im1Im2) at (barycentric cs:L1=4.0,L2=0.0,L3=-1.0);
\coordinate (C4Im4I0) at (barycentric cs:L1=5.0,L2=-3.0,L3=1.0);
\coordinate (C4Im3Im1) at (barycentric cs:L1=5.0,L2=-2.0,L3=0.0);
\draw[opacity=0.4] (Cm3I1I2) -- (Cm2I0I2);
\draw[opacity=0.4] (Cm3I1I2) -- (Cm4I1I3);
\draw[opacity=0.4] (Cm3I1I2) -- (Cm3I2I1);
\draw[opacity=0.4] (Cm3I2I1) -- (Cm2I1I1);
\draw[opacity=0.4] (C1Im3I2) -- (C1Im2I1);
\draw[opacity=0.4] (C1I2Im3) -- (C2I1Im3);
\draw[opacity=0.4] (C1I2Im3) -- (C0I2Im2);
\draw[opacity=0.4] (C1I2Im3) -- (C1I3Im4);
\draw[opacity=0.4] (C2Im3I1) -- (C3Im4I1);
\draw[opacity=0.4] (C2Im3I1) -- (C1Im3I2);
\draw[opacity=0.4] (C2Im3I1) -- (C2Im2I0);
\draw[opacity=0.4] (C2I1Im3) -- (C1I1Im2);
\draw[opacity=0.4] (Cm4I1I3) -- (Cm3I0I3);
\draw[opacity=0.4] (C1I3Im4) -- (C0I3Im3);
\draw[opacity=0.4] (C3Im4I1) -- (C3Im3I0);
\draw[opacity=0.4] (Cm1I0I1) -- (C0Im1I1);
\draw[opacity=0.4] (Cm1I0I1) -- (Cm2I0I2);
\draw[opacity=0.4] (Cm1I0I1) -- (Cm1I1I0);
\draw[opacity=0.4] (Cm1I1I0) -- (C0I0I0);
\draw[opacity=0.4] (Cm1I1I0) -- (Cm2I1I1);
\draw[opacity=0.4] (Cm1I1I0) -- (Cm1I2Im1);
\draw[opacity=0.4] (C0Im1I1) -- (C1Im2I1);
\draw[opacity=0.4] (C0Im1I1) -- (Cm1Im1I2);
\draw[opacity=0.4] (C0Im1I1) -- (C0I0I0);
\draw[opacity=0.4] (C0I1Im1) -- (C1I0Im1);
\draw[opacity=0.4] (C0I1Im1) -- (Cm1I1I0);
\draw[opacity=0.4] (C0I1Im1) -- (C0I2Im2);
\draw[opacity=0.4] (C1Im1I0) -- (C2Im2I0);
\draw[opacity=0.4] (C1Im1I0) -- (C0Im1I1);
\draw[opacity=0.4] (C1Im1I0) -- (C1I0Im1);
\draw[opacity=0.4] (C1I0Im1) -- (C2Im1Im1);
\draw[opacity=0.4] (C1I0Im1) -- (C0I0I0);
\draw[opacity=0.4] (C1I0Im1) -- (C1I1Im2);
\draw[opacity=0.4] (Cm2I0I2) -- (Cm1Im1I2);
\draw[opacity=0.4] (Cm2I0I2) -- (Cm3I0I3);
\draw[opacity=0.4] (Cm2I0I2) -- (Cm2I1I1);
\draw[opacity=0.4] (Cm2I2I0) -- (Cm1I1I0);
\draw[opacity=0.4] (Cm2I2I0) -- (Cm3I2I1);
\draw[opacity=0.4] (Cm2I2I0) -- (Cm2I3Im1);
\draw[opacity=0.4] (C0Im2I2) -- (C1Im3I2);
\draw[opacity=0.4] (C0Im2I2) -- (Cm1Im2I3);
\draw[opacity=0.4] (C0Im2I2) -- (C0Im1I1);
\draw[opacity=0.4] (C0I2Im2) -- (C1I1Im2);
\draw[opacity=0.4] (C0I2Im2) -- (Cm1I2Im1);
\draw[opacity=0.4] (C0I2Im2) -- (C0I3Im3);
\draw[opacity=0.4] (C2Im2I0) -- (C3Im3I0);
\draw[opacity=0.4] (C2Im2I0) -- (C1Im2I1);
\draw[opacity=0.4] (C2Im2I0) -- (C2Im1Im1);
\draw[opacity=0.4] (C2I0Im2) -- (C3Im1Im2);
\draw[opacity=0.4] (C2I0Im2) -- (C1I0Im1);
\draw[opacity=0.4] (C2I0Im2) -- (C2I1Im3);
\draw[opacity=0.4] (Cm1Im2I3) -- (Cm1Im1I2);
\draw[opacity=0.4] (Cm1I3Im2) -- (C0I2Im2);
\draw[opacity=0.4] (Cm1I3Im2) -- (Cm2I3Im1);
\draw[opacity=0.4] (Cm1I3Im2) -- (Cm1I4Im3);
\draw[opacity=0.4] (Cm2Im1I3) -- (Cm1Im2I3);
\draw[opacity=0.4] (Cm2Im1I3) -- (Cm3Im1I4);
\draw[opacity=0.4] (Cm2Im1I3) -- (Cm2I0I2);
\draw[opacity=0.4] (Cm2I3Im1) -- (Cm1I2Im1);
\draw[opacity=0.4] (C3Im1Im2) -- (C2Im1Im1);
\draw[opacity=0.4] (C3Im2Im1) -- (C4Im3Im1);
\draw[opacity=0.4] (C3Im2Im1) -- (C2Im2I0);
\draw[opacity=0.4] (C3Im2Im1) -- (C3Im1Im2);
\draw[opacity=0.4] (Cm1I4Im3) -- (C0I3Im3);
\draw[opacity=0.4] (Cm3Im1I4) -- (Cm3I0I3);
\draw[opacity=0.4] (Cm4I0I4) -- (Cm3Im1I4);
\draw[opacity=0.4] (Cm4I0I4) -- (Cm4I1I3);
\draw[opacity=0.4] (C0I4Im4) -- (C1I3Im4);
\draw[opacity=0.4] (C0I4Im4) -- (Cm1I4Im3);
\draw[opacity=0.4] (C4Im3Im1) -- (C3Im3I0);
\draw[opacity=0.4] (C4Im4I0) -- (C3Im4I1);
\draw[opacity=0.4] (C4Im4I0) -- (C4Im3Im1);
\draw[fill=red] (barycentric cs:L1=-1.0,L2=2.0,L3=2.0) circle (0.2);
\draw[fill=red] (barycentric cs:L1=2.0,L2=-1.0,L3=2.0) circle (0.2);
\draw[fill=red] (barycentric cs:L1=2.0,L2=2.0,L3=-1.0) circle (0.2);
\draw[fill=blue] (barycentric cs:L1=-2.0,L2=2.0,L3=3.0) circle (0.2);
\draw[fill=green] (barycentric cs:L1=-2.0,L2=3.0,L3=2.0) circle (0.2);
\draw[fill=green] (barycentric cs:L1=2.0,L2=-2.0,L3=3.0) circle (0.2);
\draw[fill=blue] (barycentric cs:L1=2.0,L2=3.0,L3=-2.0) circle (0.2);
\draw[fill=blue] (barycentric cs:L1=3.0,L2=-2.0,L3=2.0) circle (0.2);
\draw[fill=green] (barycentric cs:L1=3.0,L2=2.0,L3=-2.0) circle (0.2);
\draw[fill=green] (barycentric cs:L1=-3.0,L2=2.0,L3=4.0) circle (0.2);
\draw[fill=green] (barycentric cs:L1=2.0,L2=4.0,L3=-3.0) circle (0.2);
\draw[fill=green] (barycentric cs:L1=4.0,L2=-3.0,L3=2.0) circle (0.2);
\draw[fill=blue] (barycentric cs:L1=0.0,L2=1.0,L3=2.0) circle (0.2);
\draw[fill=green] (barycentric cs:L1=0.0,L2=2.0,L3=1.0) circle (0.2);
\draw[fill=green] (barycentric cs:L1=1.0,L2=0.0,L3=2.0) circle (0.2);
\draw[fill=blue] (barycentric cs:L1=1.0,L2=2.0,L3=0.0) circle (0.2);
\draw[fill=blue] (barycentric cs:L1=2.0,L2=0.0,L3=1.0) circle (0.2);
\draw[fill=green] (barycentric cs:L1=2.0,L2=1.0,L3=0.0) circle (0.2);
\draw[fill=red] (barycentric cs:L1=0.0,L2=0.0,L3=3.0) circle (0.2);
\draw[fill=red] (barycentric cs:L1=0.0,L2=3.0,L3=0.0) circle (0.2);
\draw[fill=green] (barycentric cs:L1=-1.0,L2=1.0,L3=3.0) circle (0.2);
\draw[fill=blue] (barycentric cs:L1=-1.0,L2=3.0,L3=1.0) circle (0.2);
\draw[fill=blue] (barycentric cs:L1=1.0,L2=-1.0,L3=3.0) circle (0.2);
\draw[fill=green] (barycentric cs:L1=1.0,L2=3.0,L3=-1.0) circle (0.2);
\draw[fill=red] (barycentric cs:L1=3.0,L2=0.0,L3=0.0) circle (0.2);
\draw[fill=green] (barycentric cs:L1=3.0,L2=-1.0,L3=1.0) circle (0.2);
\draw[fill=blue] (barycentric cs:L1=3.0,L2=1.0,L3=-1.0) circle (0.2);
\draw[fill=green] (barycentric cs:L1=0.0,L2=-1.0,L3=4.0) circle (0.2);
\draw[fill=blue] (barycentric cs:L1=0.0,L2=4.0,L3=-1.0) circle (0.2);
\draw[fill=blue] (barycentric cs:L1=-1.0,L2=0.0,L3=4.0) circle (0.2);
\draw[fill=green] (barycentric cs:L1=-1.0,L2=4.0,L3=0.0) circle (0.2);
\draw[fill=red] (barycentric cs:L1=-2.0,L2=1.0,L3=4.0) circle (0.2);
\draw[fill=red] (barycentric cs:L1=1.0,L2=4.0,L3=-2.0) circle (0.2);
\draw[fill=green] (barycentric cs:L1=4.0,L2=0.0,L3=-1.0) circle (0.2);
\draw[fill=blue] (barycentric cs:L1=4.0,L2=-1.0,L3=0.0) circle (0.2);
\draw[fill=red] (barycentric cs:L1=4.0,L2=-2.0,L3=1.0) circle (0.2);
\draw[fill=green] (barycentric cs:L1=0.0,L2=5.0,L3=-2.0) circle (0.2);
\draw[fill=green] (barycentric cs:L1=-2.0,L2=0.0,L3=5.0) circle (0.2);
\draw[fill=blue] (barycentric cs:L1=-3.0,L2=1.0,L3=5.0) circle (0.2);
\draw[fill=blue] (barycentric cs:L1=1.0,L2=5.0,L3=-3.0) circle (0.2);
\draw[fill=green] (barycentric cs:L1=5.0,L2=-2.0,L3=0.0) circle (0.2);
\draw[fill=blue] (barycentric cs:L1=5.0,L2=-3.0,L3=1.0) circle (0.2);
\draw[fill=red] (barycentric cs:L1=1.0,L2=1.0,L3=1.0) circle (0.2);
\tikzset{>=latex}
\draw[->,line width=0.4mm,color=black!30!green] (C1Im3I2) -- (C0Im1I1);
\draw[->,line width=0.4mm,color=blue] (C2I0Im2) -- (C1I1Im2);
\draw[->,line width=0.4mm,color=blue] (C1I1Im2) -- (C1I2Im3);
\draw[->,line width=0.4mm,color=blue] (C1I2Im3) -- (C0I3Im3);
\draw[->,line width=0.4mm,color=blue] (C0I3Im3) -- (C0I4Im4);
\draw[->,line width=0.4mm,color=black!30!green] (C3Im4I1) -- (C2Im2I0);
\draw[->,line width=0.4mm,color=blue] (Cm2I1I1) -- (Cm2I2I0);
\draw[->,line width=0.4mm,color=blue] (Cm1I0I1) -- (Cm2I1I1);
\draw[->,line width=0.4mm,color=black!30!green] (Cm1I1I0) -- (Cm2I3Im1);
\draw[->,line width=0.4mm,color=black!30!green] (C0Im1I1) -- (Cm1I1I0);
\draw[->,line width=0.4mm,color=blue] (Cm1Im1I2) -- (Cm1I0I1);
\draw[->,line width=0.4mm,color=blue] (C0Im2I2) -- (Cm1Im1I2);
\draw[->,line width=0.4mm,color=blue] (C0I0I0) -- (C0I1Im1);
\draw[->,line width=0.4mm,color=blue] (C1Im1I0) -- (C0I0I0);
\draw[->,line width=0.4mm,color=black!30!green] (C1I0Im1) -- (C0I2Im2);
\draw[->,line width=0.4mm,color=black!30!green] (Cm2I0I2) -- (Cm3I2I1);
\draw[->,line width=0.4mm,color=blue] (Cm3I0I3) -- (Cm3I1I2);
\draw[->,line width=0.4mm,color=blue] (Cm2Im1I3) -- (Cm3I0I3);
\draw[->,line width=0.4mm,color=blue] (C0I1Im1) -- (Cm1I2Im1);
\draw[->,line width=0.4mm,color=blue] (Cm1I2Im1) -- (Cm1I3Im2);
\draw[->,line width=0.4mm,color=black!30!green] (C0I2Im2) -- (Cm1I4Im3);
\draw[->,line width=0.4mm,color=blue] (C2Im3I1) -- (C1Im2I1);
\draw[->,line width=0.4mm,color=blue] (C1Im2I1) -- (C1Im1I0);
\draw[->,line width=0.4mm,color=black!30!green] (C2Im2I0) -- (C1I0Im1);
\draw[->,line width=0.4mm,color=black!30!green] (Cm1Im2I3) -- (Cm2I0I2);
\draw[->,line width=0.4mm,color=blue] (C2Im1Im1) -- (C2I0Im2);
\draw[->,line width=0.4mm,color=blue] (C3Im2Im1) -- (C2Im1Im1);
\draw[->,line width=0.4mm,color=black!30!green] (Cm3Im1I4) -- (Cm4I1I3);
\draw[->,line width=0.4mm,color=blue] (C4Im4I0) -- (C3Im3I0);
\draw[->,line width=0.4mm,color=blue] (C3Im3I0) -- (C3Im2Im1);
\draw[fill=red] (barycentric cs:L1=-1.0,L2=2.0,L3=2.0) circle (0.2);
\draw[fill=red] (barycentric cs:L1=2.0,L2=-1.0,L3=2.0) circle (0.2);
\draw[fill=red] (barycentric cs:L1=2.0,L2=2.0,L3=-1.0) circle (0.2);
\draw[fill=blue] (barycentric cs:L1=-2.0,L2=2.0,L3=3.0) circle (0.2);
\draw[fill=green] (barycentric cs:L1=-2.0,L2=3.0,L3=2.0) circle (0.2);
\draw[fill=green] (barycentric cs:L1=2.0,L2=-2.0,L3=3.0) circle (0.2);
\draw[fill=blue] (barycentric cs:L1=2.0,L2=3.0,L3=-2.0) circle (0.2);
\draw[fill=blue] (barycentric cs:L1=3.0,L2=-2.0,L3=2.0) circle (0.2);
\draw[fill=green] (barycentric cs:L1=3.0,L2=2.0,L3=-2.0) circle (0.2);
\draw[fill=green] (barycentric cs:L1=-3.0,L2=2.0,L3=4.0) circle (0.2);
\draw[fill=green] (barycentric cs:L1=2.0,L2=4.0,L3=-3.0) circle (0.2);
\draw[fill=green] (barycentric cs:L1=4.0,L2=-3.0,L3=2.0) circle (0.2);
\draw[fill=blue] (barycentric cs:L1=0.0,L2=1.0,L3=2.0) circle (0.2);
\draw[fill=green] (barycentric cs:L1=0.0,L2=2.0,L3=1.0) circle (0.2);
\draw[fill=green] (barycentric cs:L1=1.0,L2=0.0,L3=2.0) circle (0.2);
\draw[fill=blue] (barycentric cs:L1=1.0,L2=2.0,L3=0.0) circle (0.2);
\draw[fill=blue] (barycentric cs:L1=2.0,L2=0.0,L3=1.0) circle (0.2);
\draw[fill=green] (barycentric cs:L1=2.0,L2=1.0,L3=0.0) circle (0.2);
\draw[fill=red] (barycentric cs:L1=0.0,L2=0.0,L3=3.0) circle (0.2);
\draw[fill=red] (barycentric cs:L1=0.0,L2=3.0,L3=0.0) circle (0.2);
\draw[fill=green] (barycentric cs:L1=-1.0,L2=1.0,L3=3.0) circle (0.2);
\draw[fill=blue] (barycentric cs:L1=-1.0,L2=3.0,L3=1.0) circle (0.2);
\draw[fill=blue] (barycentric cs:L1=1.0,L2=-1.0,L3=3.0) circle (0.2);
\draw[fill=green] (barycentric cs:L1=1.0,L2=3.0,L3=-1.0) circle (0.2);
\draw[fill=red] (barycentric cs:L1=3.0,L2=0.0,L3=0.0) circle (0.2);
\draw[fill=green] (barycentric cs:L1=3.0,L2=-1.0,L3=1.0) circle (0.2);
\draw[fill=blue] (barycentric cs:L1=3.0,L2=1.0,L3=-1.0) circle (0.2);
\draw[fill=green] (barycentric cs:L1=0.0,L2=-1.0,L3=4.0) circle (0.2);
\draw[fill=blue] (barycentric cs:L1=0.0,L2=4.0,L3=-1.0) circle (0.2);
\draw[fill=blue] (barycentric cs:L1=-1.0,L2=0.0,L3=4.0) circle (0.2);
\draw[fill=green] (barycentric cs:L1=-1.0,L2=4.0,L3=0.0) circle (0.2);
\draw[fill=red] (barycentric cs:L1=-2.0,L2=1.0,L3=4.0) circle (0.2);
\draw[fill=red] (barycentric cs:L1=1.0,L2=4.0,L3=-2.0) circle (0.2);
\draw[fill=green] (barycentric cs:L1=4.0,L2=0.0,L3=-1.0) circle (0.2);
\draw[fill=blue] (barycentric cs:L1=4.0,L2=-1.0,L3=0.0) circle (0.2);
\draw[fill=red] (barycentric cs:L1=4.0,L2=-2.0,L3=1.0) circle (0.2);
\draw[fill=green] (barycentric cs:L1=0.0,L2=5.0,L3=-2.0) circle (0.2);
\draw[fill=green] (barycentric cs:L1=-2.0,L2=0.0,L3=5.0) circle (0.2);
\draw[fill=blue] (barycentric cs:L1=-3.0,L2=1.0,L3=5.0) circle (0.2);
\draw[fill=blue] (barycentric cs:L1=1.0,L2=5.0,L3=-3.0) circle (0.2);
\draw[fill=green] (barycentric cs:L1=5.0,L2=-2.0,L3=0.0) circle (0.2);
\draw[fill=blue] (barycentric cs:L1=5.0,L2=-3.0,L3=1.0) circle (0.2);
\draw[fill=red] (barycentric cs:L1=1.0,L2=1.0,L3=1.0) circle (0.2);
\end{tikzpicture}}

 \caption{\label{fig:forest_zero} The rooted $G(v,t)$-forest that corresponds to the boundary $0$-path in $\Nvt$.}
\end{figure}
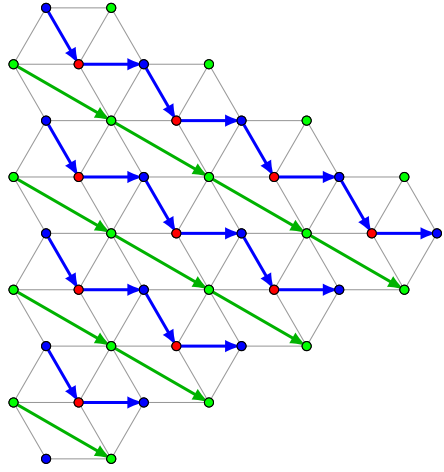

\def\extrawt{W_{(v,t)}}
\begin{theorem}\label{thm:bij}
 The map $\bij$ is a bijection between rooted $G(v,t)$-forests $F$ and boundary $r$-paths in $\Nvt$ for $r\geq 0$. Moreover, this bijection is weight-preserving:
 \[\wt(F)=\wt(\bij(F))\cdot \extrawt,\quad \text{where } \extrawt=
 \frac{x_{a_1}x_{a_3}\cdots x_{a_{2t-1}}}{x_{a_2}x_{a_4}\cdots x_{a_{2t-2}}}\]
 is the weight of the unique rooted $G(v,t)$-forest shown in Figure~\ref{fig:forest_zero} that corresponds to the boundary $0$-path in $\Nvt$.
\end{theorem}
\begin{proof}
 We first show that $\bij(F)$ is indeed a boundary $r$-path for some $r\geq 0$ and not merely a collection of edges of $\Nvt$. Since there are no edges in $\Nvt$ connecting two boundary vertices to each other, it suffices to show the following:
 \begin{enumerate}[(i)]
  \item\label{item:boundary} every boundary vertex of $\Nvt$ has at most one outgoing edge and at most one incoming edge in $\bij(F)$;
  \item\label{item:non_boundary} every non-boundary vertex of $\Nvt$ either is isolated or has precisely one incoming and one outgoing edge in $\bij(F)$.
 \end{enumerate}
Consider any vertex $u$ of $\Nvt$. If $u$ is the center of some lozenge of $G(v,t)$ then~\eqref{item:non_boundary} is obvious by construction of $\bij$, so suppose $u\in\Vert(v,t)$. Note that there is at most one edge oriented towards $u$ in $\Nvt$. Moreover, there is at most one oriented edge $u\to w$ in $F$, so the outdegree of $u$ in $\bij(F)$ is also at most one. This proves~\eqref{item:boundary} for all boundary vertices $u\in\Gbound$. Suppose now that $u\in\Vert(v,t)\setminus\Gbound$. In this case, the edge $u\to w$ in $F$ exists and is unique.

There is exactly one lozenge $L(u)$ of $G(v,t)$ that contains the unique incoming edge of $u$ in $\Nvt$. This is the lozenge $L(u)$ in Figure~\ref{fig:nvt} for which $u=d_{L(u)}$. There is also exactly one lozenge $L'$ of $G(v,t)$ that contains the edge $u\to w$. If $L'\neq L(u)$ then $\bij(F)$ clearly contains exactly one edge directed towards $u$ (namely, the edge inside $L(u)$) and exactly one edge directed from $u$ (namely, the one inside $L'$). Otherwise, if $L'=L(u)$ then $u$ has to be isolated. We are done with~\eqref{item:non_boundary} and thus we have shown that $\bij(F)$ is a boundary $r$-path for some $r\geq 0$.

Suppose now that $\Pbft$ is a boundary $r$-path in $\Nvt$. We are going to define $F=\bijinv(\Pbft)$ as a collection of oriented edges with vertex set $\Vert(v,t)$. Consider any lozenge $L$ of $G(v,t)$ and label the corresponding vertices of $\Nvt$ by $a_L,b_L,c_L,d_L,e_L$ as in Figure~\ref{fig:nvt}. Since $e_L$ is not a boundary vertex, there are four options of how the edges of $\Pbft$ inside $L$ can look like: 
\begin{enumerate}
 \item there are no edges of $\Pbft$ inside $L$;
 \item the only two edges are $a_L\to e_L\to d_L$;
 \item the only two edges are $b_L\to e_L\to d_L$;
 \item the only two edges are $c_L\to e_L\to d_L$.
\end{enumerate}
For each of the four options, we will choose the unique oriented edge of $F$ inside $L$:
\begin{enumerate}
 \item if there are no edges of $\Pbft$ inside $L$, choose $(d_L\to b_L)\in F$;
 \item if the edges are $a_L\to e_L\to d_L$, choose $(a_L\to c_L)\in F$;
 \item if the edges are  $b_L\to e_L\to d_L$, choose $(b_L\to d_L)\in F$;
 \item if the edges are  $c_L\to e_L\to d_L$, choose $(c_L\to a_L)\in F$.
\end{enumerate}
This defines $F$ as a collection of oriented edges. It is clear that if $F$ is a rooted $G(v,t)$-forest then $\bij(F)=\Pbft$ because we are basically inverting the local rule for $\bij$ inside every lozenge. We will show that $F$ is a rooted $G(v,t)$-forest where the edges of every connected component $C$ of $F$ are oriented towards some boundary vertex $\root_C\in\Gbound$. It suffices to show the following:
 \begin{enumerate}[(i)]
  \item\label{item:boundary_inv} every boundary vertex of $G(v,t)$ has at most one outgoing edge in $F$;
  \item\label{item:non_boundary_inv} every non-boundary vertex of $G(v,t)$ has exactly one outgoing edge in $F$;
  \item\label{item:cycles} there are no cycles in $F$.
 \end{enumerate}
 Consider any vertex $u\in\Vert(v,t)$. Since $\Pbft$ is an $r$-path in $\Nvt$, there is at most one outgoing edge of $u$ in $\Pbft$. If there is exactly one such edge then from the definition of $F=\bijinv(\Pbft)$, it is clear that therefore there is at most one outgoing edge of $u$ in $F$. On the other hand, if there is no such edge then $u$ is either a source in $F$ in which case we are done or $u$ is isolated in $\Pbft$ and there is a lozenge $L(u)$ of $G(v,t)$ such that $u=d_{L(u)}$ in $L(u)$ and then the unique edge in $F$ coming out from $u$ will be $d_{L(u)}\to b_{L(u)}$ inside $L(u)$. If $u$ is not on the boundary then either $u$ has an outgoing edge in $\Pbft$ or $u$ is isolated in $\Pbft$, because every path in $\Pbft$ starts and ends on the boundary. In both cases we have shown that $u$ has exactly one outgoing edge in $F$ which proves~\eqref{item:boundary_inv} and~\eqref{item:non_boundary_inv} together. 
 
 To prove~\eqref{item:cycles}, suppose that there is a cycle $C$ in $F$. By~\eqref{item:boundary_inv} and~\eqref{item:non_boundary_inv}, $C$ has to be a directed cycle. Note also that the vertices in $C$ are either all green or all red-blue. Suppose that they are all green. One easily observes\footnote{Indeed, take any (green) edge $e$ of $C$ and consider the lozenge $L$ containing it. It has a red and a blue vertex, and one of them therefore necessarily lies inside of $C$ because they lie on different sides of $e$.} that then there must be a red or a blue vertex inside of $C$. Every such vertex does not belong to $\Gbound$ and thus necessarily has an outgoing edge in $F$. This edge cannot intersect $C$ so its end is another red or blue vertex inside of $C$. Continuing in this fashion, we get a red-blue cycle $C_1$ inside $C$. But it is also easy to see that every red-blue cycle has to contain a green point inside of it, and so by the above argument we will get a green cycle $C_2$ inside $C_1$. This process has to terminate at some point leading to a cycle in $F$ that has no vertices inside of it which is a contradiction since such a cycle cannot exist in $F$. We are done with~\eqref{item:cycles}.
 
 We have thus defined two maps $\bij$ and $\bijinv$, it is obvious that they are inverse to each other, and by the above series of claims, $\bij$ maps each rooted $G(v,t)$-forest to an $r$-path $\Pbft$ in $\Nvt$ for some $r$, and conversely, for every such $\Pbft$ the map $\bijinv$ yields a rooted $G(v,t)$-forest. Therefore we are done with the claim that $\bij$ is a bijection.
 
 To see why we have $\wt(F)=\wt(\bij(F))\cdot\extrawt$, note that the map $\bij$ actually extends to arbitrary collections of directed edges in $G(v,t)$ such that for every lozenge of $G$ we choose exactly one of its four possible oriented diagonals. Every such collection $F$ has weight given by~\eqref{eq:weight_s}, and its image is some collection $\bij(F)$ of edges in $\Nvt$ whose weight can be defined as the product of the edges contained in it. Clearly when $\bij(F)$ contains no edges then we have $\wt(F)=1\cdot \extrawt$ so the formula is correct. It is easy to see that it remains correct when we alter just one edge of $F$. Since we can obtain all rooted $G(v,t)$-forests in this way, the result follows. We are done with Theorem~\ref{thm:bij}.
\end{proof}

Let $F$ now be a $G(v,t)$-grove. We choose a \emph{canonical root} for each connected component $C$ of $F$ as follows. If $C$ contains a boundary vertex $c_i$ for some $i$ then we set $c_i$ as the root of $C$. Otherwise the intersection of $C$ with $\Gbound$ consists of two vertices $b_i$ and $a_{2t-i}$ for some $1\leq i<t$, in which case we choose $b_i$ as the root of $C$. Thus for each $G(v,t)$-grove $F$, we set $\bij(F)$ to be the $r$-path in $\Nvt$ that corresponds to the rooted $G(v,t)$-forest obtained from $F$ by orienting its edges towards the roots that we have chosen above. For the $G(v,t)$-grove from Figure~\ref{fig:gvt_grove}, the corresponding canonically rooted $G(v,t)$-forest and the $r$-path $\bij(F)$ are shown in Figure~\ref{fig:grove_bij}.

\begin{figure}
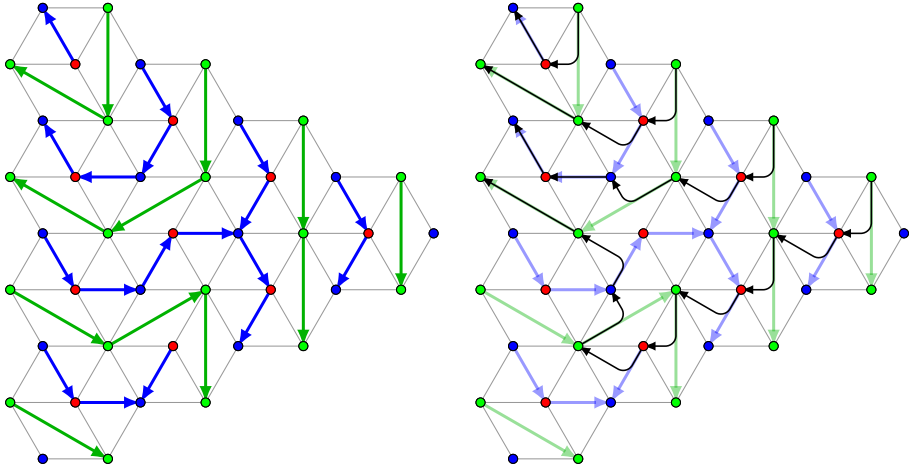

 \centering



 \caption{\label{fig:grove_bij} A $G(v,t)$-grove rooted in a canonical way and the corresponding $4$-path in $\Nvt$.}
\end{figure}

Let $\ubft=(a_2,a_4,\dots,a_{2t-2})$ and $\wbft=(b_{t-1},b_{t-2},\dots,b_1)$ be two boundary $(t-1)$-vertices. We let $\pathst(\ubft,\wbft)$ denote the set of $(t-1)$-paths from $\ubft$ to $\wbft$ in $\Nvt$.

We will later show that if $F$ is a canonically rooted $G(v,t)$-grove then $\bij(F)\in\pathst(\ubft,\wbft)$. One may hope that the image of $\bij$ is the whole $\pathst(\ubft,\wbft)$. However, Figure~\ref{fig:counter12} demonstrates that this is not the case. We now describe the preimage of $\pathst(\ubft,\wbft)$ under $\bij$.

\def\roots{R}
Given a rooted $G(v,t)$-forest $F$, we denote by $\roots(F)\subset \Gbound$ the set of its roots. Thus for any $G(v,t)$-grove $F$ rooted canonically as above we have 
\[\roots(F)=\roots_0:=\{c_1,c_2,\dots,c_{2t-1},b_1,b_2,\dots,b_{t-1}\}.\]
\begin{remark}\label{rmk:counter}
Note that the $G(v,t)$-forest $F$ in Figure~\ref{fig:counter12} satisfies $\roots(F)=\roots_0$ even though $F$ is not a $G(v,t)$-grove.
\end{remark}

\begin{theorem}\label{thm:bij_groves}
 The map $\bij$ is a bijection between the set of rooted $G(v,t)$-forests $F$ with $\roots(F)=\roots_0$ and the set $\pathst(\ubft,\wbft)$.
\end{theorem}
\begin{proof}
 To prove that the image of $\bij$ is contained in $\pathst(\ubft,\wbft)$, it suffices to show that for every $G(v,t)$-forest $F$ with $\roots(F)=\roots_0$, $\bij(F)$ is a $(t-1)$-path and that every path in $\bij(F)$ starts at some vertex of $\ubft$ and ends at some vertex of $\wbft$. If we establish this then the ordering of the vertices of $\ubft$ and the vertices of $\wbft$ will be unique because the vertices of $\ubft$ appear earlier in the clockwise order than the vertices of $\wbft$ and the paths in $\bij(F)$ cannot cross each other. Recall that in this case the $(t-1)$-vertices $\ubft$ and $\wbft$ are called \emph{non-permutable}.
 
 It follows from the proof of Theorem~\ref{thm:bij} that some path of $\bij(F)$ starts at a vertex $u\in\Vert(v,t)$ if and only if both of the following conditions are satisfied:
 \begin{itemize}
  \item the lozenge $L(u)$ of $G$ from Figure~\ref{fig:nvt} has to lie outside of $G(v,t)$;
  \item $u$ has to have an outgoing edge in $F$.
 \end{itemize}
 Indeed, if the first condition is not satisfied then we have shown in the proof of Theorem~\ref{thm:bij} that $u$ is either isolated or has both an incoming and an outgoing edge. If the second condition is not satisfied then $u$ does not have an outgoing edge in $\bij(F)$ as well. Conversely, if both conditions are satisfied then $u$ has an outgoing edge but does not have an incoming edge in $\bij(F)$ which is exactly the case when some path in $\bij(F)$ starts at $u$. 
 
 Thus the set of vertices where some path of $\bij(F)$ starts is exactly the set $\{a_2,a_4,\dots,a_{2t-2}\}$. A completely similar argument shows that the set of vertices where some path of $\bij(F)$ ends is exactly $\{b_1,b_2,\dots,b_{t-1}\}$ because these are the vertices $u$ such that the lozenge $L(u)$ lies inside $G(v,t)$ and that do not have an outgoing edge in $F$. We have shown that $\bij(F)\in\pathst(\ubft,\wbft)$.

 Now consider any $(t-1)$-path $\Pbft\in\pathst(\ubft,\wbft)$ and let $F=\bijinv(\Pbft)$ be the corresponding $G(v,t)$-forest. We claim that $\roots(F)=\roots_0$. Recall that $\roots(F)$ consists of all the vertices $u\in\Vert(v,t)$ that do not have an outgoing edge in $F$. For each such vertex $u$, an argument analogous to the above shows that we have exactly two possibilities:
 \begin{itemize}
  \item the lozenge $L(u)$ of $G$ from Figure~\ref{fig:nvt} lies outside of $G(v,t)$ and no path in $\Pbft$ starts at $u$;
  \item the lozenge $L(u)$ lies inside of $G(v,t)$ and some path in $\Pbft$ ends at $u$.
 \end{itemize}
It is clear now that the set of roots of $F$ is precisely $\roots_0$. We are done with the proof of Theorem~\ref{thm:bij_groves}.
\end{proof}

\begin{figure}
 \centering
%
%\begin{tabular}{|c|c|}\hline 	
%  \includegraphics[width=0.5\linewidth]{counter_1.png} &
%  \includegraphics[width=0.5\linewidth]{counter_2.png}\\\hline
%  \end{tabular}%

 \caption{\label{fig:counter12} A $G(v,t)$-forest $F$ that is not a $G(v,t)$-grove even though $\bij(F)\in\pathst(\ubft,\wbft)$.}
\end{figure}

Given a $G(v,t)$-forest $F$ with $\roots(F)=\roots_0$, we would like to give a necessary condition for $F$ to be a canonically oriented $G(v,t)$-grove. 

\def\Fbound{\partial F}
\begin{definition}
	Let $F$ be any $G(v,t)$-forest. Define the map $\Fbound:\Gbound\to\roots(F)$ as follows: for any vertex $u\in\Gbound$, we put $\Fbound(u):=\root_C$ where $C$ is the connected component of $F$ containing $u$ and $\root_C$ is its root.
\end{definition}

Several properties of the map $\Fbound$ are immediate. First, if $u\in\GboundB$ then $\Fbound(u)$ is blue, otherwise it is green. In other words, $\Fbound$ preserves the colors of vertices, so we call an arrow $u\to\Fbound(u)$ \emph{green} (resp., \emph{blue}) if both $u$ and $\Fbound(u)$ are green (resp., blue). Second, the (combinatorial) arrows $u\to \Fbound(u)$ are pairwise non-intersecting, that is, given two vertices $u\neq w\in \Gbound$ so that $\Fbound(u)\neq\Fbound(w)$, it is not the case that $(u,w,\Fbound(u),\Fbound(w))$ are cyclically oriented on the boundary of $G(v,t)$. Third, for any $u\in\roots(F)$ we have $\Fbound(u)=u$.

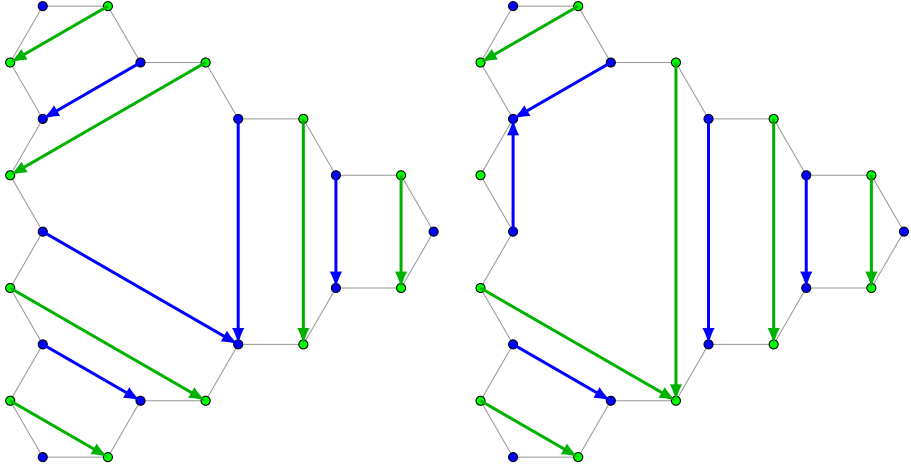
\begin{figure}

\begin{tabular}{cc}
\scalebox{1.0}{
\begin{tikzpicture}[scale=0.3]
\coordinate (L1) at (90:5.00);
\coordinate (L2) at (330:5.00);
\coordinate (L3) at (210:5.00);
\coordinate (Cm4I0I4) at (barycentric cs:L1=-3.0,L2=1.0,L3=5.0);
\coordinate (Cm4I1I3) at (barycentric cs:L1=-3.0,L2=2.0,L3=4.0);
\coordinate (Cm3Im1I4) at (barycentric cs:L1=-2.0,L2=0.0,L3=5.0);
\coordinate (Cm3I1I2) at (barycentric cs:L1=-2.0,L2=2.0,L3=3.0);
\coordinate (Cm3I2I1) at (barycentric cs:L1=-2.0,L2=3.0,L3=2.0);
\coordinate (Cm2Im1I3) at (barycentric cs:L1=-1.0,L2=0.0,L3=4.0);
\coordinate (Cm2I2I0) at (barycentric cs:L1=-1.0,L2=3.0,L3=1.0);
\coordinate (Cm2I3Im1) at (barycentric cs:L1=-1.0,L2=4.0,L3=0.0);
\coordinate (Cm1Im2I3) at (barycentric cs:L1=0.0,L2=-1.0,L3=4.0);
\coordinate (Cm1I3Im2) at (barycentric cs:L1=0.0,L2=4.0,L3=-1.0);
\coordinate (Cm1I4Im3) at (barycentric cs:L1=0.0,L2=5.0,L3=-2.0);
\coordinate (C0Im2I2) at (barycentric cs:L1=1.0,L2=-1.0,L3=3.0);
\coordinate (C0I4Im4) at (barycentric cs:L1=1.0,L2=5.0,L3=-3.0);
\coordinate (C1Im3I2) at (barycentric cs:L1=2.0,L2=-2.0,L3=3.0);
\coordinate (C1I2Im3) at (barycentric cs:L1=2.0,L2=3.0,L3=-2.0);
\coordinate (C1I3Im4) at (barycentric cs:L1=2.0,L2=4.0,L3=-3.0);
\coordinate (C2Im3I1) at (barycentric cs:L1=3.0,L2=-2.0,L3=2.0);
\coordinate (C2I0Im2) at (barycentric cs:L1=3.0,L2=1.0,L3=-1.0);
\coordinate (C2I1Im3) at (barycentric cs:L1=3.0,L2=2.0,L3=-2.0);
\coordinate (C3Im4I1) at (barycentric cs:L1=4.0,L2=-3.0,L3=2.0);
\coordinate (C3Im2Im1) at (barycentric cs:L1=4.0,L2=-1.0,L3=0.0);
\coordinate (C3Im1Im2) at (barycentric cs:L1=4.0,L2=0.0,L3=-1.0);
\coordinate (C4Im4I0) at (barycentric cs:L1=5.0,L2=-3.0,L3=1.0);
\coordinate (C4Im3Im1) at (barycentric cs:L1=5.0,L2=-2.0,L3=0.0);
\draw[opacity=0.4] (Cm3I1I2) -- (Cm4I1I3);
\draw[opacity=0.4] (Cm3I1I2) -- (Cm3I2I1);
\draw[opacity=0.4] (C1I2Im3) -- (C2I1Im3);
\draw[opacity=0.4] (C1I2Im3) -- (C1I3Im4);
\draw[opacity=0.4] (C2Im3I1) -- (C3Im4I1);
\draw[opacity=0.4] (C2Im3I1) -- (C1Im3I2);
\draw[opacity=0.4] (Cm2I2I0) -- (Cm3I2I1);
\draw[opacity=0.4] (Cm2I2I0) -- (Cm2I3Im1);
\draw[opacity=0.4] (C0Im2I2) -- (C1Im3I2);
\draw[opacity=0.4] (C0Im2I2) -- (Cm1Im2I3);
\draw[opacity=0.4] (C2I0Im2) -- (C3Im1Im2);
\draw[opacity=0.4] (C2I0Im2) -- (C2I1Im3);
\draw[opacity=0.4] (Cm1I3Im2) -- (Cm2I3Im1);
\draw[opacity=0.4] (Cm1I3Im2) -- (Cm1I4Im3);
\draw[opacity=0.4] (Cm2Im1I3) -- (Cm1Im2I3);
\draw[opacity=0.4] (Cm2Im1I3) -- (Cm3Im1I4);
\draw[opacity=0.4] (C3Im2Im1) -- (C4Im3Im1);
\draw[opacity=0.4] (C3Im2Im1) -- (C3Im1Im2);
\draw[opacity=0.4] (Cm4I0I4) -- (Cm3Im1I4);
\draw[opacity=0.4] (Cm4I0I4) -- (Cm4I1I3);
\draw[opacity=0.4] (C0I4Im4) -- (C1I3Im4);
\draw[opacity=0.4] (C0I4Im4) -- (Cm1I4Im3);
\draw[opacity=0.4] (C4Im4I0) -- (C3Im4I1);
\draw[opacity=0.4] (C4Im4I0) -- (C4Im3Im1);
\draw[fill=blue] (barycentric cs:L1=-2.0,L2=2.0,L3=3.0) circle (0.2);
\draw[fill=green] (barycentric cs:L1=-2.0,L2=3.0,L3=2.0) circle (0.2);
\draw[fill=green] (barycentric cs:L1=2.0,L2=-2.0,L3=3.0) circle (0.2);
\draw[fill=blue] (barycentric cs:L1=2.0,L2=3.0,L3=-2.0) circle (0.2);
\draw[fill=blue] (barycentric cs:L1=3.0,L2=-2.0,L3=2.0) circle (0.2);
\draw[fill=green] (barycentric cs:L1=3.0,L2=2.0,L3=-2.0) circle (0.2);
\draw[fill=green] (barycentric cs:L1=-3.0,L2=2.0,L3=4.0) circle (0.2);
\draw[fill=green] (barycentric cs:L1=2.0,L2=4.0,L3=-3.0) circle (0.2);
\draw[fill=green] (barycentric cs:L1=4.0,L2=-3.0,L3=2.0) circle (0.2);
\draw[fill=blue] (barycentric cs:L1=-1.0,L2=3.0,L3=1.0) circle (0.2);
\draw[fill=blue] (barycentric cs:L1=1.0,L2=-1.0,L3=3.0) circle (0.2);
\draw[fill=blue] (barycentric cs:L1=3.0,L2=1.0,L3=-1.0) circle (0.2);
\draw[fill=green] (barycentric cs:L1=0.0,L2=-1.0,L3=4.0) circle (0.2);
\draw[fill=blue] (barycentric cs:L1=0.0,L2=4.0,L3=-1.0) circle (0.2);
\draw[fill=blue] (barycentric cs:L1=-1.0,L2=0.0,L3=4.0) circle (0.2);
\draw[fill=green] (barycentric cs:L1=-1.0,L2=4.0,L3=0.0) circle (0.2);
\draw[fill=green] (barycentric cs:L1=4.0,L2=0.0,L3=-1.0) circle (0.2);
\draw[fill=blue] (barycentric cs:L1=4.0,L2=-1.0,L3=0.0) circle (0.2);
\draw[fill=green] (barycentric cs:L1=0.0,L2=5.0,L3=-2.0) circle (0.2);
\draw[fill=green] (barycentric cs:L1=-2.0,L2=0.0,L3=5.0) circle (0.2);
\draw[fill=blue] (barycentric cs:L1=-3.0,L2=1.0,L3=5.0) circle (0.2);
\draw[fill=blue] (barycentric cs:L1=1.0,L2=5.0,L3=-3.0) circle (0.2);
\draw[fill=green] (barycentric cs:L1=5.0,L2=-2.0,L3=0.0) circle (0.2);
\draw[fill=blue] (barycentric cs:L1=5.0,L2=-3.0,L3=1.0) circle (0.2);
\tikzset{>=latex}
\draw[->,line width=0.4mm,color=black!30!green] (C1I3Im4) -- (Cm1I4Im3);
\draw[->,line width=0.4mm,color=blue] (C1I2Im3) -- (Cm1I3Im2);
\draw[->,line width=0.4mm,color=black!30!green] (C2I1Im3) -- (Cm2I3Im1);
\draw[->,line width=0.4mm,color=blue] (C2I0Im2) -- (Cm2I2I0);
\draw[->,line width=0.4mm,color=black!30!green] (C3Im1Im2) -- (C1Im3I2);
\draw[->,line width=0.4mm,color=blue] (C3Im2Im1) -- (C2Im3I1);
\draw[->,line width=0.4mm,color=black!30!green] (C4Im3Im1) -- (C3Im4I1);
\draw[->,line width=0.4mm,color=blue] (C0Im2I2) -- (Cm2I2I0);
\draw[->,line width=0.4mm,color=black!30!green] (Cm1Im2I3) -- (Cm3I2I1);
\draw[->,line width=0.4mm,color=blue] (Cm2Im1I3) -- (Cm3I1I2);
\draw[->,line width=0.4mm,color=black!30!green] (Cm3Im1I4) -- (Cm4I1I3);
\end{tikzpicture}}
&
\scalebox{1.0}{
\begin{tikzpicture}[scale=0.3]
\coordinate (L1) at (90:5.00);
\coordinate (L2) at (330:5.00);
\coordinate (L3) at (210:5.00);
\coordinate (Cm4I0I4) at (barycentric cs:L1=-3.0,L2=1.0,L3=5.0);
\coordinate (Cm4I1I3) at (barycentric cs:L1=-3.0,L2=2.0,L3=4.0);
\coordinate (Cm3Im1I4) at (barycentric cs:L1=-2.0,L2=0.0,L3=5.0);
\coordinate (Cm3I1I2) at (barycentric cs:L1=-2.0,L2=2.0,L3=3.0);
\coordinate (Cm3I2I1) at (barycentric cs:L1=-2.0,L2=3.0,L3=2.0);
\coordinate (Cm2Im1I3) at (barycentric cs:L1=-1.0,L2=0.0,L3=4.0);
\coordinate (Cm2I2I0) at (barycentric cs:L1=-1.0,L2=3.0,L3=1.0);
\coordinate (Cm2I3Im1) at (barycentric cs:L1=-1.0,L2=4.0,L3=0.0);
\coordinate (Cm1Im2I3) at (barycentric cs:L1=0.0,L2=-1.0,L3=4.0);
\coordinate (Cm1I3Im2) at (barycentric cs:L1=0.0,L2=4.0,L3=-1.0);
\coordinate (Cm1I4Im3) at (barycentric cs:L1=0.0,L2=5.0,L3=-2.0);
\coordinate (C0Im2I2) at (barycentric cs:L1=1.0,L2=-1.0,L3=3.0);
\coordinate (C0I4Im4) at (barycentric cs:L1=1.0,L2=5.0,L3=-3.0);
\coordinate (C1Im3I2) at (barycentric cs:L1=2.0,L2=-2.0,L3=3.0);
\coordinate (C1I2Im3) at (barycentric cs:L1=2.0,L2=3.0,L3=-2.0);
\coordinate (C1I3Im4) at (barycentric cs:L1=2.0,L2=4.0,L3=-3.0);
\coordinate (C2Im3I1) at (barycentric cs:L1=3.0,L2=-2.0,L3=2.0);
\coordinate (C2I0Im2) at (barycentric cs:L1=3.0,L2=1.0,L3=-1.0);
\coordinate (C2I1Im3) at (barycentric cs:L1=3.0,L2=2.0,L3=-2.0);
\coordinate (C3Im4I1) at (barycentric cs:L1=4.0,L2=-3.0,L3=2.0);
\coordinate (C3Im2Im1) at (barycentric cs:L1=4.0,L2=-1.0,L3=0.0);
\coordinate (C3Im1Im2) at (barycentric cs:L1=4.0,L2=0.0,L3=-1.0);
\coordinate (C4Im4I0) at (barycentric cs:L1=5.0,L2=-3.0,L3=1.0);
\coordinate (C4Im3Im1) at (barycentric cs:L1=5.0,L2=-2.0,L3=0.0);
\draw[opacity=0.4] (Cm3I1I2) -- (Cm4I1I3);
\draw[opacity=0.4] (Cm3I1I2) -- (Cm3I2I1);
\draw[opacity=0.4] (C1I2Im3) -- (C2I1Im3);
\draw[opacity=0.4] (C1I2Im3) -- (C1I3Im4);
\draw[opacity=0.4] (C2Im3I1) -- (C3Im4I1);
\draw[opacity=0.4] (C2Im3I1) -- (C1Im3I2);
\draw[opacity=0.4] (Cm2I2I0) -- (Cm3I2I1);
\draw[opacity=0.4] (Cm2I2I0) -- (Cm2I3Im1);
\draw[opacity=0.4] (C0Im2I2) -- (C1Im3I2);
\draw[opacity=0.4] (C0Im2I2) -- (Cm1Im2I3);
\draw[opacity=0.4] (C2I0Im2) -- (C3Im1Im2);
\draw[opacity=0.4] (C2I0Im2) -- (C2I1Im3);
\draw[opacity=0.4] (Cm1I3Im2) -- (Cm2I3Im1);
\draw[opacity=0.4] (Cm1I3Im2) -- (Cm1I4Im3);
\draw[opacity=0.4] (Cm2Im1I3) -- (Cm1Im2I3);
\draw[opacity=0.4] (Cm2Im1I3) -- (Cm3Im1I4);
\draw[opacity=0.4] (C3Im2Im1) -- (C4Im3Im1);
\draw[opacity=0.4] (C3Im2Im1) -- (C3Im1Im2);
\draw[opacity=0.4] (Cm4I0I4) -- (Cm3Im1I4);
\draw[opacity=0.4] (Cm4I0I4) -- (Cm4I1I3);
\draw[opacity=0.4] (C0I4Im4) -- (C1I3Im4);
\draw[opacity=0.4] (C0I4Im4) -- (Cm1I4Im3);
\draw[opacity=0.4] (C4Im4I0) -- (C3Im4I1);
\draw[opacity=0.4] (C4Im4I0) -- (C4Im3Im1);
\draw[fill=blue] (barycentric cs:L1=-2.0,L2=2.0,L3=3.0) circle (0.2);
\draw[fill=green] (barycentric cs:L1=-2.0,L2=3.0,L3=2.0) circle (0.2);
\draw[fill=green] (barycentric cs:L1=2.0,L2=-2.0,L3=3.0) circle (0.2);
\draw[fill=blue] (barycentric cs:L1=2.0,L2=3.0,L3=-2.0) circle (0.2);
\draw[fill=blue] (barycentric cs:L1=3.0,L2=-2.0,L3=2.0) circle (0.2);
\draw[fill=green] (barycentric cs:L1=3.0,L2=2.0,L3=-2.0) circle (0.2);
\draw[fill=green] (barycentric cs:L1=-3.0,L2=2.0,L3=4.0) circle (0.2);
\draw[fill=green] (barycentric cs:L1=2.0,L2=4.0,L3=-3.0) circle (0.2);
\draw[fill=green] (barycentric cs:L1=4.0,L2=-3.0,L3=2.0) circle (0.2);
\draw[fill=blue] (barycentric cs:L1=-1.0,L2=3.0,L3=1.0) circle (0.2);
\draw[fill=blue] (barycentric cs:L1=1.0,L2=-1.0,L3=3.0) circle (0.2);
\draw[fill=blue] (barycentric cs:L1=3.0,L2=1.0,L3=-1.0) circle (0.2);
\draw[fill=green] (barycentric cs:L1=0.0,L2=-1.0,L3=4.0) circle (0.2);
\draw[fill=blue] (barycentric cs:L1=0.0,L2=4.0,L3=-1.0) circle (0.2);
\draw[fill=blue] (barycentric cs:L1=-1.0,L2=0.0,L3=4.0) circle (0.2);
\draw[fill=green] (barycentric cs:L1=-1.0,L2=4.0,L3=0.0) circle (0.2);
\draw[fill=green] (barycentric cs:L1=4.0,L2=0.0,L3=-1.0) circle (0.2);
\draw[fill=blue] (barycentric cs:L1=4.0,L2=-1.0,L3=0.0) circle (0.2);
\draw[fill=green] (barycentric cs:L1=0.0,L2=5.0,L3=-2.0) circle (0.2);
\draw[fill=green] (barycentric cs:L1=-2.0,L2=0.0,L3=5.0) circle (0.2);
\draw[fill=blue] (barycentric cs:L1=-3.0,L2=1.0,L3=5.0) circle (0.2);
\draw[fill=blue] (barycentric cs:L1=1.0,L2=5.0,L3=-3.0) circle (0.2);
\draw[fill=green] (barycentric cs:L1=5.0,L2=-2.0,L3=0.0) circle (0.2);
\draw[fill=blue] (barycentric cs:L1=5.0,L2=-3.0,L3=1.0) circle (0.2);
\tikzset{>=latex}
\draw[->,line width=0.4mm,color=black!30!green] (C1I3Im4) -- (Cm1I4Im3);
\draw[->,line width=0.4mm,color=blue] (C1I2Im3) -- (Cm1I3Im2);
\draw[->,line width=0.4mm,color=black!30!green] (C2I1Im3) -- (Cm2I3Im1);
\draw[->,line width=0.4mm,color=blue] (C2I0Im2) -- (Cm2I2I0);
\draw[->,line width=0.4mm,color=black!30!green] (C3Im1Im2) -- (Cm3I2I1);
\draw[->,line width=0.4mm,color=blue] (C3Im2Im1) -- (C2Im3I1);
\draw[->,line width=0.4mm,color=black!30!green] (C4Im3Im1) -- (C3Im4I1);
\draw[->,line width=0.4mm,color=blue] (C0Im2I2) -- (C2Im3I1);
\draw[->,line width=0.4mm,color=black!30!green] (Cm1Im2I3) -- (Cm3I2I1);
\draw[->,line width=0.4mm,color=blue] (Cm2Im1I3) -- (Cm3I1I2);
\draw[->,line width=0.4mm,color=black!30!green] (Cm3Im1I4) -- (Cm4I1I3);
\end{tikzpicture}}
\\

\end{tabular}

	\caption{\label{fig:fbound_good} The map $\Fbound_0$ (left). The map $\Fbound$ for the $G(v,t)$-forest $F$ in Figure~\ref{fig:counter12} (right).}
\end{figure}

For any $G(v,t)$-grove $F$, the collection of arrows $u\to\Fbound(u)$ for all $u\in\Gbound$ is shown in Figure~\ref{fig:fbound_good} (left). We denote this map by $\Fbound_0$. On the other hand, for the $G(v,t)$-forest in Figure~\ref{fig:counter12}, the map $\Fbound$ is shown in Figure~\ref{fig:fbound_good} (right). It is clear that a $G(v,t)$-forest $F$ is a $G(v,t)$-grove rooted in a canonical way if and only if $\Fbound(u)=\Fbound_0(u)$ for all $u\in\Gbound$.

We say that a vertex $u\in\Gbound$ is a \emph{root vertex} if $u\in\roots_0$ and a \emph{non-root vertex} otherwise.

\begin{proposition}\label{prop:chord}
	Let $F$ be a $G(v,t)$-forest with $\roots(F)=\roots_0$. Then $F$ is a $G(v,t)$-grove if and only if $\Fbound(a_{t+1})=b_{t-1}$.
\end{proposition}
\begin{proof}
	Since $\Fbound_0(a_{t+1})=b_{t-1}$, it is obvious that if $\Fbound(a_{t+1})\neq b_{t-1}$ then $F$ is not a $G(v,t)$-grove. Conversely, suppose for example that the arrow $a_{t+1}\to b_{t-1}$ of $F$ is green (the case when it is blue is completely analogous). Then we claim that for any other non-root green vertex $u\in\GboundG$, we have $\Fbound(u)\neq b_{t-1}$. Indeed, one easily observes that otherwise there would be a blue non-root vertex $w$ such that the blue arrow $w\to\Fbound(w)$ necessarily intersects either $a_{t+1}\to b_{t-1}$ or $u\to b_{t-1}$ because there are no blue vertices from $\roots_0$ in the region of the complement of these two arrows that contains $w$.
	
	We claim that 
	\begin{equation}\label{eq:correct_connectivity}
		\Fbound(a_i)=c_{2t-i}\quad \text{for $1\leq i\leq t$};\quad \Fbound(b_i)=c_{2t-i}\quad\text{for $t\leq i<2t$}.
	\end{equation}
	
	This is true because the number of color changes in the sequence $a_1,a_2,\dots,a_t,b_t,b_{t+1},\dots,b_{2t-1}$ is $2(t-1)$ which is the same as the number of color changes in the sequence $c_{2t-1},c_{2t-2},\dots,c_1$. Therefore since $\Fbound$ preserves the colors and since the arrows do not intersect,~\eqref{eq:correct_connectivity} follows. Here we use the fact that none of the arrows points to $b_{i}$ for $1\leq i\leq t-1$.
	
	By the same argument, it follows that $\Fbound(a_i)=b_{2t-i}$ for $t<i<2t$ and therefore we get $\Fbound(u)=\Fbound_0(u)$ for all $u\in\Gbound$. We are done with the proposition.
\end{proof}

% 
% 
% \begin{figure}
%  \centering
%  \includegraphics[width=0.5\linewidth]{counter_3.png}
%  \caption{\label{fig:counter3} This  counterexample shows that even though each particular path is ok, together they give something bad.}
% \end{figure}

\def\Finfty{F_<}
\subsection{Cylindric groves}
We now explain how to use Theorem~\ref{thm:CS} to give a formula similar to~\eqref{eq:CS} for the cube recurrence in a cylinder. Recall that $\TS_m$ is the strip which is a subset of $\TP_0$ given by $0\leq i\leq m$. Let $v=(i,j,k)\in\TS_m$ be a vertex and consider an integer $t\geq 2$ such that $t+1\equiv\e_v\pmod 3$. Define the graph $G_m(v,t)$ to be the restriction of $G(v,t)$ to $\TS_m$ and denote $\Vert_m(v,t)=\Vert(v,t)\cap \TS_m$ to be its vertex set. 

We now choose a special $G(v,t)$-grove $\Finfty(v,t)$. For every lozenge face $L$ of $G(v,t)$, we say that it is \emph{above} (resp., \emph{below}) $v$ if the first coordinate of each of its vertices is greater than or equal (resp., less than or equal) to $i$. To define $\Finfty(v,t)$, we need to specify for each lozenge $L$ whether we choose a green diagonal or a red-blue diagonal. And the rule is, we choose a green diagonal if and only if one of the following conditions is satisfied:
\begin{itemize}
 \item $L$ is above $v$ and its green diagonal is parallel to $(1,1,-2)$, or
 \item $L$ is below $v$ and its green diagonal is parallel to $(-1,2,-1)$.
\end{itemize}
In particular, if $L$ is neither below nor above $v$ then we choose a red-blue diagonal in it. An example of $\Finfty(v,t)$ together with $\bij(\Finfty(v,t))$ is given in Figure~\ref{fig:finfty}.

\begin{figure}
 \centering



 \caption{\label{fig:finfty} The grove $\Finfty(v,t)$.}
\end{figure}

Finally, let $F_m(v,t)$ be the restriction of $\Finfty(v,t)$ to $\TS_m$. 

\begin{definition}
A \emph{$G_m(v,t)$-grove} is a forest $F$ with vertex set $\Vert_m(v,t)$ satisfying the following conditions:
\begin{enumerate}
 \item\label{item:all_edges} $F$ contains all edges $(v,v+e_{23})$ where $v$ is a red boundary vertex of $\TS_m$;
 \item for every lozenge face of $G_{m}(v,t)$, $F$ contains exactly one of its two diagonals, and there are no other edges in $F$;
 \item two boundary vertices of $G_m(v,t)$ belong to the same connected component of $F$ if and only if they belong to the same connected component of $F_m(v,t)$. 
\end{enumerate}
\end{definition}

\def\Ft{{\tilde F}}

In other words, every $G_m(v,t)$-grove $F$ can be viewed as a $G(v,t)$-grove $\Ft$ that coincides with $\Finfty(v,t)$ outside of the strip $\TS_m$ and coincides with $F$ inside $\TS_m$.

\def\q{q}
\def\k{\kappa}
\begin{proposition}\label{prop:CS}
 For a non-boundary vertex $v\in\TS_m$ and $t\geq 2$ such that $t+1\equiv \e_v\pmod 3$, the cube recurrence in the cylinder satisfies
 \[\Cubec_v(t+1)=\sum_F \wt(F),\]
 where the sum is taken over all $G_m(v,t)$-groves $F$ and their weight is defined by $\wt(F):=\wt(\Ft)$.
\end{proposition}
\begin{proof}
 We use the same trick as in the proof of~\cite[Theorem~3.1.4]{GP2}: for every vertex $u\in\TR(v,t)$ that does not belong to the interior of $\TS_m$, we are going to substitute $x_u:=\q^{\k(u)}$ for some integer $\k(u)$ that depends on the first coordinate of $u$ and is a convex function. More precisely, if $u=(a,b,c)$ then we set 
 \[\k(u)=\k(a)=\begin{cases}
         0,&\text{if $0\leq a\leq m$}\\
         a(m-a)2^{2^{a(m-a)}},&\text{otherwise}.
        \end{cases}\]
Note that when $u$ is a non-boundary vertex of $\TS_m$ then $\k(u)=0$ but we do not substitute $x_u=1$.
 
 After such a substitution, the values $\Cube_w(t)$ of the unbounded cube recurrence become rational functions in $\q$ and $\x^\TC$. We would like to show that for $w\in\TS_m$ and $t\equiv\e_w\pmod 3$, the value of the unbounded cube recurrence $\Cube_w(t)$ tends to the value $\Cubec_w(t)$ of the cube recurrence in a cylinder as $\q$ tends to $0$. We prove this by induction on $t$, where the induction hypothesis is that for any vertex $w$ of $\TP_0$ and any $t'<t$ such that $t'\equiv\e_w\pmod 3$, the value of the unbounded cube recurrence satisfies
 \begin{equation}\label{eq:asymp}
\Cube_w(t')\asymp \q^{\k(w)}
 \end{equation}
 as $\q\to0$. Here by $g(\q)\asymp h(\q)$ we mean that $\frac{g(\q)}{h(\q)}$ tends to a non-zero rational function in $\x^\TC$. The base $t=0,1,2$ of induction is clear and the induction step is an easy direct computation. Indeed, take any vertex $u=(a,b,c)$. 
%  If $u$ is a non-boundary vertex of $\TS_m$ then we have $\Cube_w(t+1)\asymp 1$ or $\Cube_w(t+2)\asymp 1$ for all the neighbors $w$ of $u$, and thus we will have $\Cube_u(t+3)\asymp 1$ as well. Otherwise, l
 We have 
 \[\Cube_u(t+3)\Cube_u(t)\asymp \q^{\k(a+1)+\k(a-1)}+\q^{2\k(a)}.\]
 Using the convexity of $\k$, the induction step follows since $\q^{\k(a+1)+\k(a-1)}+\q^{2\k(a)}\asymp \q^{2\k(a)}$ and $\Cube_u(t)\asymp \q^{\k(a)}$. 
 Due to \eqref{eq:asymp}, when $w$ belongs to the boundary of $\TS_m$, the value of $\Cube_w(t)$ tends to $1$ as $\q\to 0$. Therefore for any $w\in\TS_m$, the value $\Cube_w(t)$ of the unbounded cube recurrence tends to the value $\Cubec_w(t)$ of the cube recurrence in a cylinder as $\q\to 0$.
 
 Theorem~\ref{thm:CS} gives a formula for the values of the unbounded cube recurrence in terms of $G(v,t)$-groves. Each grove will have a weight that is a rational function in $\x^\TC$ and $\q$. To finish the proof of the theorem, it suffices to show that $\wt(F)$ tends to zero as $\q\to 0$ unless $F$ coincides with $\Finfty(v,t)$ outside of the strip $\TS_m$. A more precise formulation of this condition is that for any two vertices $u$ and $w$ that are not in the interior of $\TS_m$, there is an edge connecting them in $F$ if and only if there is an edge connecting them in $\Finfty(v,t)$. We show this claim by induction on $a$ where $u=(a,b,c)$. The base case is when $u=v+(t-1)e_{12}$ so that $a=i+t-1$ and the claim for this case follows trivially from the definition of a grove. For the induction step, note that if for all the vertices with the first coordinate greater than $a$, the groves $F$ and $\Finfty(v,t)$ coincide then they have to coincide for all the vertices with the first coordinate $a$ because otherwise we will have $\q^{\k(a)}$ in the numerator and since $\k(a)$ grows rapidly in $a$, the vertices with smaller values of $\k(a)$ together will not be able to compensate for it. This finishes the proof of the proposition. %TODO REWRITE THIS MORE RIGOROUSLY?
\end{proof}

\def\Ntnm{{\Nt_{n,m}}}

One easily observes that if a path in $\bij(\Finfty(v,t))$ starts outside of $\TS_m$ then it stays outside of $\TS_m$, see Figure~\ref{fig:finfty}. Thus for a $G_m(v,t)$-grove $F$, we can define $\bij(F):=\bij(\Ft)\cap \TS_m$ in the sense that we remove the paths of $\bij(\Ft)$ that stay outside of $\TS_m$. We are finally ready to describe the planar cylindrical network $\Ntnm$ to which we will apply the results of Section~\ref{sect:networks}.

\begin{definition}
The vertex set of $\Ntnm$ is $\TS_m$ together with the centers of all lozenges $L$ of $G$ such that all four vertices of $L$ belong to $\TS_m$. The part of $\Ntnm$ inside of each such lozenge $L$ is given in Figure~\ref{fig:nvt}. Additionally, for every pair $(u,u+e_{23})$ of vertices on the boundary of $\TS_m$, the network $\Ntnm$ contains an edge $u+e_{23}\to u$ of weight $1$. %A part of $\Ntnm$ for $n=3$ and $m=4$ is given in Figure~TODO~NTM.
\end{definition}

% such that the image of $\bij(\Ft)$ would be an $r$-path in $\Nvtm$. Note that the graph $G_m$ has lozenge faces and also some boundary triangle faces. For each lozenge face $L$, the part of $\Nvtm$ inside $L$ is defined in the same was as in Figure~\ref{fig:nvt}. Each boundary triangle $T$ of $G_m$ has a red, a blue, and a green vertex which we denote $r,g,$ and $b$ respectively. The part of  $\Nvtm$ inside $T$ by definition contains just a single edge $b\to r$ with weight $1$. 

This definition together with Theorem~\ref{thm:bij_groves} yields the following. 
\begin{corollary} \label{cor:paths_cyl}
 Let $v=(i,j,k)\in\TS_m$ be a vertex and define $r:=m-i$. Then the map $\bij$ restricts to an injection from the set of $G_m(v,t)$-groves to the set $\pathst(\ubft_r,\wbft_r)$ of $r$-paths $\Pbft$ in $\Ntnm$ that start at $\ubft_r=(a_2,a_4,\dots,a_{2r})$ and end at $\wbft_r=(b_{t-1},b_{t-2},\dots,b_{t-r})$. Moreover, we have
 \[\wt(F)=\wt(\bij(F)).\]
\end{corollary}
\begin{proof}
	We only need to prove the part about weights because the rest of the statements are clear. The fact that $\wt(F)=\wt(\bij(F))=1$ can be easily checked when $F$ coincides with $\Finfty$ inside of $\TS_m$, and then the equality of weights follows from the observation that when we flip one edge of $F$ inside some lozenge $L$, the weight of $F$ is multiplied by the same amount as the weight of $\bij(F)$. Similarly to Theorem~\ref{thm:bij}, this again requires extending $\bij(F)$ from $G_m(v,t)$-groves to arbitrary collections of edges so that in every lozenge of $G_m(v,t)$ we choose precisely one edge.
\end{proof}

An example of a $G_m(v,t)$-grove $F$ together with $\bij(F)$ is given in Figure~\ref{fig:gmvt_grove_bij}.

\begin{figure}
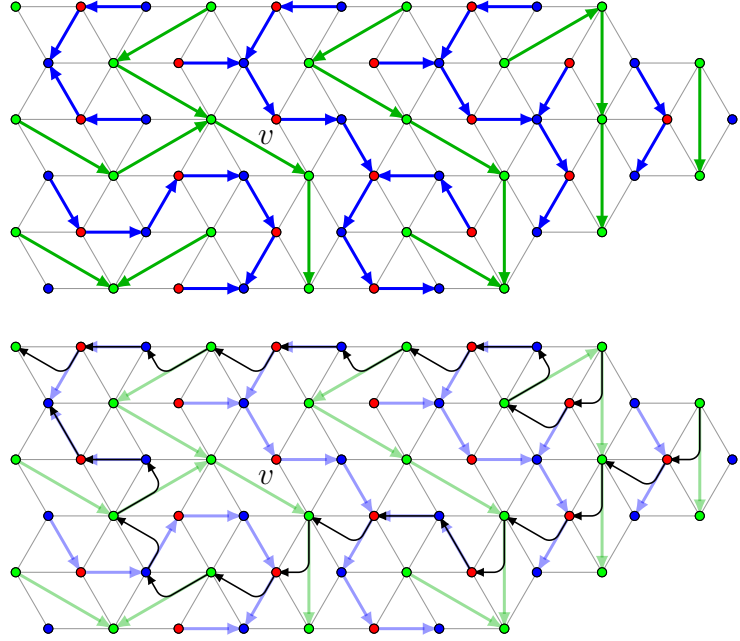




	\caption{\label{fig:gmvt_grove_bij} A $G_m(v,t)$-grove $F$ (top) and the corresponding $2$-path $\bij(F)$ (bottom).}
\end{figure}

\def\gg{g}
\def\Nnm{{N_{n,m}}}
Recall that the vertex variables $x_v$ for $v\in\TS_m$ satisfy $x_v=1$ if $v$ belongs to the boundary of $\TS_m$ and otherwise we have $x_v=x_{v+3\gg}$ where $\gg=ne_{23}=(0,n,-n)$ for some integer $n\geq 1$. Thus even though the edges and the vertices of $\Ntnm$ are periodic with respect to the shift by $3e_{23}$, the edge weights in $\Ntnm$ are periodic with respect to the shift by $\g:=3\gg$. Just as in Section~\ref{sect:networks}, we define the projection $\Nnm$ of $\Ntnm$ to the cylinder $\TC$. 

\subsection{Integrability}\label{sect:integrability}
In this section, we prove Theorems~\ref{thm:recurrence_1},~\ref{thm:cube_formula} as well as Theorem~\ref{thm:pleth} which implies both of them. The case $v=(m-1,j,k)$ described in Theorem~\ref{thm:cube_formula} is particularly simple as in this case the image of the map $\bij$ is the whole set $\pathst(\ubft_1,\wbft_1)$:

\begin{proposition} \label{prop:paths_1}
	For any $1$-path $\Pbft\in\pathst(\ubft_1,\wbft_1)$ in $\Ntnm$, the preimage $F=\bijinv(\Pbft)$ is a $G_m(v,t)$-grove.
\end{proposition}
\begin{proof}
\def\Fboundt{{\partial \tilde F}}
	We need to show that $\Fboundt(u)=\Fbound_0$ for all $u\in\Gbound$, where $\Ft$ is the extension of $F$ to a $G(v,t)$-forest. By Proposition~\ref{prop:chord}, it suffices to show that $\Fboundt(a_{t+1})=b_{t-1}$. But this is clear since the part of the $G(v,t)$-forest $\Finfty$ that lies outside of $\TS_m$ already contains a path from $a_{t+1}$ to $b_{t-1}$ and since $\Ft$ coincides with $\Finfty$ outside of $\TS_m$, we get that indeed $\Fbound(a_{t+1})=b_{t-1}$. This finishes the proof of the proposition.
\end{proof}

\begin{corollary}\label{cor:paths_1}
	Let $v=(m-1,j,k)\in\TS_m$ and let $t\geq 2$ be such that $t+1\equiv \e_v\pmod 3$. Then 
	\[\Cubec_v(t)=\sum_{\Pt} \wt(\Pt),\]
	where the sum is taken over all paths $\Pt$ in $\Ntnm$ that start at $a_2=v+(1,t-2,1-t)$ and end at $b_{t-1}$ given by
	\begin{equation}\label{eq:b_t-1}
	b_{t-1}= v+\begin{cases}
	           	(1,\frac{-1-t}2,\frac{t-1}2),&\text{if $t$ is odd};\\
	           	(1,\frac{-t}2,\frac{t-2}2),&\text{if $t$ is even}.
	           \end{cases}
	\end{equation}
\end{corollary}
\begin{proof}
	Follows from Proposition~\ref{prop:paths_1} and Corollary~\ref{cor:paths_cyl}.
\end{proof}

Since $\Ntnm$ is a planar cylindrical network, we would like to describe the polynomial $Q_\Nnm(t)$ given by~\eqref{eq:solenoids_intro}. Recall that the polynomial $Q(t)$ is given by~\eqref{eq:Q_t}.

\begin{proposition} \label{prop:Q_t}
	We have $Q(t)=Q_\Nnm(t)$. In other words, for every $r=0,1,\dots,m$, we have
	\[J_r=\sum_{\Cbf\in\Cyc^r(\Nnm)} \wt(\Cbf).\]	
\end{proposition}
\begin{proof}
	Recall that by~\eqref{eq:J_r}, $J_r$ is defined as the sum of $\wt(F)$ over all $(3n,m)$-groves $F$ from Definition~\ref{dfn:3n_m_grove} satisfying $h(F)=r$. For each connected component $C$ of $F$, choose the root $\root_C$ to be the unique green or blue vertex of $C$ that lies on the bottom boundary of $\TS_m$. Applying $\bij$ to this rooted version of $F$ (that is, applying it to every lozenge of $G$ inside $\TS_m$) yields a collection of edges in $\Ntnm$ that is periodic with respect to the shift by $\g$ and it is easy to see that every vertex of $\Ntnm$ either is isolated or has indegree and outdegree $1$ in $\bij(F)$. Therefore $\bij(F)$ projects to a vertex-disjoint collection $\Cbf$ of cycles in $\Nnm$.
	
	Similarly, given an $r$-cycle $\Cbf\in\Cyc^r(\Nnm)$, we can apply $\bijinv$ to it inside of every lozenge of $G$ and a similar argument shows that we will get a $(3n,m)$-grove $F$ rooted as in the previous paragraph. Thus there is a bijection $\bij$ between the set of $(3n,m)$-groves and the set of $r$-cycles $\Cbf$ for $0\leq r\leq m$. It is straightforward to check that $\wt(F)=\wt(\bij(F))$ but there is one more thing we need to verify, namely, that $h(F)=r$ where $r$ is the number of cycles in $\bij(F)$.

	\begin{figure}
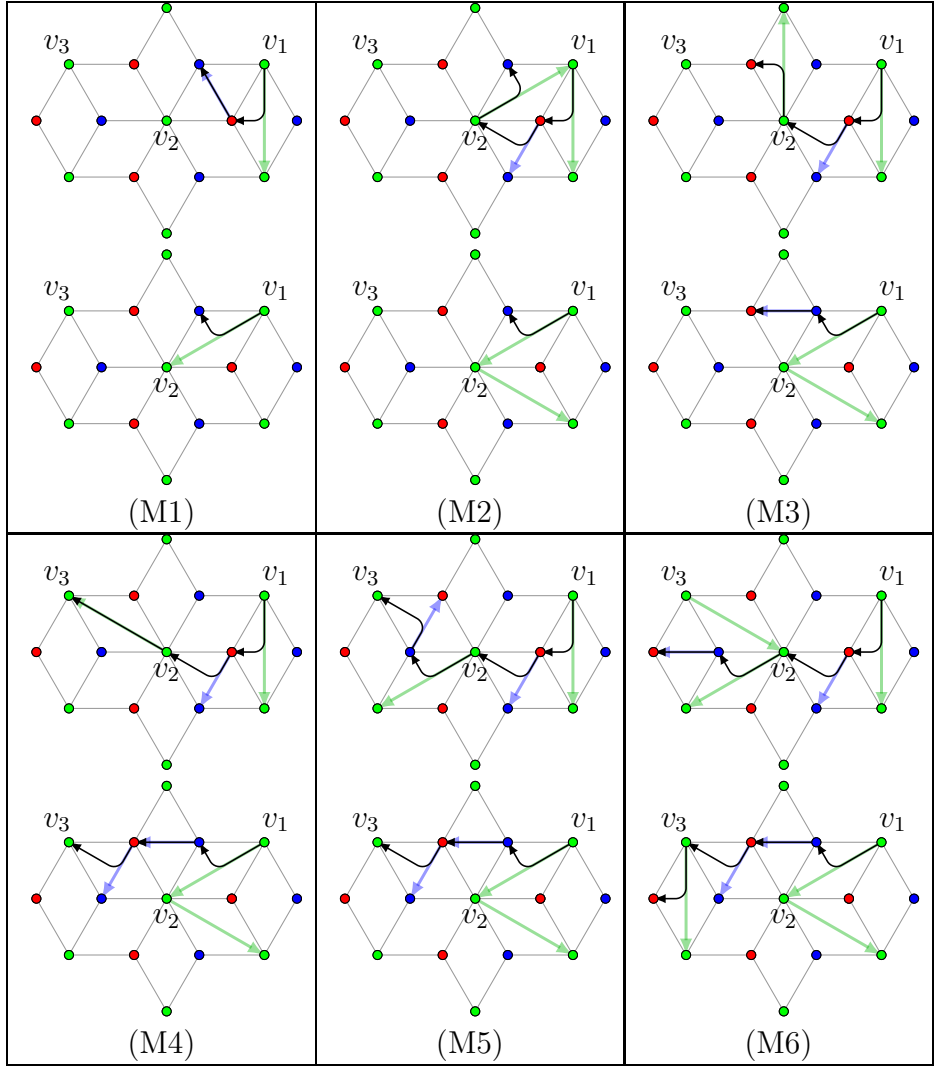




		\caption{\label{fig:six} The six local moves that one can apply to a non-horizontal cycle.}
	\end{figure}

	Let us show that for any $r$-cycle $\Cbf\in\Cyc^r(\Nnm)$, its preimage $F=\bijinv(\Cbf)$ satisfies $h(F)=r$. First, it is easy to check that if every cycle $C\in \Cbf$ is \emph{horizontal} (meaning that all vertices of $C$ that are vertices of $G$ have the same first coordinate) then $h(F)=r$. Now suppose that not every cycle in $\Cbf$ is horizontal and choose $C\in\Cbf$ so that $C$ is not horizontal but all cycles of $\Cbf$ that are above $C$ are horizontal. Let us look at the set $V(C)$ of vertices of $G$ belonging to $C$. Let $u\in V(C)$ be a vertex such that its first coordinate is minimal, and among all such vertices $u$ choose the one such that the first coordinate of the vertex $v_1\in V(C)$ that precedes $u$ in $C$ would not be minimal. In other words, $u$ is the ``earliest'' among the ``lowest'' vertices of $C$. It follows that $u$ is red and $v_1$ is green, see Figure~\ref{fig:six}. Let $v_1=(i_1,j_1,k_1)$ and define the green vertices $v_2=(i_2,j_2,k_2)$ and $v_3=(i_3,j_3,k_3)$ by
	\[i_1=i_2+1,\quad j_1=j_2+1,\quad k_1=k_2-2;\]
	\[i_3=i_2+1,\quad j_3=j_2-2,\quad k_3=k_2+1.\]
	See Figure~\ref{fig:six} for an illustration of how $v_1,v_2,$ and $v_3$ are positioned relative to each other. 
	
	There are six possibilities of how $C$ can evolve from $v_1$ shown in Figure~\ref{fig:six}. For each of them, we give a local move that modifies $C$ and prove that this local move does not affect $h(F)$. Each local move decreases the number of vertices of $C$ with minimal first coordinate and thus applying these local moves to $C$ eventually transforms $C$ into a horizontal cycle. Thus using these local moves one can make every cycle of $\Cbf$ horizontal without changing $h(F)$ which immediately yields the result. The only thing left to show is why these local moves do not affect $h(F)$. 
	
	Recall that $h(F)$ is defined  as follows. Orient $F$ in a canonical way so that the root of every connected component would be on the lower boundary of $\TC$. Take any green vertex $v'=(m,j',k')$ and look at the root $v=(0,j,k)$ of its connected component in $F$. Then $h(F)=(j'-j+2m)/3$ and it does not depend on the choice of $v'$. For any green vertex $v$ of $F$, we define $[v]$ to be the root of the connected component of $F$ containing $v$. This vertex $[v]$ can be obtained from $v$ by following the oriented edges of $F$ (recall that every vertex of $F$ with nonzero first coordinate has exactly one outgoing edge). Thus it suffices to show that we do not change $[v]$ for at least one vertex $v$ on the upper boundary of $\TC$ when we apply our local move.
	
	As it is apparent from Figure~\ref{fig:six}, the only green vertices of $F$ that change their outgoing edge belong to the set $\{v_1,v_2,v_3\}$. Moreover, it is easy to see that each local move does not actually change $[v_1]$ and $[v_3]$.\footnote{Except for the move (M1) which may change $[v_1]$ if $v_2$ belongs to some other cycle in $\Cbf$. However, the move (M1) does not change $[v_2]$ and $[v_3]$ so an analogous argument applies in this case.} Thus the only case we need to consider is when for every green vertex $v$ on the upper boundary of $\TC$, the path from $v$ to $[v]$ in $F$ passes through $v_2$. In particular, this implies that there is exactly one green vertex $v$ on the upper boundary of $\TC$, so we get $n=1$ and $v_3=v_1$. The fact that in this case each of $[v_1],[v_2],$ and $[v_3]$ is preserved under each local move is verified in a straightforward way.
\end{proof}

We are now ready to prove Theorem~\ref{thm:cube_formula}.
\begin{proof}[Proof of Theorem~\ref{thm:cube_formula}]
	Let $v=(m-1,j,k)$ as in the theorem. We need to show~\eqref{eq:cube_formula}, in other words, we need to show that the sequence $s:\N\to \Q(\x)$ defined by $s(\ell)=\Cubec_{v+\ell g}(\e_v+2\ell n)$ satisfies a linear recurrence with characteristic polynomial $Q(t)$. Let us carefully apply Corollary~\ref{cor:paths_1} to this case. Let $v'=v+\ell g$ and $t'=\e_v+2\ell n-1$, then $\Cubec_{v'}(t'+1)$ counts directed paths in $\Ntnm$ from $a_2=v'+(1,t'-2,1-t')$ to $b_{t'-1}$ given by~\eqref{eq:b_t-1}. Substituting the correct values for $v'$ and $t'$ we get
	\[a_2=v+(1,t-3,2-t)+\ell\g;\quad b_{t'-1}=v+\begin{cases}
	                                                 	(1,-\frac {\e_v}2, \frac{\e_v-2}2),&\text{ if $\e_v$ is even};\\
	                                                 	(1,\frac{1-\e_v}2, \frac{\e_v-3}2),&\text{ if $\e_v$ is odd}.
	                                                 \end{cases}\]
	Thus $b_{t'-1}$ is fixed while $a_2$ increases by $\g$ every time we increase $\ell$. Thus Theorem~\ref{thm:planar} applies, and by Proposition~\ref{prop:Q_t}, the characteristic polynomial $Q(t)$ of the recurrence is precisely the one given in~\eqref{eq:cube_formula}. We are done with the proof of Theorem~\ref{thm:cube_formula}.
\end{proof}

Consider now an arbitrary vertex $v=(i,j,k)\in \TS_m$ and define $r=m-i$. If $\bij$ from Corollary~\ref{cor:paths_cyl} was a bijection between the set of $G_m(v,t)$-groves and the set $\pathst(\ubft_r,\wbft_r)$ then Theorem~\ref{thm:pleth} would follow immediately from Theorem~\ref{thm:planar}. Unfortunately, $\bij$ is just an injection: as we have mentioned in Remark~\ref{rmk:counter}, there can be $r$-paths $\Pbft$ in $\pathst(\ubft_r,\wbft_r)$ such that $\bijinv(\Pbft)$ has the wrong connectivity. Additionally, there are some $r$-paths in $\pathst(\ubft_r,\wbft_r)$ that do not stay entirely inside the graph $G_m(v,t)$. Even worse, the vertices $\ubft_r$ and $\wbft_r$ are \emph{permutable} in $\Ntnm$, for example, there exists an $r$-path starting at $\ubft_r$ and ending at $\wbft'_r:=(b_{t-2},b_{t-1},b_{t-3},b_{t-4},\dots,b_{t-r})$ (such a path would necessarily have to exit the graph $G_m(v,t)$). Thus we need to do more work in order to resolve all these issues and prove the linear recurrence relation.

\begin{definition}
	For a vertex $\vt=(i,j,k)$ of $\Ntnm$, define 
	\[h(\vt):=i+k=-j.\]
\end{definition}

\begin{lemma}\label{lemma:h}
	Suppose that $\Pt$ is a path in $\Ntnm$ and let $(\vt_1,\vt_2,\dots,\vt_p)$ be the vertices of $\Pt$ that are also vertices of $G$ rather than centers of lozenges of $G$. We have 
	\[h(\vt_1)\leq h(\vt_2)\leq\dots\leq h(\vt_p),\quad\text{and}\quad h(\vt_\ell)<h(\vt_{\ell+2})\ \forall\, 1\leq\ell\leq p-2.\]
\end{lemma}
\begin{proof}
	This follows immediately from the definition of $\Ntnm$, see Figure~\ref{fig:nvt}. In particular, if $\vt_\ell$ is either red or blue then $h(\vt_\ell)<h(\vt_{\ell+1})$ which implies the second claim.
\end{proof}

% \def\const{ \operatorname{const}}
% Note that $h(\vt)$ is a linear function and thus we can consider lines $h(\vt)=y$ for $y\in\Z$ which are certain sets of vertices of $\Ntnm$. 
Note that since $\g=(0,3n,-3n)$, we get that $h(\vt+\g)=h(\vt)-3n$.

\begin{definition}
	We say that an $r$-vertex $\vbft=(\vt_1,\vt_2,\dots,\vt_r)$ in $\Ntnm$ is \emph{$h$-constant} if we have \[h(\vt_1)=h(\vt_2)=\dots=h(\vt_r).\]
	In this case, we define $h(\vbft):=h(\vt_i)$ where $1\leq i\leq r$ is arbitrary.
\end{definition}

An immediate consequence of Lemma~\ref{lemma:h} is the following.
\begin{corollary}\label{cor:non_permutable}
	Let $\ubft$ and $\vbft$ be two $h$-constant $r$-vertices in $\Ntnm$. Then they are non-permutable.
\end{corollary}

 Let now $v=(i,j,k)\in\TS_m$ be a vertex and define $r:=m-i$. As we know from Corollary~\ref{cor:paths_cyl}, $\bij$ maps each $G_m(v,t)$-grove $F$ to some $r$-path $\bij(F)\in\pathst(\ubft_r,\wbft_r)$ that starts at $\ubft_r=(a_2,a_4,\dots,a_{2r})$ and ends at $\wbft_r=(b_{t-1},b_{t-2},\dots,b_{t-r})$. We give a necessary and sufficient condition for an $r$-path $\Pbft\in\pathst(\ubft_r,\wbft_r)$ to belong to the image of $\bij$.

\begin{lemma}\label{lemma:M}
There exist two integers $M_1$ and $M_2$ depending only on $n$, $m$, and $r$ such that $\Pbft$ belongs to the image of $\bij$ if and only if there is another $r$-path $\Pbft'\in\pathst(\ubft_r,\wbft_r)$ that belongs to the image of $\bij$ and coincides with $\Pbft$ for all vertices $\ut$ of $\Ntnm$ satisfying either $h(b_t)+M_2\leq h(\ut)$ or $h(\ut)\leq h(a_2)+M_1$.
\end{lemma}
\begin{proof}
	We only need to prove one direction since if $\Pbft$ itself belongs to the image of $\bij$ then we can set $\Pbft'=\Pbft$. 
	
	We need to specify $M_1$ and $M_2$. Set $M_1=h(a_{2r})-h(a_2)$. It is a bit harder to describe $M_2$. Given a $G_m(v,t)$-grove $F$, we set $M_2(F)$ to be the minimum of $h(b_t)-h(u)$ where $u\in \TS_m$ is connected to $b_t$ in $F$. We set $M_2$ to be the minimum of $M_2(F)$ where $F$ runs over the set of all $G_m(v,t)$-groves. It is easy to observe that $M_2$ does not depend on $t$ because there is a specific \emph{right-justified} $G_m(v,t)$-grove $F'$ such that every connected component of $F'$ is weakly to the right of every connected component of $F$ and for this $F'$, $M_2(F')$ is the same for all sufficiently large $t$. 
	
	Suppose now that $\Pbft,\Pbft'\in\pathst(\ubft_r,\wbft_r)$ satisfy the requirements of the lemma with the above choice of $M_1$ and $M_2$. We want to show that $\Pbft$ belongs to the image of $\bij$. Because of the way we chose $M_1$, $\Pbft$ stays inside of $G_m(v,t)$ and thus the preimage $\bijinv(\Pbft)$ is a certain $G_m(v,t)$-forest. On the other hand, our choice of $M_2$ ensures that the connected components of $b_t$ in $\bijinv(\Pbft)$ and in $\bijinv(\Pbft')$ are the same. The result follows by Proposition~\ref{prop:chord}.
\end{proof}

% 
% The following refinement of Corollary~\ref{cor:paths_cyl} will be crucial for the rest of our proof:
% 
% \begin{lemma}
% 	 Let $v=(i,j,k)\in\TS_m$ be a vertex, define $r:=m-i$. Then there exist two positive integers $M_1$ and $M_2$ such that:
% 	 \begin{enumerate}
% 	 	\item 
% 	 \end{enumerate}
% 
% 	 
% 	 the map $\bij$ restricts to an injection from the set of $G_m(v,t)$-groves to the set $\pathst(\ubft_m,\wbft_m)$ of $r$-paths $\Pbft$ in $\Ntnm$ that start at $\ubft_m=(a_2,a_4,\dots,a_{2r})$ and end at $\wbft_m=(b_{t-1},b_{t-2},\dots,b_{t-r})$. Moreover, we have
%  \[\wt(F)=\wt(\bij(F)).\]
% \end{lemma}

%Our proof consists of two parts. In Section~\ref{sect:networks}, we generalize our approach from~\cite{GP2} to give a general result on linearizability for planar cylindrical networks. In Section~\ref{sect:temperley} we then give a weight-preserving bijection between groves and tuples of non-crossing lattice paths in a certain network. As we will see, combining these two results together will yield Theorems~\ref{thm:recurrence_1} and~\ref{thm:cube_formula} almost immediately. 

We are ready to finish the proof of Theorem~\ref{thm:pleth}.

\begin{proof}[Proof of Theorem~\ref{thm:pleth}]
% Let us first outline the key steps in the proof. 
We need to show that the sequence of values $(\Cubec_{v+\ell g}(\e_v+2\ell n))_{\ell\in\N}$ of the cube recurrence in the cylinder satisfies a linear recurrence with characteristic polynomial $Q^\plee{r}(t)$ where $Q(t)$ is given by~\eqref{eq:Q_t}. By Proposition~\ref{prop:CS}, $\Cubec_{v+\ell g}(\e_v+2\ell n)$ is a sum of weights of all $G_m(v,t)$-groves. The map $\bij$ assigns to each such $G_m(v,t)$-grove $F$ an $r$-path $\Pbft_\ell=\bij(F)\in\pathst(\ubft_\ell,\wbft)$ between two $r$-vertices $\ubft_\ell$ and $\wbft$ of $\Ntnm$, but not all $r$-paths in $\pathst(\ubft_\ell,\wbft)$ belong to the image of $\bij$. We decompose each path $\Pbft\in\pathst(\ubft_\ell,\wbft)$ as a concatenation of three paths $\Pbft^\parr1,\Pbft^\parr2,$ and $\Pbft^\parr3$. The decomposition is similar to the one we used in the proof of~\cite[Theorem~4.9]{Networks}. Namely, the path $\Pbft^\parr i$ is the restriction of $\Pbft$ to the set of all vertices $\ut$ of $\Ntnm$ satisfying:
\[
\begin{cases}
	h(\ut)\leq M_1-3n\ell,&\text{if $i=1$};\\
	M_1-3n\ell\leq h(\ut)\leq M_2,&\text{if $i=2$};\\
	M_2\leq h(\ut),&\text{if $i=3$.}
\end{cases}\]
Here $M_1, M_2\in\Z$ are the constants given by Lemma~\ref{lemma:M}. By Lemma~\ref{lemma:M}, the fact that the $r$-path $\Pbft$ belongs to the image of $\bij$ depends only on $\Pbft^\parr1$ and $\Pbft^\parr3$ but not on $\Pbft^\parr2$. Since there are finitely many choices for $\Pbft^\parr1$ and $\Pbft^\parr3$, we get that $\Cubec_{v+\ell g}(\e_v+2\ell n)$ decomposes as a finite linear combination of sequences that satisfy a linear recurrence with characteristic polynomial $Q^\plee{r}(t)$. This implies that $\Cubec_{v+\ell g}(\e_v+2\ell n)$ itself satisfies a linear recurrence with the same characteristic polynomial, which finishes the proof of the theorem. 
\end{proof}

\section*{Acknowledgment}
The author is indebted to Pavlo Pylyavskyy for his numerous contributions to this project.

\bibliographystyle{plain}
\bibliography{cube}

\end{document}